\documentclass[psamsfonts]{amsart}

\usepackage[utf8]{inputenc}
\usepackage[margin=1in, bottom = 1.25in]{geometry}
\usepackage{setspace}
\usepackage[font=small, labelfont=bf, margin=1cm]{caption}
\usepackage{multicol}
\usepackage{hyperref}
\usepackage{mwe}
\usepackage{subcaption}

\usepackage{amsmath}
\usepackage{amsthm}
\usepackage{amssymb,amsfonts}
\usepackage{dsfont}
\usepackage{bbm}
\usepackage{tikz-cd}
\usepackage[all,arc]{xy}

\theoremstyle{plain}
\newtheorem{thm}{Theorem}[section]
\newtheorem{cor}[thm]{Corollary}
\newtheorem{prop}[thm]{Proposition}
\newtheorem{lem}[thm]{Lemma}

\theoremstyle{definition}
\newtheorem{defn}[thm]{Definition}

\newtheorem{exmp}[thm]{Example}
\newtheorem{exmps}[thm]{Examples}
\newtheorem{notn}[thm]{Notation}

\newtheorem{rem}[thm]{Remark}

\newcommand{\Tprod}{\operatorname{Tot^\Pi}}
\newcommand{\Tot}{\operatorname{Tot}}
\newcommand{\Hom}{\operatorname{Hom}}
\newcommand{\RHom}{\operatorname{\mathbb{R}Hom}}

\newcommand{\Der}{\operatorname{Der}}
\newcommand{\Inner}{\operatorname{Inner}}

\newcommand{\degree}{\operatorname{deg}}
\newcommand{\M}{MF\big( \J, \ell \big)}
\newcommand{\Q}{\mathcal{Q}}
\newcommand{\C}{\mathbb{C}}
\newcommand{\R}{\mathbb{R}}
\newcommand{\D}{\mathbb{D}}
\newcommand{\J}{J(\mathcal{Q})}

\newcommand{\JW}{J(\mathcal{Q})[\ell^{-1}]}
\newcommand{\Perf}{\mathcal{P}}
\newcommand{\bk}{\mathds{k}}
\newcommand{\Lie}{\mathcal{L}}
\newcommand{\Z}{\mathbb{Z} / 2 \mathbb{Z}}
\newcommand{\z}{\mathbb{Z}}

\newcommand{\tor}{\operatorname{Tor}}
\newcommand{\ext}{\operatorname{Ext}}

\newcommand{\B}{\operatorname{Bar}}

\newcommand{\zc}{\mathcal{Z}}
\newcommand{\zcl}{\mathcal{Z}[\ell^{-1}]}
\newcommand{\hA}{\widehat{A}}
\newcommand{\hB}{\widehat{\B}}
\newcommand{\hBc}{\widehat{B}}
\newcommand{\hL}{\widehat{L}}
\newcommand{\hD}{\widehat{D}}

\newcommand{\Div}{\operatorname{div}}

\newcommand{\ur}{\widetilde{\Sigma}}

\newcommand{\N}{N_{\C}}
\newcommand{\Ni}{N^{in}_{\C}}
\newcommand{\No}{N^{out}_{\C}}
\newcommand{\st}{\mathcal{O}}
\newcommand{\Int}{\operatorname{Int}}
\newcommand{\Su}{S_{\uparrow}}
\newcommand{\Sd}{S_{\downarrow}}
\newcommand{\mJ}{\mathcal{J}}
\newcommand{\E}{\mathcal{E}}

\begin{document}

\title[Dimer Models and Hochschild Cohomology]{Dimer Models and Hochschild Cohomology}

\author{Michael Wong}

\address{Department of Mathematics, University College London, Gower Street, London WC1E 6BT}

\email{michael.wong@ucl.ac.uk}



\begin{abstract}
Dimer models provide a method of constructing noncommutative crepant resolutions of affine toric Gorenstein threefolds. In homological mirror symmetry, they can also be used to describe noncommutative Landau--Ginzburg models dual to punctured Riemann surfaces. For a zigzag consistent dimer embedded in a torus, we explicitly compute the Hochschild cohomology of its Jacobi algebra in terms of dimer combinatorics. This includes a full characterization of the Batalin--Vilkovisky structure induced by the Calabi--Yau structure of the Jacobi algebra. We then compute the compactly supported Hochschild cohomology of the category of matrix factorizations for the Jacobi algebra with its canonical potential.
\end{abstract}

\maketitle

\tableofcontents

\section{Introduction}

A dimer model is a type of quiver embedded in a Riemann surface, decomposing it into oriented faces, arrows, and vertices. The decomposition gives rise to a superpotential, an element of zeroth Hochschild homology of the path algebra, that determines the relations of an associative algebra called the Jacobi algebra of the dimer. While generally noncommutative, this algebra exhibits nice homological and geometric properties. For instance, under certain conditions on the quiver, known generally as consistency, the Jacobi algebra is Calabi-Yau of dimension $3$ \cite{Davison}. If furthermore the ambient surface is a torus, the center is an affine toric Gorenstein threefold, of which the Jacobi algebra is a noncommutative crepant resolution (NCCR) \cite{Broomhead, Bocklandt2012}. Indeed, any Gorenstein affine toric threefold can be realized and resolved by dimer models in this way.

As Bocklandt shows in \cite{Bocklandt2016}, dimers provide an alternative framework for homological mirror symmetry of Riemann surfaces with $n \geq 3$ punctures. Traditionally, the mirror to a noncompact symplectic manifold is a Landau--Ginzburg model $(X, W)$ consisting of a smooth variety $X$ and a regular function $W$, called the potential. This notion can be extended to the noncommutative setting by replacing $X$ with an associative algebra $A$ and letting $W \in A$ be a central element. The relevant Landau--Ginzburg model in \cite{Bocklandt2016} and the focus of this work is $\big( \J, \ell )$, where $\Q$ is a dimer in a surface $\Sigma$, $\J$ is the Jacobi algebra, and $\ell$ is a canonical central element (see \S \ref{section Jacobi algebras}).

A matrix factorization of $(A, W)$ is a curved complex of $A$-modules. Precisely, in the $\Z$-graded setting, it is a diagram of the form  
\[
\begin{tikzcd}
M_0 \arrow[r, shift left, "d_0"] & M_1 \arrow[l, shift left, "d_1"]
\end{tikzcd} 
\]
where $M_0$ and $M_1$ are projective left $A$-modules in even and odd degrees, respectively, and $d_0$ and $d_1$ are $A$-linear maps satisfying
\[
d_1 d_0 = W \cdot Id_{P_0}, \; \; d_0 d_1 = W \cdot Id_{P_1}.
\]
Such objects form a differential $\Z$-graded category $MF(A, W)$. For the Landau--Ginzburg model $(\J, \ell)$, every arrow of $\Q$ furnishes an example of a matrix factorization (Example \ref{arrow mf}); the collection of these objects gives a full subcategory $mf(\Q) \subset \M$, which is the category on the B-side of Bocklandt's noncommutative version of mirror symmetry. 

The A-side is described in terms of a dual dimer $\Q^\vee$ that is embedded in another Riemann surface $\Sigma^\vee$. The set of arrows of $\Q^\vee$ is the same as that of $\Q$, but the vertices of $\Q^\vee$ correspond to special cycles in $\Q$ called zigzag cycles: closed paths that alternately follow clockwise and anticlockwise faces of $\Q$ (see \S \ref{section Zigzag cycles}). We consider the the $\Z$-graded wrapped Fukaya category of the ambient surface $\Sigma^\vee$ punctured at the vertices, denoted $wFuk(\Sigma^\vee \setminus \Q_0^\vee)$. The objects are exact Lagrangian submanifolds of the punctured surface, including the arrows of $\Q^\vee$ viewed as curves. The full subcategory from the arrows is denoted $fuk(\Q^\vee)$.

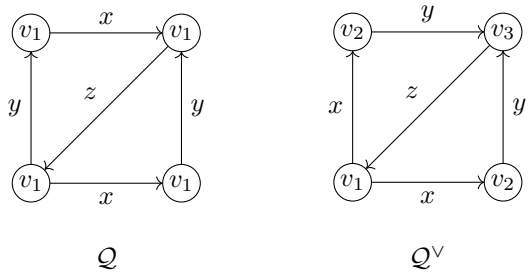
\begin{figure}[h] \label{dimer dual}
\[ 
\begin{tikzpicture}
\node at (0, 0) {$v_1$};
\draw (0,0) circle [radius=0.25];
\node at (0, 2) {$v_1$};
\draw (0,2) circle [radius=0.25];
\node at (2, 0) {$v_1$};
\draw (2,0) circle [radius=0.25];
\node at (2, 2) {$v_1$};
\draw (2,2) circle [radius=0.25];
\draw[->] (0, .25) -- (0, 1.75);
\node[left] at (0, 1) {$y$};
\draw[->] (.25, 0) -- (1.75, 0);
\node[below] at (1, 0) {$x$};
\draw[->] (.25, 2) -- (1.75, 2);
\node[above] at (1, 2) {$x$};
\draw[->] (2, .25) -- (2, 1.75);
\node[right] at (2, 1) {$y$};
\draw[<-] (0.18, .18) -- (1.82, 1.82);
\node[above left] at (1, 1) {$z$};
\node at (1, -1) {$\Q$};
\end{tikzpicture}
\qquad \qquad
\begin{tikzpicture}
\node at (0, 0) {$v_1$};
\draw (0,0) circle [radius=0.25];
\node at (0, 2) {$v_2$};
\draw (0,2) circle [radius=0.25];
\node at (2, 0) {$v_2$};
\draw (2,0) circle [radius=0.25];
\node at (2, 2) {$v_3$};
\draw (2,2) circle [radius=0.25];
\draw[->] (0, .25) -- (0, 1.75);
\node[left] at (0, 1) {$x$};
\draw[->] (.25, 0) -- (1.75, 0);
\node[below] at (1, 0) {$x$};
\draw[->] (.25, 2) -- (1.75, 2);
\node[above] at (1, 2) {$y$};
\draw[->] (2, .25) -- (2, 1.75);
\node[right] at (2, 1) {$y$};
\draw[<-] (0.18, .18) -- (1.82, 1.82);
\node[above left] at (1, 1) {$z$};
\node at (1, -1) {$\Q^\vee$};
\end{tikzpicture}
\]
\caption{ \label{dimer dual} The dimer $\Q$, with one vertex $v_1$ and arrows $x, y, z$, is embedded in a torus. The dual dimer $\Q^\vee$, which has the same arrow set but has three vertices, is embedded in a sphere.}
\end{figure}
 
\begin{thm}[\cite{Bocklandt2016} Corollary 8.4] \label{dimer mirror sym}
Suppose $\Q$ is a zigzag consistent dimer (see \S \ref{section consistency}) in a compact Riemann surface of positive genus. Then there exists an $A_\infty$-quasi-isomorphism
\[
mf(\Q) \cong fuk(\Q^\vee). 
\]
\end{thm}

Dimer models thus offer an algebraic, combinatorial way of understanding resolutions of affine toric Gorenstein threefolds and mirror constructions for punctured surfaces. Our purpose is to use dimer models to study deformation theoretic aspects of these geometries. As a general principle, deformations of algebraic structures are governed by Hochschild cohomology. Our main result is a computation of the Hochschild cohomology of the Jacobi algebra and a version of the Hochschild cohomology of the matrix factorization category when the dimer is zigzag consistent and embedded in a torus.

As is well-known, Hochschild (co)homology of an associative algebra carries the rich structure of a noncommutative calculus \cite{TTnc}. This includes the cup and cap products, the Gerstenhaber bracket, and the Connes differential, which generalize standard operations on polyvector fields and differential forms on an algebraic variety. Moreover, Van den Bergh \cite{VdB} establishes that, when the algebra is Calabi--Yau, there is a Poincar\'e--type duality between homology and cohomology that endows the latter with a Batalin--Vilkovisky (BV) structure. 

For a zigzag consistent dimer $\Q$ embedded in a torus $\Sigma$, we calculate the Hochschild cohomology of $\J$ in terms of the combinatorics of the dimer. Let $\Q_0$ be the set of vertices of the quiver, and enumerate the zigzag cycles as $Z_1, Z_2, \dots, Z_K$. For each zigzag, there are special cycles running in the opposite direction that we call antizigzags (see \S \ref{section Zigzag cycles}). Let
\[
\eta_1, \dots, \eta_N \in H_1(\Sigma, \z)
\]
be the distinct homology classes represented by the antizigzag cycles, so the opposites of these classes are represented by the zigzag cycles. Note that $K \geq N$ as there may be multiple (anti)zigzags with the same homology. Each class $\eta_i$ generates a ray 
\[
\langle \eta_i \rangle : = \R_{\geq 0} \cdot \eta_i  \subset H_1(\Sigma, \z) \otimes_{\z} \R \cong \R^2.
\]
Together, the rays divide the plane into $N$ two-dimensional cones $\sigma_1, \dots, \sigma_N$ and determine a fan (cf. \cite{Broomhead}). The center $\zc$ of the Jacobi algebra can be constructed from the antizigzag fan by methods of toric geometry. For each $1 \leq i \leq N$, let $\Int \sigma_i$ be the interior of the cone and $S_i$ be the semigroup algebra $\C[ \Int \sigma_i \cap \z^2]$. We write $N^{out}$ for the lattice of one-parameter subgroups of outer automorphisms of $\J$, which can be described combinatorially by perfect matchings (see \S \ref{section Perfect matchings}), and let $\No = N^{out} \otimes_{\z} \C$.

\begin{thm} \label{HH of Jacobi algebra}
Suppose $\Q$ is a zigzag consistent dimer in a torus. Additively, the Hochschild cohomology of the Jacobi algebra is
\begin{align*}
HH^0(\J) & \cong \zc \\
HH^1(\J) & \cong \zc \otimes \No \, \oplus \, \bigoplus_{1 \leq i \leq N} S_i \\
HH^2(\J) & \cong \zc \otimes \No \wedge \No \, \oplus \, \bigoplus_{1 \leq i \leq N} S_i \wedge \No \, \oplus \; \bigoplus_{\substack{1 \leq j \leq K \\ 1 < n}} \C \, Z_j^n  \\
HH^3(\J) & \cong \zc  \otimes \No \wedge \No \wedge \No \, \oplus \,\bigoplus_{1 \leq i \leq N} S_i \wedge \No \wedge \No \, \oplus \, \bigoplus_{\substack{1 \leq j \leq K \\ 1 \leq n}} \C \, Z_j^n \, \oplus \, \C \Q_0.
\end{align*}
\end{thm}

We also provide a complete computation of the the Batalin--Vilkovisky structure induced from the Calabi--Yau structure of $\J$. When the ambient surface of the dimer has higher genus, we give partial results and suggest a way to compute the Hochschild cohomology by relating it to the Chas--Sullivan string topology of the hyperbolic surface \cite{ChasSullivan}. 

Next, we turn to computing the Hochschild cohomology of the matrix factorization category. A Landau--Ginzburg model $(A, W)$ consists of the same data as a $\Z$-graded curved algebra: that is, a curved $\Z$-graded $A_\infty$-algebra in which all structure maps $m_n$, $n \geq 0$, are trivial except for $m_2$, the associative multiplication of $A$, and $m_0 = W$. The category of matrix factorizations of $(A, W)$ is the category of curved differential graded modules over the curved algebra. One might expect, as in Morita theory in the differential graded setting, the Hochschild cohomology of $MF(A, W)$ is isomorphic to that of $(A, W)$.

However, the presence of curvature generally makes this false; Hochschild cohomology of an algebra with nontrivial curvature is actually trivial \cite{CaldararuTu}. What is true, as shown by Polischuk--Positselski \cite{PolPos}, is that the compactly supported Hochschild cohomologies of $MF(A, W)$ and $(A, W)$ are the same. This variant of Hochschild cohomology is defined essentially by taking a direct sum totalization instead of the normal direct product totalization. It is an example of a derived functor of the second kind, the theory of which is established in \cite{Positselski}. 

C\u{a}ld\u{a}raru--Tu \cite{CaldararuTu} show how to compute the compactly supported Hochschild cohomology of a curved algebra by a spectral sequence. They carry out the computation specifically for a Landau--Ginzburg model consisting of an affine variety and a potential with isolated singularities. By their method and Theorem \ref{HH of Jacobi algebra}, we obtain the following.

\begin{thm} \label{main theorem}
Suppose $\Q$ is a zigzag consistent dimer in a torus, and let $g^\vee$ be the genus of the mirror dual. Additively, the compactly supported Hochschild cohomology of the matrix factorization category is  
\begin{align*}
HH^{\text{even}}_c \big( \M \big) & \cong \C \,  \oplus \, \bigoplus_{\substack{1 \leq j \leq K \\ 1 \leq n}} \C \, Z_j^n \\
HH^{\text{odd}}_c \big( \M \big) & \cong \C^{2 g^\vee + K- 1} \, \oplus \, \bigoplus_{\substack{1 \leq j \leq K \\ 1 \leq n}} \C \, Z_j^n.
\end{align*}
\end{thm} 

The zigzag cycles of $\Q$ are in bijective correspondence with the vertices of $\Q^\vee$ (\cite{Bocklandt2016} \S 8). Hence, the unit and the summand $\C^{2g^\vee + K - 1}$ can be viewed as the contribution from $H^*(\Sigma^\vee \setminus \Q_0^\vee, \C)$, the cohomology of the punctured surface. Additionally, each puncture contributes an even and an odd copy of $Z_j^n$. We explicitly describe the BV structure that the compactly supported cohomology inherits from $HH^*(\J)$ via the spectral sequence.

It remains to determine if the natural map from compactly supported to ordinary Hochschild cohomology is an isomorphism. Various authors prove the equivalence for certain commutative examples of Landau--Ginzburg models \cite{Dyckerhoff2011, CaldararuTu, LinPomerleano, PolPos}. For our noncommutative Landau--Ginzburg model, we do not address the question directly, but comparison with the A-side gives strong evidence for the affirmative. Ganatra \cite{Ganatra} shows that Hochschild cohomology of the wrapped Fukaya category of a punctured surface is isomorphic to the symplectic cohomology of the punctured surface. As a differential graded complex, the compactly supported cohomology in Theorem \ref{main theorem} agrees with the symplectic cohomology of the punctured surface described by the dual dimer.  

\subsection{Structure of the paper}

Section $2$ is a brief overview of background material, including dimer models, Calabi-Yau algebras, matrix factorizations, and the two kinds of Hochschild (co)homology. In section $3$, we relate the Hochschild invariants of an algebra to those of a central localization of the algebra. In particular, it is shown that, if the algebra is Calabi--Yau, Hochschild cohomology ``commutes" with central localization on the level of the BV structure. We then characterize the Hochschild cohomology of the localization of $\J$ with respect to the central subalgebra $\C[\ell]$.  

Section $4$ is devoted to the Hochschild cohomology of the Jacobi algebra, with focus on the genus $1$ case. Each cohomology group is computed, and the BV structure is calculated using results of section 3. Finally, in section $5$, the compactly supported Hochschild cohomology of $\M$ is computed, and we determine the BV structure that descends from $HH^*(\J)$ via the spectral sequence.

\subsection{Acknowledgements}

The author is indebted to Travis Schedler for guidance and support throughout the project. He would also like to thank Raf Bocklandt, Ed Segal, and Yanki Lekili for extensive discussions that refined and deepened his understanding of the problem. He is grateful to Nicol\`o Sibilla, Daniel Pomerleano, Dan Kaplan, and Jack Smith for helpful conversations and taking interest in this work.

\section{Preliminaries}

The content in this section is adapted from various sources: \S \ref{section Dimer models} to \S \ref{section Perfect matchings} comes mostly from \cite{Bocklandt2012, Bocklandt2016, abc, Broomhead}, \S \ref{section Curved algebras} to \S \ref{section Hochschild cohomology} from \cite{PolPos}, and \S \ref{section Noncommutative calculus} to \S \ref{section Calabi-Yau algebras} from \cite{VdB, TTnc, Ginzburg}. 

\subsection{Notation and conventions}

The following notation will be common throughout the text.
\begin{itemize}
\item $\Q$ is a quiver with finite vertex set $\Q_0$ and finite arrow set $\Q_1$.
\item $t, h: \Q_1 \rightarrow \Q_0$ are the tail and head functions, respectively.
\item $\C \Q$ is the path algebra of $\Q$ over the complex numbers.
\item $\bk : = \C \Q_0$ is the semisimple subalgebra of $\C \Q$ generated by the vertex idempotents. Abusing notation, we write $v \in \Q_0$ for the vertex and the idempotent. 
\item A symbol $p: v \to w$ indicates a path $p$ such that $t(p) = v$ and $h(p) = w$.
\item The unadorned tensor product $\otimes$ stands for $\otimes_{\C}$.
\end{itemize}

We use the convention of forward concatenation: for paths $p$ and $q$, the product $pq$ is nontrivial in $\C \Q$ if and only if $h(p) = t(q)$.

\subsection{Dimer models} \label{section Dimer models}

Let $\Sigma$ be a compact Riemann surface of genus $g$. We say a quiver $\Q$ \textbf{embeds} into $\Sigma$ if 
\begin{enumerate}
\item $\Q_0$ is identified with a finite subset of $\Sigma$;
\item each arrow $a \in \Q_1$ has a smooth embedding $\phi_a: [0, 1] \to \Sigma$ such that $\phi_a(0) = t(a)$, $\phi_a(1) = h(a)$, and $\phi_a([0, 1])$ is not a contractible loop;
\item the images of distinct arrows intersect only at vertices.
\end{enumerate}
Furthermore, the quiver is said to \textbf{split} $\Sigma$ if $\Sigma \setminus \Q$ is a disjoint union of open disks. The closure of such a disk is called a \textbf{face} of $\Q$, and the collection of all faces is denoted $\Q_2$.

\begin{defn}
Let $\Q$ be a quiver that splits a compact Riemann surface. We say $\Q$ is a \textbf{dimer model} if, for any face, the arrows in its boundary can be ordered as $\{ a_i \, | \, 1 \leq i \leq n \}$ where $n \geq 3$ and 
\[
h(a_i) = t(a_{i+1}) \; \forall \, i \; \text{mod} \; n.
\]
The path $a_1 a_2 \dots a_n$ is called a \textbf{boundary path}.
\end{defn}

In other words, the boundary of a face has path length at least $3$, and the arrows contained therein are oriented in the same direction. We say that a face is \textbf{positive} if the boundary arrows are oriented anticlockwise and $\textbf{negative}$ if the boundary arrows are oriented clockwise. The subsets of $\Q_2$ consisting of all positive and negative faces are denoted $\Q_2^+$ and $\Q_2^-$, respectively. Notice that every arrow is contained in exactly one positive face and one negative face.

\begin{rem}
Historically, a dimer model is defined as a bipartite graph that splits a Riemann surface. Such a graph is obtained from the cellular decomposition of $\Sigma$ dual to the quiver, the two sets of vertices coming from the positive and negative faces of $\Q$.
\end{rem}

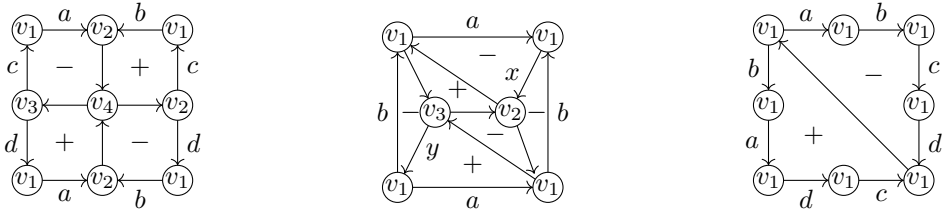
\begin{figure}[h]
\[
\begin{tikzpicture}
\draw (0,0) circle [radius=0.2];
\node at (0,0) {$v_1$};
\draw (0, 1) circle [radius=0.2];
\node at (0,1) {$v_3$};
\draw (0, 2) circle [radius=0.2];
\node at (0,2) {$v_1$};
\draw (1, 0) circle [radius=0.2];
\node at (1,0) {$v_2$};
\draw (1, 1) circle [radius=0.2];
\node at (1,1) {$v_4$};
\draw (1, 2) circle [radius=0.2];
\node at (1,2) {$v_2$};
\draw (2, 0) circle [radius=0.2];
\node at (2,0) {$v_1$};
\draw (2, 1) circle [radius=0.2];
\node at (2,1) {$v_2$};
\draw (2, 2) circle [radius=0.2];
\node at (2,2) {$v_1$};
\draw[->] (0.2, 0) -- (0.8, 0);
\node[below] at (0.5, 0) {$a$};
\draw[<-] (1.2, 0) -- (1.8, 0);
\node[below] at (1.5, 0) {$b$};
\draw[<-] (0, 0.2) -- (0, 0.8);
\node[left] at (0, 0.5) {$d$};
\draw[->] (0, 1.2) -- (0, 1.8);
\node[left] at (0, 1.5) {$c$};
\draw[<-] (0.2, 1) -- (0.8, 1);
\draw[<-] (1, 0.8) -- (1, 0.2);
\draw[<-] (1, 1.2) -- (1, 1.8);
\draw[->] (1.2, 1) -- (1.8, 1);
\draw[->] (2, 0.8) -- (2, 0.2);
\node[right] at (2, 0.5) {$d$};
\draw[->] (2, 1.2) -- (2, 1.8);
\node[right] at (2, 1.5) {$c$}; 
\draw[<-] (1.2, 2) -- (1.8, 2);
\node[above] at (1.5, 2) {$b$};
\draw[->] (0.2, 2) -- (0.8, 2);
\node[above] at (0.5, 2) {$a$};
\node at (0.5, 0.5) {$+$};
\node at (1.5, 0.5) {$-$};
\node at (0.5, 1.5) {$-$};
\node at (1.5, 1.5) {$+$};
\end{tikzpicture}
\qquad \qquad \qquad
\begin{tikzpicture}
\draw (0, 0) circle [radius=0.2];
\node at (0, 0) {$v_1$};
\draw (0, 2) circle [radius=0.2];
\node at (0, 2) {$v_1$};
\draw (2, 0) circle [radius=0.2];
\node at (2, 0) {$v_1$};
\draw (2, 2) circle [radius=0.2];
\node at (2, 2) {$v_1$};
\draw (0.5, 1) circle [radius=0.2];
\node at (0.5, 1) {$v_3$};
\draw (1.5, 1) circle [radius=0.2];
\node at (1.5, 1) {$v_2$};
\draw[->] (0.2, 0) -- (1.8, 0);
\node[below] at (1, 0) {$a$};
\draw[->] (0.2, 2) -- (1.8, 2);
\node[above] at (1, 2) {$a$};
\draw[->] (0, 0.2) -- (0, 1.8);
\node[left] at (0, 1) {$b$};
\draw[->] (2, 0.2) -- (2, 1.8);
\node[right] at (2, 1) {$b$};
\draw[<-] (0.09, 0.18) -- (0.41, 0.82);
\node[below right] at (0.25, 0.7) {$y$};
\draw[->] (0.7, 1) -- (1.3, 1);
\draw[->] (0.09, 1.82) -- (0.41, 1.18);
\draw[->] (1.91, 1.82) -- (1.59, 1.18);
\node[above left] at (1.75, 1.3) {$x$};
\draw[->] (1.59, 0.82) -- (1.82, 0.18);
\draw[->] (1.33, 1.11) -- (0.17, 1.89);
\draw[->] (1.83, 0.11) -- (0.67, 0.89);
\node at (1, 0.3) {$+$};
\node at (0.175, 1) {$-$};
\node at (1.3, 0.7) {$-$};
\node at (0.8, 1.3) {$+$};
\node at (1.2, 1.75) {$-$};
\node at (1.85, 1) {$-$};
\end{tikzpicture}
\qquad \qquad \qquad
\begin{tikzpicture}
\draw (0,0) circle [radius=0.2];
\node at (0,0) {$v_1$};
\draw (0, 1) circle [radius=0.2];
\node at (0,1) {$v_1$};
\draw (0, 2) circle [radius=0.2];
\node at (0,2) {$v_1$};
\draw (1, 0) circle [radius=0.2];
\node at (1,0) {$v_1$};
\draw (1, 2) circle [radius=0.2];
\node at (1,2) {$v_1$};
\draw (2, 0) circle [radius=0.2];
\node at (2,0) {$v_1$};
\draw (2, 1) circle [radius=0.2];
\node at (2,1) {$v_1$};
\draw (2, 2) circle [radius=0.2];
\node at (2,2) {$v_1$};
\draw[->] (0.2, 0) -- (0.8, 0);
\node[below] at (0.5, 0) {$d$};
\draw[->] (1.2, 0) -- (1.8, 0);
\node[below] at (1.5, 0) {$c$};
\draw[<-] (0, 0.2) -- (0, 0.8);
\node[left] at (0, 0.5) {$a$};
\draw[<-] (0, 1.2) -- (0, 1.8);
\node[left] at (0, 1.5) {$b$};
\draw[->] (2, 0.8) -- (2, 0.2);
\node[right] at (2, 0.5) {$d$};
\draw[<-] (2, 1.2) -- (2, 1.8);
\node[right] at (2, 1.5) {$c$}; 
\draw[->] (1.2, 2) -- (1.8, 2);
\node[above] at (1.5, 2) {$b$};
\draw[->] (0.2, 2) -- (0.8, 2);
\node[above] at (0.5, 2) {$a$};
\draw[->] (1.86, 0.14) -- (0.14, 1.86);
\node at (0.6, 0.6) {$+$};
\node at (1.4, 1.4) {$-$};
\end{tikzpicture}
\]
\caption{ \label{dimer examples} Examples of dimers from \cite{Bocklandt2016}. The first two are embedded in a torus, while the third is embedded in a genus $2$ surface.}
\end{figure}

\subsection{Jacobi algebras} \label{section Jacobi algebras}

A \textbf{superpotential} of a quiver $\Q$ is any element 
\[
\Phi \in HH_0(\C \Q) \cong \C \Q / [\C \Q, \C \Q],
\] 
the vector space spanned by closed paths up to cyclic permutation of the arrows. Ginzburg \cite{Ginzburg} defines a linear map 
\[
\partial_a: \C \Q / [\C \Q, \C \Q] \rightarrow \C \Q
\]
called the cyclic derivative with respect to $a \in \Q_1$: if $a_1, \dots, a_n \in \Q_1$ and $[a_1 \dots a_n]$ is the class of their product in the quotient, then
\[
\partial_a[a_1 \dots a_m] = 
\sum_{i  \, | \, a_i = a} a_{i+1} \dots a_m a_1 \dots a_{i-1}.
\]
The \textbf{Jacobi algebra} of the pair $(\Q, \Phi)$ is the quotient of $\C \Q$ by the two-sided ideal generated by the cyclic derivatives of $\Phi$,
\[
J(\Q, \Phi) : = \frac{\C \Q}{( \partial_a \Phi \, | \, a \in \Q_1)}.
\]    

If $\Q$ is a dimer model, a face $F \in \Q_2$ determines a unique element $\partial F \in \C \Q / [ \C \Q, \C \Q]$, the projection of any boundary path of $F$. So there is a canonical superpotential
\begin{equation} 
\Phi_0 = \sum_{F \in \Q_2^+} \partial F - \sum_{F \in \Q_2^-} \partial F.
\end{equation}

\begin{defn}
The \textbf{Jacobi algebra} of a dimer model $\Q$ is the algebra $\J : = J(\Q, \Phi_0)$.
\end{defn} 

We can write the relations defining $\J$ explicitly. Given $a \in \Q_1$, let $R^+_a$ be the path in $\Q$ such that $a R^+_a$ is the boundary of a positive face, and let $R^-_a$ be the path such that $a R^-_a$ is the boundary of a negative face. Then
\begin{equation} \label{relations}
\partial_a \Phi_0 = R^+_a - R^-_a.
\end{equation}
Since the boundary of a face has path length at least $3$, the terms $R^{\pm}_a$ have path length at least $2$, but they need not have equal length. We call $R_a^+$ and $R_a^-$ \textbf{partial cycles}.

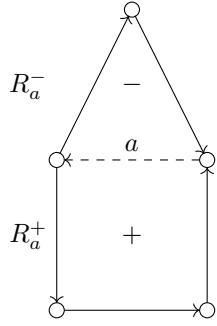
\[
\begin{tikzpicture}
\draw (0,0) circle [radius=0.1];
\draw (0,-2) circle [radius=0.1];
\draw (2,-2) circle [radius=0.1];
\draw (2,0) circle [radius=0.1];
\draw (1, 2) circle [radius=0.1];
\node at (1, -1) {$+$};
\node at (1, 1) {$-$};
\node[left] at (0, -1) {$R^+_a$};
\node[left] at (0, 1) {$R^-_a$};
\draw[->, dashed] (1.9, 0) -- (0.1, 0);
\node[above] at (1, 0) {$a$}; 
\draw[->] (0, -0.1) -- (0, -1.9);
\draw[->] (0.1, -2) -- (1.9, -2);
\draw[->] (2, -1.9) -- (2, -0.1);
\draw[->] (0.04, 0.08) -- (0.96, 1.92);
\draw[->] (1.04, 1.92) -- (1.96, 0.08);
\end{tikzpicture}
\]

Let $\zc$ be the center of the algebra $\J$. There is a canonical element $\ell \in \zc$, called the \textbf{potential} of $\J$, that is constructed as follows. For each vertex $v \in \Q_0$, choose a boundary path beginning and ending at $v$. Let $p_v$ be its image in $\J$, and let $\ell$ be the sum
\[
\ell : = \sum_{v \in \Q_0} p_v.
\]
It is straightforward to check that $\ell$ is independent of the choices made and is indeed central.

\begin{exmps}
For the first dimer in Figure \ref{dimer dual}, the path algebra is the free associative algebra $\C \langle x, y, z \rangle$. The Jacobi algebra is
\[
\J = \frac{ \C \langle x, y, z \rangle}{(yz - zy, zx - xz, xy -yx)} \cong \C[x, y, z],
\]
and the potential is $\ell = xyz$. In this exceptional case, the Landau--Ginzburg model is commutative, coinciding with the mirror to the pair of pants in \cite{Abouzaid}.

For the last dimer in Figure \ref{dimer examples}, the Jacobi algebra is 
\[
\J = \frac{\C \langle a, b ,c, d \rangle}{( bcd - dcb, cda - adc, dab - bad, abc - cba)},
\]
which is noncommutative, and the potential is $\ell = abcd$.
\end{exmps}

\subsection{Zigzag cycles} \label{section Zigzag cycles}

Let $\Q$ be a dimer embedded in $\Sigma$, and let $\pi: \ur \rightarrow \Sigma$ be the universal covering map. We may lift the dimer to a quiver $\tilde{\Q} = (\tilde{\Q}_0, \tilde{\Q_1})$ embedded in $\ur$, with infinite sets of vertices and arrows when the genus is positive.

\begin{defn}
A \textbf{zigzag flow} in $\tilde{\Q}$ is an infinite path
\[
\tilde{Z} : = \dots a_{-2} a_{-1} a_0 a_1 a_2 \dots, \; \; a_i \in \tilde{\Q}_1
\]
such that $h(a_i) = t(a_{i+1})$ and
\begin{enumerate}
\item for all $i \in \z$, $a_i a_{i+1}$ is contained in a positive face when $i$ is even and a negative face when $i$ is odd
\item or, for all $i \in \z$, $a_i a_{i+1}$ is contained in a negative face when $i$ is even and a positive face when $i$ is odd.
\end{enumerate}
\end{defn}

Since $\Q$ is finite, the projection of a zigzag flow under the covering map is periodic. We refer to a single period as a \textbf{zigzag path} or, when considered up to cyclic permutation of the arrows, a \textbf{zigzag cycle}. If $a_i a_{i+1}$ is contained in a positive face, then the arrow $a_i$ is called a \textbf{zig} of the zigzag flow (path, cycle); otherwise, it is called a \textbf{zag}. Any arrow $a \in \Q_1$ is contained in exactly two zigzags, one in which $a$ is a zig and one in which $a$ is a zag. 

For a given zigzag flow $\tilde{Z}$, two infinite paths can be constructed that run in the opposite direction: the \textbf{positive antizigzag flow} to $\tilde{Z}$, denoted $O^+(\tilde{Z})$, consisting of all arrows in the positive faces that meet $\tilde{Z}$ but are not in $\tilde{Z}$, and the \textbf{negative antizigzag flow} to $\tilde{Z}$, denoted $O^-(\tilde{Z})$, consisting of all arrows in the negative faces that meet $\tilde{Z}$ but are not in $\tilde{Z}$. We define and notate antizigzag paths and cycles analogously.

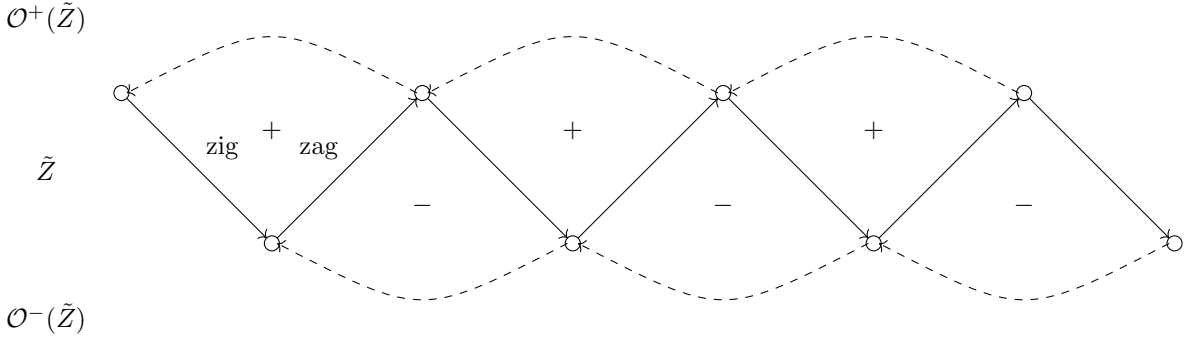
\[
\begin{tikzpicture}
\draw (-2, 2) circle [radius=0.1];
\draw (0,0) circle [radius=0.1];
\draw (2,2) circle [radius=0.1];
\draw (4,0) circle [radius=0.1];
\draw (6,2) circle [radius=0.1];
\draw (8,0) circle [radius=0.1];
\draw (10,2) circle [radius=0.1];
\draw (12,0) circle [radius=0.1];
\draw[->](-1.93, 1.93) -- (-0.07, 0.07);
\node[above right] at (-1, 1) {zig};
\draw[->] (0.07, 0.07) -- (1.93, 1.93);
\node[above left] at (1, 1) {zag};
\draw[->] (2.07, 1.93) -- (3.93, 0.07);
\draw[->] (4.07, 0.07) -- (5.93, 1.93);
\draw[->] (6.07, 1.93) -- (7.93, 0.07);
\draw[->] (8.07, 0.07) -- (9.93, 1.93);
\draw[->] (10.07, 1.93) -- (11.93, 0.07);
\draw[->, dashed] (3.93, 0)..controls (2, -1) and (2, -1)..(0.07, 0);
\draw[->, dashed] (7.93, 0)..controls (6, -1) and (6, -1)..(4.07, 0);
\draw[->, dashed] (11.93, 0)..controls (10, -1) and (10, -1)..(8.07, 0);
\draw[->, dashed] (1.93, 2)..controls (0, 3) and (0, 3)..(-1.93,2);
\draw[->, dashed] (5.93, 2)..controls (4, 3) and (4, 3)..(2.07,2);
\draw[->, dashed] (9.93, 2)..controls (8, 3) and (8, 3)..(6.07,2);
\node at (2, 0.5) {$-$};
\node at (6, 0.5) {$-$};
\node at (10, 0.5) {$-$};
\node at (0, 1.5) {$+$};
\node at (4, 1.5) {$+$};
\node at (8, 1.5) {$+$};
\node at (-3, 3) {$\st^+(\tilde{Z})$};
\node at (-3, 1) {$\tilde{Z}$};
\node at (-3, -1) {$\st^-(\tilde{Z})$};
\end{tikzpicture}
\]

Zigzag and antizigzag paths are topological $1$-cycles in $\Sigma$. We write 
\[
\eta_1, \eta_2, \dots, \eta_N \in H_1(\Sigma, \z)
\]
for all distinct homology classes represented by the antizigag paths in $\Q$, so then
\[
-\eta_1, -\eta_2, \dots, -\eta_N \in H_1(\Sigma, \z)
\] 
are the distinct homology classes represented by the zigzag paths. Note that there may be multiple (anti)zigzag cycles representing the same homology class, as demonstrated in Example \ref{pinchpoint}. Let $m_i$ be the number of zigzag cycles of homology $-\eta_i$. The total number of zigzag cycles is then 
\[
K : = \sum_{i = 1}^N m_i. 
\]

\subsection{Consistency} \label{section consistency}

For a dimer model $\Q$, consider the algebra 
\[
\JW : = \J \otimes_{\C[\ell]} \C[\ell, \ell^{-1}],
\]
the Ore localization of $\J$ with respect to the multiplicative set $\{ \ell^n \, | \, n \in \z_{\geq 0} \}$. This algebra is also realized as the quotient of the path algebra of the double quiver $\overline{\Q}$ by the relations \eqref{relations} and
\begin{equation} \label{inverse relations}
aa^{-1} = t(a), \; a^{-1}a = h(a) \; \; \forall \, a \in \Q_1.
\end{equation}
Let $L: \J \rightarrow \JW$ be the map sending a path to its image in the localization. From the first description of $\JW$, it is clear that the kernel of $L$ is precisely the torsion of $\J$ under the action of $\C[\ell]$. From the second description, it is apparent that $\JW$ has the following cancellation property: if $p$ and $q$ are paths in $\JW$ and $a \in \overline{\Q}_1$ such that $h(p) = h(q) = t(a)$, then
\[
p a = q a \Longrightarrow p = q,
\]
or if $t(p) = t(q) = h(a)$, then
\[
ap = aq \Longrightarrow p = q.
\] 
So if $\J$ is torsion-free and $L$ is injective, then $\J$ inherits the cancellation property, and we say $\J$ is \textbf{cancellation}.

When the genus of $\Sigma$ is positive, this algebraic condition is equivalent to a geometric one.  Given an arrow $a \in \tilde{\Q}_1$, truncate the two zigzag flows containing $a$ to obtain semi-infinite paths
\begin{align*}
\tilde{Z}^+_a & =  a a_1 a_2 \dots a_i \dots \\
\tilde{Z}^-_a & = a b_1 b_2 \dots b_j \dots
\end{align*}

\begin{defn}
A dimer $\Q$ is \textbf{zigzag consistent} if, for all $a \in \tilde{\Q}_1$, $\tilde{Z}^+_a$ and $\tilde{Z}^-_a$ intersect only in $a$: that is, $a_i \neq b_j$ for any $i, j \in \z_{>0}$.
 \end{defn}

Note that a dimer in a sphere can never be zigzag consistent because $\tilde{\Q}$ is finite. According to Theorem 5.5 of \cite{Bocklandt2012}, if $\Q$ is a dimer embedded in a surface of positive genus, then $\J$ is cancellation if and only if $\Q$ is zigzag consistent.

\begin{exmp}
The first and third dimers in Figure \ref{dimer examples} are zigzag consistent. The second dimer is not, as the two zigzag paths starting at $x$ intersect in $y$.
\end{exmp}

Zigzag consistency implies that zigzag paths are homotopically nontrivial simple closed curves in $\Sigma$. Therefore, the homology classes $\eta_i$, $1 \leq i \leq N$, are nonzero primitive elements of $H_1(\Sigma, \z)$. Moreover, in the genus $1$ case, zigzag flows behave almost like straight lines in the plane, as suggested by the following properties:


\begin{prop}[\cite{Bocklandt2012, Broomhead}] \label{intersection properties}
Suppose $\Q$ is a zigzag consistent dimer model in a torus. 
\begin{enumerate}
\item If two zigzag cycles have the same homology class, they do not intersect in any arrow.
\item If two zigzag cycles have linearly independent homology, then they intersect in at least one arrow. 
\end{enumerate}
\end{prop}

Because of the first statement of the proposition, distinct zigzag cycles (paths) of the same homology class are said to be \textbf{parallel}.




 
\subsection{Perfect matchings} \label{section Perfect matchings}

Every dimer $\Q$ provides a cellular cochain complex
\[
\begin{tikzcd}
\z^{\Q_0} \arrow[r, "\partial"] & \z^{\Q_1} \arrow[r, "\partial"] & \z^{\Q_2}.
\end{tikzcd}
\]
With $v^*$, $a^*$, and $F^*$ the duals to a vertex $v$, an arrow $a$, and a face $F$, the coboundary maps are defined by the formulas
\[
\partial v^*: a \mapsto 
\begin{cases}
1 & \text{if} \; t(a) \neq v  = h(a) \\
-1 & \text{if} \; t(a) = v \neq h(a) \\
0 & \text{otherwise}
\end{cases},
\qquad
\partial a^* : F \mapsto 
\begin{cases}
1 & \text{if} \; a \in \partial F \\
0 & \text{otherwise}
\end{cases}
\]
It is apparent that $H^0(\Sigma, \z) \cong \z$ is the rank $1$ submodule generated by the primitive element 
\[ 
\sum_{v \in \Q_0} v^*.
\]
Let $N^{in} = \partial (\z^{\Q_0})$, the space of $1$-coboundaries. 

An element of $\varphi \in \z^{\Q_1}$ is an integer grading on the path algebra $\C \Q$. As explained in \cite{Broomhead}, the associated $\C^*$-action defines a one-parameter subgroup of the automorphism group of $\C \Q$: for any $t \in \C^*$ and $a \in \Q_1$,
\[
a \mapsto t^{\varphi (a)} \cdot a.
\]
In order for the grading to descend to the Jacobi algebra, the relations \eqref{relations} must be homogeneous in the grading; in particular, the sum of the degrees of the arrows in the boundary of a face must be the same for all faces. This is precisely the condition of being contained in
\[
N : = \partial^{-1} \big( \z \cdot \sum_{F \in \Q_2} F^* \big).
\] 
The quotient $N^{out} : = N / N^{in}$ is a rank $2g+1$ free abelian group containing $H^1(\Sigma, \z) \cong \partial^{-1}(0) / N^{in}$. Note that the one-parameter subgroup to which $\partial v^*$ corresponds is conjugation by
\[ 
tv + \sum_{\substack{w \in \Q_0 \\ w \neq v}} w, \; \; t \in \C^*,
\]
so the 1-coboundaries span the one-parameter subgroups of inner automorphisms. Hence, $N^{out}$ can be identified as the lattice of one-parameter subgroups of outer automorphisms of $\J$.


\begin{defn} 
A \textbf{perfect matching} (or simply matching) $\Perf$ is a subset of $\Q_1$ containing exactly one arrow from every face. Such a subset defines an element of $N$ by
\[
\degree_{\Perf} (a) = \begin{cases}
1 & \text{if} \; a \in \Perf \\ 
0 & \text{otherwise}.
\end{cases}
\]
We label the collection of all perfect matchings by $PM(\Q)$.
\end{defn}

Not all dimers admit a perfect matching. However, as long as $\Q$ admits a strictly positive grading, then a perfect matching exists, and $N$ is integrally generated by the perfect matchings (\cite{Broomhead} Lemma 2.11).

The difference of two matchings is a cocycle. Fixing a reference matching $\Perf_o$, we obtain a lattice polytope from the convex hull of the image of $\{ \Perf - \Perf_o \, | \, \Perf \; \text{a perfect matching in} \; \Q\}$ in $H^1(\Sigma, \z) \otimes_{\z} \R \cong \R^{2g}$, unique to $\Q$ up to affine integral transformation. This polytope, denoted $MP(\Q)$, is called the \textbf{matching polytope}.

When $\Q$ is a zigzag consistent dimer embedded in a torus, the combinatorics of perfect matchings are especially well understood. A matching is classified as
\begin{itemize}
\item an \emph{internal matching} if its image lies in the interior of $MP(\Q)$,
\item a \emph{boundary matching} if its image lies on the boundary of $MP(\Q)$, and
\item a \emph{corner matching} if its image lies at a corner of $MP(\Q)$.
\end{itemize}
Each lattice point not on a corner is the image of at least one matching, and each corner is the image of a unique matching. The homology classes of the zigzag cycles $\{-\eta_i \, | \, 1 \leq i \leq N \}$ are, in fact, the outward pointing normal vectors to $MP(\Q)$ \cite{Gulotta}; thus, they are in bijection with the corner matchings, which can then be labeled as $\{\Perf_i \, | \, 1 \leq i \leq N \}$. We may assume the ordering on corner matchings and homology classes is such that
\begin{enumerate}
\item $\Perf_{i+1}$ succeeds $\Perf_i$ ($i$ mod $N$) in $MP(\Q)$ in the counterclockwise direction and 
\item $-\eta_{i+1}$ is the outward normal vector to the boundary segment between $\Perf_i$ and $\Perf_{i+1}$ ($i$ mod $N$).
\end{enumerate}

\begin{thm}[\cite{Gulotta} \S 3; see also \cite{abc} Theorem 1.47] \label{matching polygon}
Suppose $\Q$ is a zigzag consistent dimer in a torus. 
\begin{enumerate}
\item The corner matchings $\Perf_i$ and $\Perf_{i+1}$ contain the zigs and zags, respectively, of all zigzag cycles of homology $-\eta_{i+1}$. In each boundary of a face that does not meet a zigzag cycle of homology $-\eta_{i+1}$, $\Perf_i$ and $\Perf_{i+1}$ coincide.
\item The number of lattice points on the boundary of $MP(\Q)$ between $\Perf_i$ and $\Perf_{i+1}$, inclusive, is $m_{i+1} + 1$. On this segment, the $d^{th}$ lattice point from $\Perf_i$ represents perfect matchings that are the union of $\Perf_i \cap \Perf_{i+1}$, all arrows in $\Perf_i$ from $d$ chosen zigzag cycles of homology $-\eta_{i+1}$, and all arrows in $\Perf_{i+1}$ from the remaining $m_{i+1}-d$ zigzag cycles of homology $-\eta_{i+1}$.
\item An internal matching meets every nontrivial closed path of $\Q$.
\end{enumerate}
\end{thm}

As can be deduced from the theorem, the antizigzag paths $\mathcal{O}^{\pm}(Z)$ to a zigzag $Z$ of homology $-\eta_{i+1}$ have degree $0$ in $\Perf_i$, $\Perf_{i+1}$, and all boundary matchings between them. Since the potential $\ell$ has degree $1$ in all matchings, an antizigzag is not a multiple of $\ell$ in $\J$.

\begin{defn}
A path (cycle) is \textbf{minimal} if it is not a multiple of $\ell$ in $\J$.
\end{defn}

In a zigzag consistent toric dimer, there exists a minimal path between any two vertices and of any given homotopy class (\cite{Broomhead} Chapter 6; \cite{abc} Lemma 3.18). The minimal path is unique in $\J$, since homotopic paths differ by a factor of $\ell^n$ for some $n \geq 0$ (\cite{Bocklandt2012} Lemma 7.4).

\begin{exmp} \label{pinchpoint}
Consider the following dimer in a torus, known as the suspended pinchpoint (\cite{abc} Example 1.5).
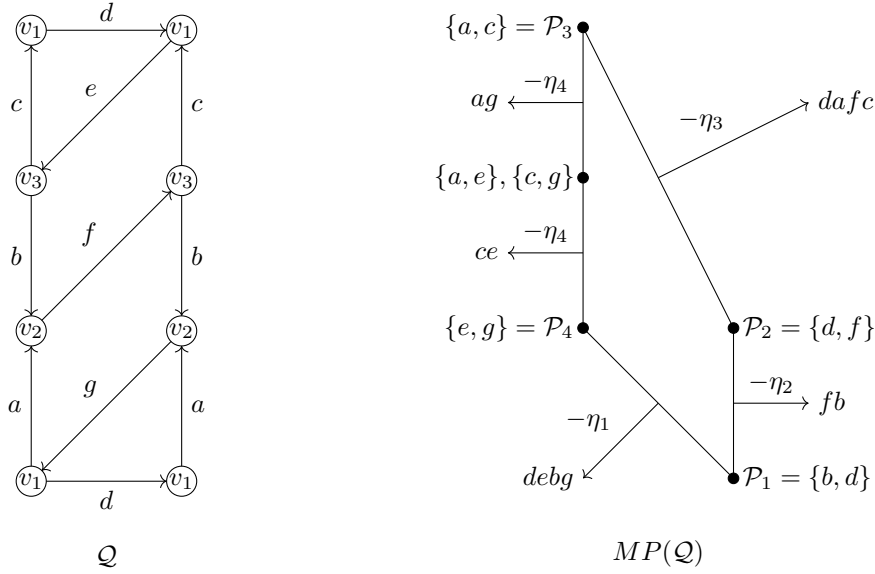
\[ 
\begin{tikzpicture}
\draw (0,0) circle [radius=0.2];
\node at (0, 0) {$v_1$};
\draw (0,2) circle [radius=0.2];
\node at (0, 2) {$v_2$};
\draw (0,4) circle [radius=0.2];
\node at (0, 4) {$v_3$};
\draw (0,6) circle [radius=0.2];
\node at (0, 6) {$v_1$};
\draw (2,0) circle [radius=0.2];
\node at (2, 0) {$v_1$};
\draw (2,2) circle [radius=0.2];
\node at (2, 2) {$v_2$};
\draw (2,4) circle [radius=0.2];
\node at (2, 4) {$v_3$};
\draw (2,6) circle [radius=0.2];
\node at (2, 6) {$v_1$};
\node at (1, -1) {$\Q$};
\draw[->] (.2, 0) -- (1.8, 0);
\node[below] at (1,0) {$d$};
\draw[->] (0, 0.2) -- (0, 1.8);
\node[left] at (0, 1) {$a$};
\draw[<-] (0, 2.2) -- (0, 3.8);
\node[left] at (0, 3) {$b$};
\draw[->] (0, 4.2) -- (0, 5.8);
\node[left] at (0, 5) {$c$};
\draw[->] (0.2, 6) -- (1.8, 6);
\node[above] at (1, 6) {$d$};
\draw[->] (2, 0.2) -- (2, 1.8);
\node[right] at (2, 1) {$a$};
\draw[<-] (2, 2.2) -- (2, 3.8);
\node[right] at (2, 3) {$b$};
\draw[->] (2, 4.2) -- (2, 5.8);
\node[right] at (2, 5) {$c$};
\draw[->] (1.86, 5.86) -- (0.14, 4.14);
\node[above left] at (1, 5) {$e$};
\draw[->] (0.14, 2.14) -- (1.86, 3.86);
\node[above left] at (1, 3) {$f$};
\draw[->] (1.86, 1.86) -- (0.14, 0.14);
\node[above left] at (1, 1) {$g$};
\end{tikzpicture}
\\
\qquad \qquad \qquad \qquad \qquad
\\
\begin{tikzpicture}
\draw (0,2) -- (0,6);
\draw (0,6) -- (2,2);
\draw (2,2) -- (2, 0);
\draw (2,0) -- (0,2);
\draw[->] (0, 3) -- (-1, 3);
\node[left] at (-1, 3) {$ce$};
\draw[->] (0, 5) -- (-1, 5);
\node[above] at (-0.5, 3) {$-\eta_4$};
\node[left] at (-1, 5) {$ag$};
\node[above] at (-0.5, 5) {$-\eta_4$};
\draw[->] (1, 4) -- (3, 5);
\node[right] at (3, 5) {$dafc$};
\node[above left] at (2, 4.5) {$-\eta_3$};
\draw[->] (2, 1) -- (3, 1);
\node[right] at (3, 1) {$fb$};
\node[above] at (2.5, 1) {$-\eta_2$};
\draw[->] (1, 1) -- (0, 0);
\node[left] at (0, 0) {$debg$};
\node[above left] at (0.5, 0.5) {$-\eta_1$};
\filldraw[black] (0,2) circle (2pt) node[anchor=east] {$\{e, g\} = \Perf_4$ }; 
\filldraw[black] (0,4) circle (2pt) node[anchor=east] {$\{a, e\}, \{c, g\}$};
\filldraw[black] (0,6) circle (2pt) node[anchor=east] {$ \{ a, c\} = \Perf_3 $};
\filldraw[black] (2,2) circle (2pt) node[anchor=west] {$ \Perf_2 = \{d, f\} $};
\filldraw[black] (2, 0) circle (2pt) node[anchor= west] {$\Perf_1 = \{b, d\}$};
\node at (1, -1) {$MP(\Q)$};
\end{tikzpicture}
\]
The matching polygon $MP(\Q)$ is presented with respect to the basis of $H_1(\Sigma, \z)$ given by the homology classes of $d$ and $afcec$. The zigzag paths are listed next to the outward normal vectors representing their homology classes. Notice that $ag$ and $ce$ are parallel while every other homology class consists of only one zigzag cycle. The perfect matchings are listed next to the corresponding lattice points: there are four corner matchings, two boundary matchings, and no internal matchings.
\end{exmp}

\subsection{Topological gradings} \label{section The localization}

Let $\Q$ be a dimer model embedded in $\Sigma$. The integer grading on the Jacobi algebra by a perfect matching can be extended to the localization by the rule
\[
\degree_{\Perf}(a^{-1}) : = - \degree_{\Perf}(a) \; \; \forall a \in \Q_1.
\]

In addition, we consider gradings on $\JW$ encoding topological information. The embedding of $\Q$ into $\Sigma$ can be extended to the double $\bar{\Q}$ by letting the image of $a^{-1}$ be the inverse path of the image of $a \in \Q_1$. Then the terms in each relation of the localization \eqref{relations}, \eqref{inverse relations} are homotopic paths in the surface. In fact, the class of a path in $\JW$  is uniquely determined by its homotopy class and degree in any perfect matching (\cite{Bocklandt2007} Lemma 7.2).

Fix a perfect matching $\Perf_o$, a vertex $v_o$ as a basepoint, and for every $v \in \Q_0$, a path $p_v: v_o \to v$ in $\JW$. We may take $p_{v_o}$ to be the idempotent $v_o$ and, multiplying by the appropriate power of $\ell$, may ensure that $\degree_{\Perf_o}(p_v) = 0$ for all vertices. Then define a $\pi_1(\Sigma, v_o)$-grading on $\JW$ by assigning to a path $p$ the homotopy class of the path $p_{t(p)} \, p \, p_{h(p)}^{-1}$, denoted $|p|$. Passing to the abelianization of the fundamental group gives a grading by $H_1(\Sigma, \z)$, independent of the choice of basepoint and connecting paths $p_v$.  The Jacobi algebra inherits these gradings via the universal map $L: \J \rightarrow \JW$. 

By keeping track of the gradings and the head and tail data, $\JW$ can be realized as a matrix algebra.

\begin{thm}[\cite{Bocklandt2007} Theorem 7.4] \label{matrix algebra}
Let $\Q$ be a dimer model and $\Perf_o$ be a perfect matching. The map
\[
\Psi: \JW \rightarrow Mat_{\# \Q_0}(\C[\pi_1(\Sigma, v_o)] \otimes \C[z^{\pm 1}])
\]
sending a path $p: v \rightarrow w$ to
\[
\big( |p| \otimes z^{\degree_{\Perf_o}(p)} \big) e_{vw},
\] 
where $e_{vw}$ is the $(v, w)$-elementary matrix, is an isomorphism of algebras.
\end{thm}

\begin{rem}
The map $\Psi$ depends on the choice of $\Perf_o$, $v_o$, and connecting paths $p_v$, though we exclude these from the notation.
\end{rem}  

When the ambient surface is a torus, $\JW$ is Morita equivalent to the algebra of Laurent polynomials in three variables. In the hyperbolic case, the group algebra $\C[\pi_1(\Sigma, v_0)]$ is noncommutative, but its Hochschild (co)homology is still well understood (see \S \ref{section BV localized Jacobi}).

\subsection{Curved algebras and matrix factorizations} \label{section Curved algebras}

Let $\Gamma$ be either $\z$ or $\Z$. A $\Gamma$-graded \textbf{curved algebra} is a pair $(A, W)$ where $A$ is a $\Gamma$-graded associative algebra and $W \in A$ is a central element of degree 2. This is a special example of a curved $A_\infty$-algebra, in which all the structure maps are trivial except those of order $0$ and $2$. When $A$ is entirely in even degree, we call $(A, W)$ a \textbf{Landau--Ginzburg model}. A \textbf{matrix factorization} of a Landau--Ginzburg model $(A, W)$ is a curved differential graded module that is finitely generated and projective as a graded $A$-module. As a diagram, this can be presented as
\[
(M_*, d_M) : =
\begin{tikzcd}
M_0 \arrow[r, shift left, "d_M"] & M_1 \arrow[l, shift left, "d_M"]
\end{tikzcd} 
\]
where $M_0$ and $M_1$ are finitely generated, projective left $A$-modules concentrated in even and odd degrees, respectively, and $d_M$ is a degree $+1$ $A$-linear map satisfying $d_M^2 = W \cdot Id$. We write $MF(A, W)$ for the differential $\Gamma$-graded category whose
\begin{enumerate}
\item objects are matrix factorizations of $(A, W)$ and
\item morphism space between two objects $M_*$ and $N_*$ is the internal $\Hom$ of $\Gamma$-graded vector spaces
\[
\underline{\Hom}(M_*, N_*),
\]
equipped with the commutator differential
\[
\delta(f) : = d_{N} f - (-1)^{|f|} f d_{M}.
\] 
\end{enumerate}
One easily checks that, indeed, $\delta^2 = 0$, despite the objects having curvature.

\begin{exmp} \label{arrow mf} 
For a dimer model $\Q$, we consider $(\J, \ell)$ to be a $\Z$-graded Landau--Ginzburg model with $\J$ concentrated in even degree. Suppose $p$ and $q$ are paths in $\Q$ that factorize the boundary of a face. In the Jacobi algebra, $\ell \cdot t(p) = pq$ and $\ell \cdot h(p) = qp$. Thus, we have a matrix factorization
\[
\begin{tikzcd}
M^{p, q} : = \J t(p) \arrow[r, shift left, "\cdot p"] & \J h(p) \arrow[l, shift left, "\cdot q"]
\end{tikzcd}
\]
where $\cdot p$ and $\cdot q$ indicate multiplication on the right by $p$ and $q$, respectively.

If $v$ and $w$ are two vertices, a left $\J$-module morphism from $\J v$ to $\J w$ is determined by the image of $v$, which must be an element of $v \J w$. Hence, if $r$ and $s$ are also paths factorizing the boundary of a face, then as $\Z$-graded vector spaces, 
\[
\underline{\Hom}(M^{p, q}, M^{r, s}) = t(p) \J t(r) \oplus h(p) \J h(r) \oplus t(p) \J h(r) [1] \oplus h(p) \J t(r) [1],
\]
where $[1]$ denotes the shift in parity. 

Observe that every arrow $a \in \Q_1$ yields a matrix factorization
\[
\begin{tikzcd}
M^{a, R_a^+} : = \J t(a) \arrow[r, shift left, "\cdot a"] & \J h(a) \arrow[l, shift left, "\cdot R_a^+"].
\end{tikzcd} 
\]
where $R_a^+$ is the path in \eqref{relations}. Using $R_a^-$ instead of $R_a^+$ defines the same matrix factorization because the two paths are equivalent in $\J$. The full subcategory of $\M$ consisting of the objects $\{ M^{a, R_a^+} \, | \, a \in \Q_1 \}$ is the category $mf(\Q)$ in Theorem \ref{dimer mirror sym}. Specifically, the matrix factorization $M^{a, R_a^+}$ corresponds to the Lagrangian $a \in \Q_1^\vee$ of the mirror punctured Riemann surface.
\end{exmp}

\subsection{Hochschild cohomology} \label{section Hochschild cohomology}

Let $(A, W)$ be a $\Gamma$-graded curved algebra. We define the two kinds of Hochschild (co)homology of $(A, W)$ and compare them. 

For each $n \geq 1$, the tensor product $A^{\otimes n}$ has the induced $\Gamma$-grading. We denote the homogeneous degree $m$ component in this grading by $(A^{\otimes n})_m$. The product is also considered to be in tensor degree $n$, modulo $2$ if $\Gamma = \Z$. The \textbf{Hochschild chain complex} of $(A, W)$, denoted $C_*(A)$, is the $\Gamma$-graded complex whose homogeneous degree $-k$ component is the direct sum totalization
\[
C_{-k}(A) : = \bigoplus_{m + n = k} (A^{\otimes n})_m
\]  
with differential given by the sum of two terms:
\begin{eqnarray} \label{chain}
d_A (a_0 \otimes \dots \otimes a_n) & = & \sum_{i = 0}^{n-1} (-1)^{i} a_0 \otimes \dots \otimes a_i a_{i+1} \otimes \dots \otimes a_n \\
& & + \, (-1)^{n + |a_n|( |a_0| + \dots + | a_{n-1}|) } a_n a_0 \otimes a_1 \otimes \dots \otimes a_{n-1} \nonumber \\
d_{W}(a_0 \otimes \dots \otimes a_n) & = & \sum_{i = 0}^n (-1)^{i} a_0 \otimes \dots \otimes a_{i} \otimes W \otimes a_{i+1} \otimes \dots \otimes a_n \nonumber
\end{eqnarray}
where $a_i \in A$ for all $i$. It is straightforward to check that $d_A^2 = d_W^2 = 0$ and $d_A d_W = - d_W d_A$. Alternatively, the \textbf{Borel-Moore Hochschild chain complex}, written $C_*^{BM}(A)$, is the $\Gamma$-graded complex whose homogeneous degree $-k$ component is the direct product totalization
\[
C^{BM}_{-k}(A) : = \prod_{m + n = k} (A^{\otimes n})_m
\]
with differential given as above. The cohomologies of these complexes are respectively the \textbf{Hochschild homology} of $A$, denoted $HH_*(A)$, and the \textbf{Borel-Moore Hochschild homology} of $A$, denoted $HH^{BM}_*(A)$.

Dually, for each $n \geq 0$, we have the the internal $\Hom$ of $\Gamma$-graded vector spaces
\[
\underline{\Hom} (A^{\otimes n},  A).
\]
It is also considered to be in tensor degree $n$, modulo $2$ if $\Gamma = \Z$. Let 
\[
\underline{\Hom}^m (A^{\otimes n},  A)
\] 
be the homogeneous internal degree $m$ component. We define the \textbf{Hochschild cochain complex} $C^*(A)$ as the $\Gamma$-graded complex whose homogeneous degree $k$ component is the direct product totalization
\[
C^k(A) : = \prod_{m + n = k} \underline{\Hom}^m (A^{\otimes n},  A)
\]
with differential given by the sum of two terms:
\begin{eqnarray} \label{cochain}
d_A f( a_1 \otimes \dots \otimes a_{n+1}) & = & (-1)^{|f| |a_1|} a_1 f( a_2 \otimes \dots \otimes a_{n+1}) \\
& & \, + \sum_{i = 1}^n (-1)^i f(a_1 \otimes \dots \otimes a_i a_{i+1} \otimes \dots \otimes a_{n+1} ) \nonumber \\
& & \, +(-1)^{n+1} f(a_1 \otimes \dots \otimes a_n) a_{n+1} \nonumber \\
d_W f(a_1 \otimes \dots \otimes a_{n-1})  & = & \sum_{i =0 }^{n-1} (-1)^{i+1} f(a_1 \otimes \dots \otimes a_i \otimes W \otimes a_{i+1} \otimes \dots \otimes a_{n-1}) \nonumber
\end{eqnarray}
where $f \in C^*(A)$ and $a_i \in A$ for all $i$. Alternatively, the \textbf{compactly supported Hochschild cochain complex}, written $C^*_c(A)$, is the $\Gamma$-graded complex whose homogeneous degree $k$ component is the direct sum totalization
\[
C^k_c(A) : = \bigoplus_{m + n = k} \underline{\Hom}^m (A^{\otimes n},  A)
\]
with differential given by the sum of the same two terms. The cohomologies of these complexes are respectively the \textbf{Hochschild cohomology} of $A$, denoted $HH^*(A)$, and the \textbf{compactly supported Hochschild cohomology} of $A$, denoted $HH^*_c(A)$.

If the curvature $W$ is trivial, $\Gamma = \z$, and $A$ is concentrated in degree 0, then the usual definition of Hochschild (co)homology of an associative algebra is recovered,
\[
HH_*(A) = \tor^{A^e}_*(A, A), \; \; HH^*(A) = \ext^*_{A^e}(A, A).
\]
As the grading on the complexes coincides with the grading by tensor powers, the direct product and sum totalizations actually agree, so there is no distinction between the two kinds of cohomology and homology.

In contrast, for a $\Z$-graded Landau--Ginzburg model $(A, W)$, the two kinds of Hochschild (co)homology are different. The direct sum and direct product totalizations are no longer equal, since each parity has infinitely many tensor components. In fact, the ordinary Hochschild invariants of $(A, W)$ are trivial when $W \neq 0$ (\cite{CaldararuTu} Theorem 4.2), so only the compactly suppported Hochschild cohomology and Borel--Moore homology can possibly provide useful information about the Landau--Ginzburg model.

\begin{exmp}[\cite{CaldararuTu}] \label{Example compact HH}
Let $A$ be the coordinate ring of a smooth affine variety of dimension $n$, considered to be in even degree. Let $W$ be a regular function with isolated singularity. Then there are isomorphisms of $\Z$-graded vector spaces
\[
HH^*_c(A) \cong \frac{A}{(\partial_i W \, | \, 1 \leq i \leq n)}, \; \; \; HH_*^{BM}(A) \cong \frac{A}{(\partial_i W \, | \, 1 \leq i \leq n)}[n]
\]
where $\partial_i W$ denotes the $i^{th}$ coordinate partial derivative and $[n]$ denotes the parity shift by $n$ mod $2$.
\end{exmp}


Compact-type Hochschild (co)homology can be defined for differential graded categories just as for algebras: essentially, the direct sum and product totalizations in the usual definitions are interchanged. Derived Morita theory informs us that the Hochschild cohomology of an algebra is equivalent to that of its differential graded category of modules \cite{Toen}. There is, in fact, a curved analog of the result, relating the Landau--Ginzburg model $(A, W)$ to its category of matrix factorizations.

\begin{thm}[\cite{PolPos} \S 2.6] \label{compact type}
For a $\Z$-graded Landau-Ginzburg model $(A, W)$, there are natural isomorphisms
\begin{eqnarray*}
HH^*_c(A) & \cong & HH^*_c \big( MF(A, W) \big) \\
HH_*^{BM} (A) & \cong & HH_*^{BM} \big(MF(A, W) \big).
\end{eqnarray*}
\end{thm}

The inclusion of direct sum into direct product totalizations induces graded maps
\begin{equation*} \label{comparison maps}
HH_* \big( MF(A,W) \big) \rightarrow HH^{BM}_* \big( MF(A,W) \big), \; \; HH^*_c \big( MF(A,W) \big) \rightarrow HH^* \big( MF(A,W) \big).   
\end{equation*}
They are not isomorphisms in general (\cite{PolPos} \S4.9) but are so in the case of Example \ref{Example compact HH} due to \cite{Dyckerhoff2011}. The argument there and in related works (e.g. \cite{LinPomerleano, Ballard}) revolves around the existence of a generator of the matrix factorization category. We expect that $mf(\Q)$ of Theorem \ref{dimer mirror sym} generates the matrix factorization category for $(\J, \ell)$ and so conjecture the equivalence of the two kinds of Hochschild (co)homology.  

\subsection{Noncommutative calculus} \label{section Noncommutative calculus}

Let $\Gamma = \z$ and $A$ be an associative algebra in degree $0$. The Hochschild homology and cohomology of $A$ form a noncommuative calculus
\[
\big( HH^*(A), \, \cup, \, \{ -, -\}, \, HH_*(A), \, \cap, \, B \big),  
\] 
in which $HH_*(A)$ is a module over the Gerstenhaber algebra $HH^*(A)$ \cite{TTnc}. We review the operations in this structure.

For a projective $A$-bimodule resolution $P_*$ of $A$, there is a diagonal map
\[
D: P_* \rightarrow P_* \otimes_A P_*,
\]
lifting the identity of $A$, unique up to homotopy equivalence, The \textbf{cup product} is the associative multiplication defined by
\[
f \cup g := \mu \circ (f \otimes g) \circ D \; \; \forall\, f, g \in \Hom_{A^e}(P_*, A)
\]
where $\mu$ is the multiplication on $A$. This operation descends to a graded commutative product on $HH^*(A)$. The \textbf{cap product} is defined by
\[
f \cap \eta : = (\mu \otimes Id_P) \circ (Id_A \otimes f \otimes Id_P) \circ (Id_A \otimes D) \, \eta \; \; \forall \, f \in \Hom_{A^e}(P_*, A), \; \eta \in A \otimes_{A^e} P_*,
\]
making $HH_*(A)$ into a left module over $HH^*(A)$ \cite{AD}. If $P_* = \B(A)$, the bar resolution of $A$, then the diagonal map is
\begin{equation} \label{bar diagonal}
D : a_1 \otimes \dots \otimes a_n \mapsto \sum_{i = 0}^n ( a_1 \otimes \dots \otimes a_i \otimes 1 ) \otimes (1 \otimes a_{i+1} \otimes \dots \otimes a_n ),
\end{equation}
and we obtain formulas for the cap and cup products on the Hochschild complexes $C^*(A)$and $C_*(A)$.

The \textbf{Gerstenhaber bracket} $\{-, -\} : C^*(A) \otimes C^*(A) \rightarrow C^*(A)$ is the Lie bracket of degree $-1$ defined by
\begin{align} \label{G bracket}
\{ f, g \}(a_1, \dots, a_{d+e-1}) & = \sum_{j \geq 0} (-1)^{j (|g| + 1)} f(a_1, \dots, a_j, g(a_{j+1}, \dots, a_{j+e}), \dots, a_{d+e -1} ) \\
& - (-1)^{(|f| + 1)(|g| + 1)} \sum_{j \geq 0} (-1)^{j (| f| + 1)} g(a_1, \dots, a_j, f(a_{j+1}, \dots, a_{j+d}), \dots, a_{d+e-1} ). \nonumber 
\end{align}
The cup product and Gerstenhaber bracket make $HH^*(A)$ into a Gerstenhaber algebra \cite{Gers}. In particular, the Leibniz identity is satisfied,
\begin{equation} \label{Leibniz}
\{ f, g \cup h \} = \{ f, g\} \cup h + (-1)^{(|f| - 1)|g|} g \{ f, h\} \; \; \; \forall f, g, h \in HH^*(A). 
\end{equation}
Note that, if $W \in A$ is a central element, the summand $d_W$ of the differential \eqref{cochain} corresponds to the adjoint action of $W$,
\[
d_W( f ) = - \{ f,  W \}.
\]

The \textbf{Connes differential} $B: C_*(A) \rightarrow C_*(A)$ is the square-zero map of degree $+1$ given by
\begin{align} \label{Connes}
B(a_0 \otimes a_1 \otimes \dots \otimes a_n) & = \sum_{i = 0}^{n} (-1)^{ni} 1 \otimes a_i \otimes  \dots \otimes a_n \otimes a_0 \otimes \dots \otimes a_{i-1} \\
& + \sum_{i = 0}^{n} (-1)^{n(i+1)} a_{i-1} \otimes 1 \otimes a_i \otimes \dots \otimes a_n \otimes a_0 \otimes \dots \otimes a_{i-2}. \nonumber
\end{align}
On Hochschild homology, the commutator of $B$ and the cap product define the \textbf{Lie derivative},
\begin{equation} \label{Cartan identity}
\Lie_f : = [B, f \cap -] \; \; \; \forall f \in HH^*(A),
\end{equation}
which makes $HH_*(A)$ into a module over $HH^*(A)$ as a Lie algebra. For a central element $W \in A$, $\Lie_W$ coincides with the differential $d_W$ on $C_*(A)$ in \eqref{chain}.

\subsection{{C}alabi--{Y}au algebras and Batalin-Vilkovisky structure} \label{section Calabi-Yau algebras}

Ginzburg \cite{Ginzburg} introduces the notion of a Calabi--Yau algebra to capture certain algebraic structures associated to Calabi-Yau manifolds. Let $A$ be an associative algebra, and let $A^e$ denote the enveloping algebra. 

\begin{defn}
An algebra $A$ is \textbf{homologically smooth} if it is a perfect $A^e$-module: that is, if it has a bounded resolution by finitely generated projective $A^e$-modules. Such an algebra is \textbf{Calabi--Yau of dimension $n$} (CY-$n$) if there exists an $A$-bimodule quasi-isomorphism 
\[ 
A[n] \simeq \RHom_{A^e} (A, A \otimes A)
\]
where $[-]$ denotes the shift in cohomological degree and $\RHom_{A^e} (A, A \otimes A)$ has $A$-bimodule action induced by the inner bimodule action on $A \otimes A$.
\end{defn}

\begin{thm}[\cite{Davison}] \label{consistency thm}
Suppose $\Q$ is a dimer embedded in a surface of positive genus. If $\Q$ is zigzag consistent, then $\J$ is Calabi-Yau of dimension $3$.
\end{thm}

The quasi-isomorphism in the definition is equivalently an $A$-bimodule isomorphism
\begin{equation} \label{CY volume}
A \cong \ext^n_{A^e}(A, A \otimes A),
\end{equation}
which is determined by the image of $1 \in A$. This element, which must be central in $\ext^n_{A^e}(A, A \otimes A)$ with respect to the bimodule action, is called a \textbf{volume} of $A$. The set of all volumes is a torsor over the ring of central units of $A$. Since $A$ is homologically smooth, there is a quasi-isomorphism
\begin{equation} \label{smooth qiso}
\RHom_{A^e} ( \RHom_{A^e}(A, A \otimes A), A) \sim A \otimes^{\mathbb{L}}_{A^e} A,
\end{equation}
under which a volume corresponds to a class in $HH_n(A)$, called a \textbf{nondegenerate element} \cite{VDBV}. We label a generic volume by $\pi$ and, by abuse of notation, use the same for a nondegenerate element. 

Precomposing \eqref{smooth qiso} with the quasi-isomorphism from a volume yields
\[
\begin{tikzcd} \label{VDB qisomorphism}
\RHom_{A^e}(A[n], A) \arrow[r, "\sim"] & A \otimes_{A^e}^{\mathbb{L}} A,
\end{tikzcd}
\]
inducing the isomorphism on homology \cite{VDBV}
\begin{equation} \label{VDB cap}
\D_{\pi}: HH^*(A) \rightarrow HH_{n-*}(A), \; \; f \mapsto f \cap \pi.
\end{equation}
Such a Poincar\'e-type duality pairing between Hochschild cohomology and homology was first observed in \cite{VdB}. The isomorphism $\D_{\pi}$ exchanges $f  \, \cup -$ and $f \, \cap \, \D_{\pi}( - )$. Moreover, the Connes differential $B$ is sent under $\D_{\pi}$ to a differential $\Delta_{\pi}$ called the \textbf{Batalin-Vilkovisky (BV) operator}. The failure of the BV operator to be a derivation with respect to the cup product is measured by the Gerstenhaber bracket:
\begin{equation} \label{BVGers}
\Delta_{\pi}(f \cup g) = \Delta_{\pi}(f) \cup g + (-1)^{|f|} f \Delta_{\pi}(g) + (-1)^{|f|} \{f, g \} \; \; \; \forall \, f, g \in HH^*(A).
\end{equation}
The triple $\big( HH^*(A), \cup, \Delta \big)$ is known as a \textbf{Batalin-Vilkovisky algebra}.

\section{Hochschild cohomology of a central localization}

We relate the Hochschild (co)homology of an algebra to that of a central localization, building on the classical result in \cite{Brylinski}. In particular, we show that there is a morphism of noncommutative calculi, including the BV structure when the algebra is Calabi--Yau. Then in the following sections, this fact is used to compare the Hochschild cohomology of the Jacobi algebra to that of $\JW$. The algebra $\JW$ is Morita equivalent to the fundamental group algebra of $\Sigma \times S^1$, whose Hochschild BV structure is well known. This will help us to compute the BV structure of $HH^*(\J)$ in \S \ref{section Batalin--Vilkovisky structure}.

Throughout, $\Q$ is assumed to be a dimer model admitting a perfect matching.

\subsection{Hochschild cohomology of a central localization} \label{section central localization}

Let $A$ be an associative algebra, $\zc$ be the center of $A$, and $S \subset \zc$ a multiplicative subset containing $1$ and excluding $0$. We denote by $\widehat{\zc}$ the localization of the center with respect to $S$. Then the \textbf{Ore localization} of $A$ with respect to $S$ can be defined as the algebra
\[
\hA : = A \otimes_{\zc} \widehat{\zc}.
\]
Moreover, for any $A$-module $M$, its Ore localization is the $\hA$-module
\[
\widehat{M} : = M \otimes_{\zc} \widehat{\zc}.
\]
The universal map $L: M \rightarrow \widehat{M}$ has kernel equal to the $S$-torsion of $M$,
\[
tor_S(M) : = \{ m \in M \, | \, sm = 0 \; \text{for some} \; s \in S \}.
\]  
See for example \S 9-10 of \cite{Lam} for a detailed account about localization for noncommutative rings.

Let $\mathcal{D}(A^e)$ and $\mathcal{D}(\hA^e)$ be the derived categories of $A$ and $\hA$-bimodules, respectively. The algebra $\hA$ is flat as a left and as a right $A$-module, and there is an adjunction
\[
\begin{tikzcd}
\mathcal{D}(A^e) \arrow[r, shift left, "F"] & \mathcal{D}(\hA^e) \arrow[l, shift left, "G"]
\end{tikzcd}
\]
where $F : = \hA \otimes_A - \otimes_A \hA$ and $G$ is the restriction functor. They yield canonical maps
\begin{align*}
I_* &:  A \otimes_{A^e}^{\mathbb{L}} A \rightarrow A \otimes_{A^e}^{\mathbb{L}} GF(A) \cong \hA \otimes_{\hA^e}^{\mathbb{L}} \hA, \\
I^* &: \RHom_{A^e}(A, A) \rightarrow \RHom_{\hA^e}(F(A), F(A)) \cong \RHom_{\hA^e}(\hA, \hA).
\end{align*}
Letting 
\[ 
\hB(A) : = F(\B(A)) = \hA \otimes_{A} \B(A) \otimes_{A} \hA \cong \bigoplus_{n \in \mathbb{N}} \hA \otimes A^{\otimes n} \otimes \hA,
\] 
we can write the maps explicitly on (co)chains:
\begin{align*}
 I_* & : A \otimes_{A^e} \B(A) \rightarrow \hA \otimes_{\hA^e} \hB(A), \\
 & I_*(a_0 \otimes a_1 \otimes \dots \otimes a_n) = L(a_0) \otimes L(a_1) \otimes a_2 \dots \otimes a_{n-1} \otimes L(a_n); \\
I^* &: \Hom_{A^e}(\B(A), A) \rightarrow \Hom_{\hA^e}(\hB(A), \hA), \\
& I^*(f)(\hat{a}_1 \otimes a_2 \otimes \dots \otimes a_{n-1} \otimes \hat{a}_n) = \hat{a}_1 L \, f(1 \otimes a_2 \otimes \dots \otimes a_{n-1} \otimes 1 ) \hat{a}_n.
\end{align*}
The comparison map
\[
\hB(A) \rightarrow \B(\hA), \; \; \; \hat{a}_1 \otimes a_2 \otimes \dots \otimes a_{n-1} \otimes \hat{a}_n \mapsto \hat{a}_1 \otimes L(a_2) \otimes \dots \otimes L(a_{n-1}) \otimes \hat{a}_n.
\]
lifts the identity of $\hA$ and so is a homotopy equivalence. Writing
\[
\phi: \hA \otimes_{\hA^e} \hB(A) \rightarrow \hA \otimes_{\hA^e} \B(\hA)
\]
for the induced map on chains and
\[
\phi^\vee: \Hom_{\hA^e}(\hB(A), \hA) \rightarrow \Hom_{\hA^e}(\B(\hA), \hA)
\] 
for the map on cochains given by precomposition with the homotopy inverse, define
\begin{align*}
L_* & : = \phi \circ I_*:  C_*(A) \rightarrow C_*(\hA), \\
L_* & : = \phi^\vee \circ I^*: C^*(A) \rightarrow C^*(\hA).
\end{align*}
In particular, $L_*$ is simply the image of $L$ under the on Hochschild homology functor,
\begin{equation} \label{Lochom}
L_*: a_0 \otimes a_1 \otimes a_2 \otimes \dots \otimes a_n \mapsto L(a_0) \otimes L(a_1) \otimes L(a_2) \otimes \dots \otimes L(a_n). 
\end{equation}

Brylinski \cite{Brylinski} shows that the Hochschild homology functor commutes with central localization. The result can be slightly enhanced to relate the Connes differentials on $A$ and $\hA$, which we denote as $B_A$ and $B_{\hA}$.

\begin{prop} \label{loc of homology}
Let $A$ be an associative algebra and $S \subset \zc$ a multiplicative subset. The map
\[
\widehat{L}_*: \widehat{HH}_*(A) = HH_*(A) \otimes_{\zc} \widehat{\zc} \longrightarrow HH_*(\hA) \; \; \sigma \otimes s^{-1} \mapsto s^{-1} \cap L_*(\sigma)
\]
is an isomorphism of $\widehat{\zc}$-modules. Moreover, $\widehat{L}_*$ intertwines the Connes differential $B_{\hA}$ with the differential defined by
\[
\hBc(\sigma \otimes s^{-1}) : = B_A(\sigma) \otimes s^{-1} - \Lie_{s} (\sigma) \otimes s^{-2} , \; \; \forall s \in S.
\]
\end{prop}

\begin{proof}
The first assertion is proved in \cite{Brylinski}. It is clear from the formulas for the Connes differential \eqref{Connes} and $L_*$ \eqref{Lochom} that $L_*$ intertwines $B_A$ and $B_{\hA}$. Then the formula for the differential $\hBc$ is obtained from the calculus identities (\S \ref{section Noncommutative calculus}). Observe
\begin{align*}
B_{\hA} \, \hL_*(\sigma \otimes s^{-1}) & = B_{\hA} \big(s^{-1} \cap L_*(\sigma) \big) \\
& = s^{-1} \cap B_{\hA} \, L_*(\sigma) + \Lie_{s^{-1}} \, L_*(\sigma) \\
& = s^{-1} \cap L_* \, B_A (\sigma) +  \Lie_{s^{-1}} \, L_*(\sigma).
\end{align*}
There is also the identity $\Lie_{s^{-1}} = -s^{-2} \cap \Lie_s$, so the last expression equals
\[
s^{-1} \cap L_* \, B_A(\sigma) - s^{-2} \cap \Lie_s \, L_*(\sigma) = s^{-1} \cap L_* \, B_A(\sigma) - s^{-2} \cap L_* \, \Lie_s(\sigma).
\]
Under the isomorphism $\hL$, this is precisely the image of
\[
B_A(\sigma) \otimes s^{-1} - \Lie_s(\sigma) \otimes s^{-2}. \qedhere
\]
\end{proof}

In order to extend the result to the full calculus structure, we use the fact that the cup and cap products can be defined on Hochschild (co)chains using the resolution $\hB(A)$. Let $D_A$ be the diagonal map for $\B(A)$ \eqref{bar diagonal}; up to homotopy equivalence, the diagonal map for $\hB(A)$ has the form
\begin{gather*}
\hD: \hB(A) \rightarrow \hB(A) \otimes_{\hA} \hB(A), \\
\hat{a}_1 \otimes a_2 \otimes \dots \otimes a_{n-1} \otimes \hat{a}_n \mapsto \sum_{i = 0}^n (\hat{a}_1  \otimes a_2 \otimes \dots \otimes a_i \otimes 1 ) \otimes (1 \otimes a_{i+1} \otimes \dots \otimes a_{n-1} \otimes \hat{a}_n).
\end{gather*} 

\begin{lem} \label{algebras}
The map $L^*: HH^*(A) \rightarrow HH^*(\hA)$ is a morphism of algebras with respect to cup product.
\end{lem}

\begin{proof}
Consider the diagram of complexes
\[
\xymatrix{
\Hom_{A^e}(\B(A), A) \ar[r]^-{I^*} \ar[d]^-{f \cup -} & \Hom_{\hA^e}(\hB(A), \hA) \ar[r]^-{\phi^\vee} \ar[d]^-{I^*(f) \cup -} & \Hom_{\hA^e}(\B(\hA), \hA) \ar[d]^-{L^* (f) \cup -} \\
\Hom_{A^e}(\B(A), A) \ar[r]^-{I^*} & \Hom_{\hA^e}(\hB(A), \hA) \ar[r]^-{\phi^\vee} & \Hom_{\hA^e}(\B(\hA), \hA).
}
\]
The horizontal composition is $L^*$. Upon taking cohomology, commutativity of the second square follows from the independence of the cup product from the choice of resolution. So to prove $L^*$ is an algebra morphism, it suffices to prove commutativity of the first square. Observe
\begin{align*}
I^*(f \cup g)(\hat{a}_1  \otimes \dots \otimes \hat{a}_n ) & = \hat{a}_1  L (f \cup g)(1 \otimes a_2 \otimes \dots \otimes a_{n-1} \otimes 1) \hat{a}_n  \\
& = \hat{a}_1  L \circ \mu( f \otimes g) D_A (1 \otimes a_2 \otimes \dots \otimes a_{n-1} \otimes 1) \hat{a}_n  \\
& = \hat{a}_1 L f(1 \otimes a_2 \otimes \dots \otimes a_i \otimes 1 ) L g (1 \otimes a_{i+1} \otimes \dots \otimes a_{n-1} \otimes 1) \hat{a}_n
\end{align*}
and
\begin{align*}
I^*(f) \cup I^*(g) (\hat{a}_1  \otimes \dots \otimes \hat{a}_n) & = \mu(I^*(f) \cup I^*(g) ) \hD (\hat{a}_1 \otimes \dots \otimes \hat{a}_1 ) \\
& = I^*(f)(\hat{a}_1 \otimes a_2 \otimes a_i \otimes 1) I^*(g)(1 \otimes a_{i+1} \otimes \dots \otimes a_{n-1} \otimes \hat{a}_n ) \\
& = \hat{a}_1 L  f(1 \otimes a_2 \otimes \dots \otimes a_i \otimes 1 ) L g (1 \otimes a_{i+1} \otimes \dots \otimes a_{n-1} \otimes 1) \hat{a}_n,
\end{align*}
as desired.
\end{proof}

\begin{lem} \label{cap com}
For all $f \in HH^*(A)$ and $\sigma \in HH_*(A)$, the diagram
\[
\xymatrix{
HH^*(A) \ar[r]^-{L^*} \ar[d]^-{- \cap \sigma} & HH^*(\hA) \ar[d]^-{-\cap L_*(\sigma)} \\
HH_*(A) \ar[r]^-{L_*} & HH_*(\hA)
}
\]
commutes.
\end{lem}

\begin{proof}
We use a similar argument as for the previous lemma. Consider the diagram of complexes
\[
\xymatrix{
\Hom_{A^e}(\B(A), A) \ar[r]^-{I^*} \ar[d]^-{- \cap \sigma} & \Hom_{\hA^e}(\hB(A), \hA) \ar[r]^-{\phi^\vee} \ar[d]^-{- \cap I_*(\sigma)} & \Hom_{\hA^e}(\B(\hA), \hA) \ar[d]^-{- \cap L_*(\sigma)} \\
A \otimes_{A^e} \B(A) \ar[r]^-{I_*} & \hA \otimes_{\hA^e} \hB(A) \ar[r]^-{\phi} & \hA \otimes_{\hA^e} \B(\hA)
}
\]
The top horizontal composition is the map $L^*$, while the bottom horizontal composition is the map $L_*$. Upon taking cohomology, commutativity of the second square follows from the independence of the cap product from choice of resolution. So to prove the result, it suffices to prove commutativity of the first square. 

Without loss of generality, suppose
\[ 
\sigma = a_0 \otimes a_1 \otimes \dots \otimes a_n \in A \otimes_{A^e} A^{\otimes n}.
\]
Observe
\begin{eqnarray*}
I_*( f \cap \sigma) & = & I_* \big( a_0 \otimes (f \otimes Id) D_A ( a_1 \otimes \dots \otimes a_n) \big) \\
& = & I_* \big( a_0 a_1 f(1 \otimes a_2 \otimes \dots \otimes a_i \otimes 1) \otimes (1 \otimes a_{i+1} \otimes \dots \otimes a_n) \big) \\
& = & L(a_0) L(a_1) L f(1 \otimes a_2 \otimes \dots \otimes a_i \otimes 1) \otimes (1 \otimes a_{i+1} \otimes \dots \otimes a_{n-1} \otimes L(a_n)),
\end{eqnarray*}
and
\begin{eqnarray*}
I^*(f) \cap I_*(\sigma) & = & L(a_0) \otimes (I^*(f) \otimes Id)\hD( L(a_1) \otimes a_2 \otimes \dots \otimes a_{n-1} \otimes L(a_n)) \\ 
& = & L(a_0) I^*(f)(L(a_1) \otimes a_2 \otimes \dots \otimes a_i \otimes 1) \otimes (1 \otimes a_{i+1} \otimes \dots \otimes a_{n-1} \otimes L(a_n)) \\
& = & L(a_0) L(a_1) L f(1 \otimes a_2 \otimes \dots \otimes a_i \otimes 1) \otimes (1 \otimes a_{i+1} \otimes  \dots \otimes a_{n-1} \otimes L(a_n)),
\end{eqnarray*}
as desired.
\end{proof}

If $A$ is Calabi-Yau of dimension $n$, then its central localization $\hA$ is as well \cite{Farinati}. Since Van den Berg duality takes the form of capping with a nondegenerate element \cite{VDBV}, Proposition \ref{loc of homology} and Lemma \ref{cap com} can be used to prove $L^*$ intertwines the BV operators. 

\begin{prop} \label{L^* BV}
Suppose $A$ is CY-$n$ and $\pi$ is a nondegenerate element. Then $\hA$ is also CY-$n$, and the element $L_*(\pi)$ is a nondegenerate element for $\hA$. Moreover, the map $L^*$ is a morphism of BV-algebras when the Hochschild cohomologies of $A$ and $\hA$ are equipped with the operators $\Delta_{\pi}$ and $\Delta_{L_*(\pi)}$, respectively.
\end{prop}

\begin{proof}
For $M$ a perfect $A^e$-module, consider the commutative diagram
\[
\begin{tikzcd}
\RHom_{A^e}( \RHom_{A^e}(M, A \otimes A), A) \arrow[rr, "\simeq"] \arrow[d, "F"] & & M \otimes^{\mathbb{L}}_{A^e} A \arrow[d] \\
\RHom_{\hA^e}(F \RHom_{A^e}(M, A \otimes A), F(A)) \arrow[r, "\simeq"] \arrow[d, "\simeq"] & \RHom_{A^e}(\RHom_{A^e}(M, A \otimes A), GF(A)) \arrow[r, "\simeq"] & M \otimes^{\mathbb{L}}_{A^e}  GF(A)  \arrow[d, "\simeq"] \\
\RHom_{\hA^e} ( \RHom_{\hA^e} (F(M), \hA \otimes \hA), \hA) \arrow[rr, "\simeq"] & & F(M) \otimes^{\mathbb{L}}_{\hA^e} \hA.
\end{tikzcd}
\]
The diagram can be easily checked for the free bimodule $M = A^e$, which is enough to show its validity for a general perfect module. Under the vertical composition on the left side, a quasi-isomorphism between $\RHom_{A^e}(M, A \otimes A)$ and $A$ is mapped to one between $\RHom_{\hA^e}(F(M), \hA \otimes \hA)$ and $\hA$. Since $A$ is homologically smooth, we can take $M = A$, confirming that $\hA$ is Calabi--Yau of dimension $n$. The vertical composition on the right hand side is $I_*$. Hence, if $\pi$ is a nondegenerate element for $A$, then $L_*(\pi)$ is a nondegenerate element for $\hA$. 

By Lemma \ref{cap com}, we have the commuting square 
\[
\xymatrix{
HH^*(A) \ar[r]^-{L^*} \ar[d]^-{\D_{\pi}} & HH^*(\hA) \ar[d]^-{\D_{L_*(\pi)}} \\
HH_*(A) \ar[r]^-{L_*} & HH_*(\hA).
}
\]
Proposition \ref{loc of homology} states that $L_*$ intertwines the Connes differentials, so by definition of the BV operator, $L^*$ must intertwine $\Delta_{\pi}$ and $\Delta_{L_*(\pi)}$.  
\end{proof}

Now the dual to Proposition \ref{loc of homology}, stating roughly that Hochschild cohomology ``commutes" with central localization, can be proven.

\begin{thm} \label{loc of cohomology}
Let $A$ be an associative algebra and $S \subset \zc$ a multiplicative subset.
\begin{enumerate} 
\item The map
\[
\hL^*: \widehat{HH^*}(A) = HH^*(A) \otimes_{\zc} \widehat{\zc} \longrightarrow HH^*(\hA), \; \; f \otimes s^{-1} \mapsto L^*(f) \cup s^{-1} \; \; \forall \, s \in S
\]
is a morphism of graded $\widehat{\zc}$-algebras.
\item If $A$ is homologically smooth, then $\hL^*$ is an isomorphism.
\item Suppose $A$ is CY-$n$, and let $\pi$ be a nondegenerate element. The map $\hL^*$ is an isomorphism of BV algebras, intertwining $\Delta_{L_*(\pi)}$ with the differential defined by
\[
\widehat{\Delta}_{\pi}(f \otimes s^{-1}) : = \Delta_{\pi}(f) \otimes s^{-1} - \{s, f\} \otimes s^{-2}, \; \; \forall s \in S.
\]
\end{enumerate}
\end{thm}

\begin{proof} $ $ \\
(1) This is clear from Lemma \ref{algebras}.
$ $ \newline
$ $ \newline
(2) For $M$ a perfect $A^e$-module, consider the commutative diagram
\[
\begin{tikzcd}
\RHom_{A^e}(M, A) \arrow[r, "I^*"] \arrow[rd] & \RHom_{\hA^e}(F(M), \hA) \\
& \RHom_{A^e}(M, A) \otimes_{\zc} \hat{\zc} \arrow[u]
\end{tikzcd}
\]
where the diagonal arrow is the universal map and the vertical arrow is induced by the universal property. The latter is easily seen to be a quasi-isomorphism when $M = A^e$, so the same is true for general perfect bimodules. 

$ $\newline
$ $\newline
(3) From Proposition \ref{L^* BV} and the preceding, we have a commutative diagram of isomorphisms
\[
\xymatrix{
HH^*(A) \otimes_{\zc} \widehat{\zc} \ar[r]^-{\hL^*} \ar[d]^-{\D_{\pi} \otimes Id} & HH^*(\hA) \ar[d]^-{\D_{L_*(\pi)}} \\
HH_*(A) \otimes_{\zc} \widehat{\zc} \ar[r]^-{\hL_*} & HH_*(\hA).
}
\]
By Proposition \ref{loc of homology}, $\hL_*$ is a chain map where the left side is given differential 
\[ 
\hBc(\sigma \otimes s^{-1}) = B(\sigma) \otimes s^{-1} - \Lie_{s} (\sigma) \otimes s^{-2} \; \; \forall s \in S
\]
Under $\D_{\pi} \otimes Id$, $\hBc$ is sent to the differential $\widehat{\Delta}_{\pi}$ with the stated formula.
\end{proof}

\begin{cor} \label{kernel tor}
Let $A$ be an associative algebra and $S \subset \zc$ a multiplicative subset. Then $HH_*(\hA)$ is the localization of $HH_*(A)$ with respect to $S$, and 
\[
Ker(L_*) = tor_S(HH_*(A)).
\] 
If furthermore $A$ is homologically smooth, then $HH^*(\hA)$ is the localization of $HH^*(A)$ with respect to $S$, and
\[
Ker(L^*) = tor_S(HH^*(A)).
\]
\end{cor} 

In general, even if $A$ is torsion free (and hence $L: A \rightarrow \hA$ is injective), there may be torsion in the Hochschild (co)homology of $A$. This fact is witnessed by the Hochschild cohomology of the Jacobi algebra.

\subsection{Morita invariance} \label{section Morita invariance}

It is well known that Morita equivalence induces isomorphisms on Hochschild (co)homology. For $A$ an associative algebra and $Mat_r(A)$ the algebra of $r \times r$-matrices with coefficients in $A$, the standard isomorphisms are 
\[
\begin{tikzcd}[column sep = 3em]
HH_*(A) \arrow[r, shift left, "inc_*"] & HH_*(Mat_r(A)) \arrow[l, shift left, "tr_*"],
\end{tikzcd}
\qquad
\begin{tikzcd}[column sep = 3em]
HH^*(A) \arrow[r, shift left, "cotr^*"] & HH_*(Mat_r(A)) \arrow[l, shift left, "inc^*"]
\end{tikzcd}
\] 
whose formulas are given in \S $1.2$ and \S $1.5$ of \cite{Loday}. In particular, $tr_*$ is the generalized trace, 
\begin{equation} \label{trace}
tr_*: m_0 \otimes m_1 \otimes \dots \otimes m_n \mapsto \sum_{(i_0, \dots, i_n)} m_0^{i_0, i_1} \otimes m_1^{i_1, i_2} \otimes \dots \otimes m_n^{i_n, i_0}, \; \; \; m_k = (m_k^{i, j})_{1 \leq i, j \leq r} \in Mat_r(A),
\end{equation}
the sum being over all indices $(i_0, \dots, i_n) \in \{1, 2, \dots, r \}^{n+1}$, and $inc^*$ is the co-inclusion,
\begin{equation} \label{co-inclusion}
inc^*(\alpha)(a_0 \otimes a_1 \otimes \dots \otimes a_n) = proj_{11} \, \alpha(a_0 e_{11} \otimes a_1 e_{11} \otimes \dots \otimes a_n e_{11}),
\end{equation}
the map $proj_{11}$ being the projection onto the $(1, 1)$-coordinate. If $A$ is Calabi-Yau of dimension $n$, then so is $Mat_r(A)$ \cite{Zeng}, and the analogous statement to Theorem \ref{loc of cohomology} holds.

\begin{prop} \label{morita invariance}
Suppose $A$ is CY-$n$, and let $\pi$ be a nondegenerate element. The map $inc^*$ is an isomorphism of BV algebras when the Hochschild cohomologies of $A$ and $Mat_r(A)$ are equipped with the operators $\Delta_\pi$ and $\Delta_{inc_*(\pi)}$, respectively.
\end{prop}

\begin{proof}
It is known from general Morita theory that the isomorphisms on Hochschild (co)homology preserve the cup and cap products \cite{Armenta}. Hence, following the logic of Proposition \ref{L^* BV}, we have only to show that either $tr_*$ or $inc_*$ intertwines the Connes differentials. But this is clear from the formulas for the Connes differential \eqref{Connes} and $tr_*$ \eqref{trace}.
\end{proof}

\subsection{Batalin-Vilkovisky structure of $HH^*(\JW)$} \label{section BV localized Jacobi}

Write $\Psi^*$ and $\Psi_*$ for the isomorphisms on Hochschild (co)homology induced from $\Psi$ in Theorem \ref{matrix algebra}. Composition with the Morita isomorphisms in \S \ref{section Morita invariance} yields
\[
\begin{tikzcd}
cotr^* \circ \Psi^*: HH^*(\JW) \arrow[r, "\cong"] & HH^*(\C[\pi_1(\Sigma, v_o)] \otimes \C[z^{\pm 1}]) \\
tr_* \circ \Psi_*: HH_*(\JW) \arrow[r, "\cong"] & HH_*(\C[\pi_1(\Sigma, v_o)] \otimes \C[z^{\pm 1}]).
\end{tikzcd}
\]

If $\Sigma$ is a torus, then the matrix coefficients are Laurent polynomials in three variables,
\[
\C[\pi_1(\Sigma, v_o)] \otimes \C[z^{\pm 1}] \cong \C[ x^{\pm 1}, y^{\pm 1}, z^{\pm 1} ]  
\]
where $x$ and $y$ form a basis of the fundamental group. The Hochschild-Kostant-Rosenberg isomorphism \cite{HKR} identifies the Hochschild (co)homology with polyvectors and differential forms on the algebraic $3$-torus,
\begin{align}
& \C[x^{ \pm 1}, y^{\pm 1}, z^{\pm 1}, \partial_x, \partial_y, \partial_z], \\
& \C[x^{ \pm 1}, y^{\pm 1}, z^{\pm 1}, dx, dy, dz]. \nonumber
\end{align}
In particular, the BV differential is the divergence operator associated to a choice of $3$-form. Explicitly, if $\xi_x$, $\xi_y$, and $\xi_z$ are the coordinate vector fields $\partial_x$, $\partial_y$, and $\partial_z$, the diveregence operator associated to $\pi : = x^r y^s z^t dx \, dy \, dz$ is
\[
\Div_{\pi} : = (r x^{-1} + \partial_x) \xi_x + (s y^{-1} + \partial_y) \xi_y + (t z^{-1} + \partial_z) \xi_z.
\]

If $\Sigma$ has genus $g > 1$, the algebra $\C[\pi_1(\Sigma, v_o)]$ is Calabi-Yau of dimension 2 (see \cite{Ginzburg} Corollary 6.1.4), so its Hochschild cohomology has a BV structure from Van den Bergh duality. Vaintrob \cite{Vaintrob} establishes a BV isomorphism between $HH^*( \C[ \pi_1(\Sigma, v_o)] )$ and $\mathbb{H}_{2-*}(\Sigma)$, the homology of the free loop space of $\Sigma$. The BV structure of the latter, known as string topology \cite{ChasSullivan}, consists of the loop product and the operator $\rho: \mathbb{H}_{*}(\Sigma) \rightarrow \mathbb{H}_{*+1}(\Sigma)$ induced by the $S^1$-action on the free loop space. Because the center of $\C[\pi_1(\Sigma, v_0)]$ is simply $\C$, the volume \eqref{CY volume} for the Calabi-Yau structure is unique up to scaling, implying the BV operator is unique. Let $\pi_s \in HH_2(\C [\pi_1(\Sigma, v_o])$ be a nondegenerate element.

The Kunneth isomorphism relates the Hochschild invariants of $\C[\pi_1(\Sigma, v_0)] \otimes \C[z^{\pm1}]$ to those of its tensor factors. For homology, the isomorphism holds generally (\cite{Loday} Theorem 4.2.5), so a nondegenerate element $\pi$ for the product corresponds to $\pi_s \otimes z^t dz$ for some $t \in \z$. For cohomology, certain finiteness conditions must be satisfied for the isomorphism to hold and respect BV structures.

\begin{lem} \label{Kunneth}
Suppose $\Sigma$ is a Riemann surface of genus $g > 1$. There is an isomorphism of BV algebras
\[
\big( HH^*(\C[\pi_1(\Sigma), v_o] \otimes \C[z^{\pm1}]), \, \Delta_{\pi} \big) \cong \big( \mathbb{H}_*(\Sigma) \otimes \C[z^{\pm 1}][\partial_z], \, \rho \otimes id + id \otimes (t z^{-1} + \partial_z) \xi_z \big)
\] 
\end{lem}

\begin{proof}
Let $P_*$ be a finitely generated projective bimodule resolution of $A_1 := \C[\pi_1(\Sigma, v_o)]$. By Theorem 3.13 of \cite{AD}, we have only to show that there exists a bimodule resolution $Q_*$ of $A_2 : = \C[z^{\pm 1}]$ such that
\[
\Hom_{(A_1 \otimes A_2)^e} (P_* \otimes Q_*, A_1 \otimes A_2) \cong \Hom_{A_1^e}(P_*, A_1) \otimes \Hom_{A_2^e}(Q_*, A_2).
\] 
But this is clear if we choose $Q_*$ to be the Koszul bimodule resolution of $\C [z^{\pm 1}]$.
\end{proof}

\subsection{Batalin-Vilkovisky structure of $\widehat{HH^*}(\JW)$} \label{section BV localized Hochschild}

Let $\Q$ be a zigzag consistent dimer model, so $\J$ is Calabi--Yau of dimension $3$ (Theorem \ref{consistency thm}). The results of \S \ref{section central localization} can be used to relate the Hochschild cohomology of $\J$ to the Hochschild cohomology of $\JW$, explicitly described in the previous section. First, the volumes for the Calabi-Yau structure of $\J$ must be determined. 

The Calabi-Yau property of the Jacobi algebra is equivalent to the exactness of a certain free $\J$-bimodule complex (\cite{Ginzburg} Corollary 5.3.3), 
\[
\begin{tikzcd} \label{CY exact}
\mathbb{P}_3 \arrow[r, hookrightarrow, "\mu_3"] & \mathbb{P}_2 \arrow[r, "\mu_2"] & \mathbb{P}_1 \arrow[r, "\mu_1"] & \mathbb{P}_0 \arrow[r, twoheadrightarrow, "\mu_0"] & \J.
\end{tikzcd}
\]
The $\mathbb{P}_i$ have the form
\[
\mathbb{P}_i : = \J \otimes_{\bk} V_i \otimes_{\bk} \J
\]
where
\begin{itemize}
\item $V_0 = \bk$,
\item $V_1 = \C \Q_1$, the vector space spanned by the arrows of $\Q$,
\item $V_2 = \C \{ \partial_a \Phi_0 \, | \, a \in \Q_1 \}$, the vector space spanned by the cyclic derivatives of $\Phi_0$, and
\item $V_3 = \C \{ \Phi_0^v \, | \, v \in \Q_0 \} $, the vector space spanned by the syzygies
\[ 
\Phi_0^v : = \sum_{a \, | \, t(a) = v} a \, \partial_a \Phi_0 = \sum_{a \, | \, h(a) = v} \partial_a \Phi_0 \, a. 
\] 
\end{itemize}
In fact, $\mathbb{P}_*$ is a self-dual resolution of $\J$,
\[
\RHom_{\J^e}(\mathbb{P}_*, \J \otimes \J) \cong \mathbb{P}_{3-*}.
\]
It is moreover graded by the first homology of $\Sigma$ and the perfect matchings of $\Q$ (\S \ref{section The localization}). Hence, the Hochschild homology and cohomology of $\J$ inherit an $H_1(\Sigma, \z) \times \z^{PM(\Q)}$-grading.

\begin{lem} \label{degree VDB}
Suppose $\Q$ is a zigzag consistent dimer admitting a perfect matching.
\begin{enumerate}
\item Up to scaling, the unique volume of $\J$ is the class in $\ext^3_{\J^e}(\J, \J \otimes \J)$ of the map
\[
\pi_{0}: \mathbb{P}_3 \rightarrow \J \otimes \J, \; \;  p \otimes \Phi_0^v \otimes q \mapsto p v \otimes v q.
\]  
\item $\Delta_{\pi_0}$ is the unique BV differential induced from the Calabi-Yau structure of $\J$.
\item The Van den Bergh isomorphism 
\[
\D_{\pi_0}: HH^*(\J) \rightarrow HH_{3-*}(\J), \; \;  \alpha \mapsto \alpha \cap \pi_0
\]
is homogeneous of degree $0$ in homology and degree $1$ in all perfect matchings.
\end{enumerate}
\end{lem}

\begin{proof}
There are $\J$-bimodule isomorphisms
\[
\begin{tikzcd}
\ext^3_{\J^e}(\J, \J \otimes \J) \ar[r, "\cong"] & H_0(\mathbb{P}_*) \ar[r, "\overline{\mu}_0"] & \J \\
\end{tikzcd} 
\]
the first given by self-duality of $\mathbb{P}_*$  The pre-image $1 \in \J$ is precisely the class of $\pi_0$. Any other volume is in the $\zc^\times$-orbit of $\pi_0$. However, the only central units in $\J$ are the nonzero scalars, proving the first assertion. The second follows because volumes that differ by a scalar yield the same BV differential.

The volume $\pi_0$ has degree $0$ in $H_1(\Sigma, \z)$ and degree $-1$ in all perfect matchings. This implies that the quasi-isomorphism \eqref{VDB qisomorphism} that descends to $\D_{\pi_0}$ has degree $0$ in $H_1(\Sigma, \z)$ and degree $1$ in all perfect matchings.
\end{proof}

Now that the volume of $\J$ has been specified, grading considerations are enough to deduce the corresponding BV structure on the localization.

\begin{thm} \label{BV torus}
Suppose $\Q$ is a zigzag consistent dimer embedded in $\Sigma$ and admitting a perfect matching. 
\begin{enumerate}
\item If $\Sigma$ has genus $g = 1$, then there is an isomorphism of BV algebras 
\[
\big( \widehat{HH^*}(\J), \, \widehat{\Delta}_{\pi_0} \big) \cong \big(\C[x^{ \pm 1}, y^{\pm 1}, z^{\pm 1}, \partial_x, \partial_y, \partial_z], \, \Div_0 \big)
\]
where $\Div_0$ is the divergence operator
\[
\Div_0 : = (-x^{-1} + \partial_x) \xi_x + (-y^{-1} + \partial_y) \xi_y + \partial_z \xi_z .
\]
\item If $\Sigma$ has genus $g > 1$, then there is an isomorphism of BV algebras
\[
\big( \widehat{HH^*}(\J), \, \widehat{\Delta}_{\pi_0} \big) \cong \big( \mathbb{H}_*(\Sigma) \otimes \C[z^{\pm 1}] [\partial_z], \, \rho \otimes id + id \otimes \partial_z \xi_z \big)
\]
where $\rho$ is the string topology BV operator.
\end{enumerate}
\end{thm}

\begin{proof}
By Theorem \ref{loc of cohomology} and Proposition \ref{morita invariance}, the map 
\[
inc^* \circ \Psi^* \circ \hL^* : \widehat{HH^*}(\J) \longrightarrow HH^*(\C[\pi_1(\Sigma, v_o)] \otimes \C[z^{\pm1}])
\]
is an isomorphism of BV algebras when the BV structures are induced by the nondegenerate elements $\pi_0$ and $\pi_0' : = tr_* \Psi_* L_*(\pi_0)$. Clearly, each map $tr_*$, $\Psi_*$, and $L_*$ preserves the $H_1(\Sigma, \z) \times \z^{PM(\Q)}$-grading for all perfect matchings, so by Lemma \ref{degree VDB}, $\pi_0'$ must have degree $0$ in $H_1(\Sigma, \z)$ and degree $1$ in all perfect matchings. 

If the genus is $1$, the only nondegenrate element of these degrees is, up to to scaling, the $3$-form $x^{-1} y^{-1} dx dy dz$. The resulting BV differential is precisely the divergence operator $\Div_0$ with the stated formula.

If the genus is greater than $1$, then under the Kunneth isomorphism, the only nondegenerate element of the correct degrees is, up to scaling, $\pi_s \otimes dz$. By Lemma \ref{Kunneth}, the the resulting BV differential on the tensor product is 
\[ 
\rho \otimes id + id \otimes \partial_z \xi_z.
\]
\end{proof}

\begin{notn}
As the BV differential of the Jacobi algebra is unique, we will denote it simply as $\Delta$ when it is clear from context.
\end{notn}

\section{{H}ochschild cohomology of the {J}acobi algebra} 

Throughout, $\Q$ is assumed to be a zigzag consistent dimer that is embedded in a surface $\Sigma$ and admits a perfect matching. To a limited extent, the Hochschild cohomology of the Jacobi algebra is analyzed for arbitrary positive genus. Specializing to genus $1$, we compute the cohomology explicitly in terms of the combinatorics of the dimer, including a description of the BV structure. 

\subsection{Zeroth {H}ochschild cohomology} \label{section HH_0}

We summarize known results about the center of the Jacobi algebra. For a hyperbolic surface, the center of the fundamental group algebra $\C[\pi_1(\Sigma, v_o)]$ is trivial, so the isomorphism in Theorem \ref{matrix algebra} tells us that the center of $\JW$ is $\C[\ell^{\pm 1}]$. It is then deduced that the center of $\J$ consists of the nonnegative powers of $\ell$. 

\begin{prop} \label{higher genus center}
Suppose $\Q$ is a zigzag consistent dimer in a surface of genus $g > 1$ that admits a perfect matching. Then $HH^0(\J) \cong \C [\ell]$.
\end{prop}

If the genus is $1$, the center of $\JW$ is isomorphic to the algebra of Laurent polynomials in three variables,
\[
\C[x^{\pm 1}, y^{\pm 1}, z^{\pm 1}]
\]
where $x$ and $y$ form a basis of $H_1(\Sigma, \z)$. The monomial $x^r y^s z^t$ corresponds to a sum of closed paths, one for each vertex, with homology $rx + sy$ and degree $t$ in $\Perf_o$, the perfect matching in the definition of $\Psi$. Because any perfect matching can serve as $\Perf_o$, the central element is homogeneous in all perfect matchings. Conversely, for any closed path $p$ at a vertex $v$, there exists a unique central element $f$ with the same $H_1(\Sigma, \z) \times \z^{PM(\Q)}$-degree as $p$ such that $f \cdot v = p$. 

The subalgebra $\zc$ can be constructed by a fan arising from the zigzag and perfect matching data. Consider the rays generated by the homology classes of the antizigzag cycles,
\[
\langle \eta_i \rangle : = \R_{\geq 0} \cdot \eta_i \subset H_1(\Sigma, \z) \otimes_{\z} \R \cong \R^2, \; \; 1 \leq i \leq N
\]
and the two-dimensional cones generated by consecutive classes,
\[
\sigma_i : = \R_{\geq 0} \cdot \eta_i + \R_{\geq 0} \cdot \eta_{i+1},
\]
Let $\Int \sigma_i$ denote the interior of cone $\sigma_i$,
\[
\Int \sigma_i : = \sigma_i \setminus \big( \langle \eta_i \rangle \cup \langle \eta_{i+1} \rangle \big).
\]
For each lattice point $\alpha \in H_1(\Sigma, \z)$ and each $v \in \Q_0$, there exists a minimal closed path $u^{\alpha}_v$ at $v$ and representing $\alpha$ (\cite{abc} Lemma 3.18). Uniqueness of minimal paths implies that the $u^{\alpha}_v$ have the same perfect matching degrees, so the sum
\[
x_{\alpha} : = \sum_{v \in \Q_0} u_v^{\alpha}
\]
is a minimal central element of $\J$. If $\alpha \neq 0$ falls on the ray $\langle \eta_i \rangle$, then $x_{\alpha} = x_{\eta_i}^n$ for some $n > 0$. The element $x_{\eta_i}$ is the sum of antizigzag cycles of homology $\eta_i$, so by Theorem \ref{matching polygon}, $x_{\alpha}$ has degree $0$ in the matchings $\Perf_{i-1}$, $\Perf_{i}$, and all boundary matchings between them on $MP(\Q)$. If instead $\alpha$ lies in $\Int \sigma_i$, then positive integers $r$, $s$ and $t$ can be found for which $r \alpha = s \eta_i + t \eta_{i+1}$, implying
\[
x_{\alpha}^r = x_{\eta_i}^s x_{\eta_{i+1}}^t.
\]
Consequently, $\Perf_i$ is the the unique perfect matching in which $x_{\alpha}$ has degree $0$. Any other central element $y \in \J$ of degree $\alpha$ is an $\ell$-multiple of $x_{\alpha}$ (\cite{Bocklandt2012} Lemma 7.4).

\begin{thm}[\cite{Broomhead}; \cite{abc} Theorem 3.20]
Suppose $\Q$ is a zigzag consistent dimer embedded in a torus $\Sigma$. 
\begin{enumerate}
\item If $\alpha \in \langle \eta_i \rangle \cap \z^2$, then $x_{\alpha}$ has degree $0$ in $\Perf_{i-1}$, $\Perf_{i}$, and all the boundary matchings between them in $MP(\Q)$.
\item If $\alpha \in \Int \sigma_i \cap \z^2$, then $\Perf_i$ is the unique perfect matching in which $x_{\alpha}$ has degree $0$.
\item There is an isomorphism of algebras
\[
HH^0(\J) \cong \zc \cong \frac{\C[x_{\alpha}, \ell \, | \, \alpha \in H_1(\Sigma, \z) \setminus \{0 \}] }{ (x_{\alpha} x_{\beta} - x_{\alpha + \beta} \ell^{\degree_{\Perf_i}(x_{\alpha}) + \degree_{\Perf_i}(x_{\beta})} \, | \, \alpha + \beta \in \sigma_i )}.
\] 
\end{enumerate}
\end{thm}

In fact, $\zc$ is the coordinate ring of the affine toric threefold from the cone on $MP(\Q)$, of which the Jacobi algebra is an NCCR. The two-dimensional cones $\sigma_i$ of the antizigzag fan correspond to the facets of the dual cone in this perspective. 

\subsection{First {H}ochschild cohomology}

For $A$ equal to $\J$ or $\JW$, let
\begin{itemize}
\item $\Der_{\bk}(A)$ be the $\zc$-module of derivations of $A$ that evaluate trivially on $\bk$ and
\item $\Inner_{\bk}(A) \subset \Der_{\bk}(A)$ be the subspace spanned by inner derivations, 
\[
ad_p: q \mapsto [p, q] = pq - qp \; \; \; \forall q \in \J 
\]
for some closed path $p$ in $A$.
\end{itemize}
The first Hochschild cohomology can be computed as the quotient
\[
HH^1(A) \cong \Der_{\bk}(A) / \Inner_{\bk}(A),
\]
which is the space of outer derivations of $A$.

An element $D$ of $\Der_{\bk}(\JW)$ is determined by its evaluation on the arrows of $\overline{Q}$. To be well-defined, it must evaluate trivially on the defining relations of $\JW$ \eqref{relations}, \eqref{inverse relations}. In particular, if $a \in \Q_1$ and
\begin{align*}
R_a^+ & = a_1 a_2 \dots a_m, \; \; a_i \in \Q_1 \\
R_a^- & = b_1 b_2 \dots b_n, \; \; b_j \in \Q_1,
\end{align*}  
then we must have
\begin{equation} \label{derivation}
\sum_{i = 1}^{m} a_1 \dots a_{i-1}D(a_i)a_{i+1} \dots a_m = \sum_{j = 1}^{n} b_1 \dots b_{j-1}D(b_j)b_{j+1} \dots b_n.
\end{equation}
Conversely, any map $D:  \C \Q_1 \rightarrow \JW$ respecting the $\bk$-bimodule structure and satisfying the above equation for all arrows defines a derivation.  

Let $\N = N \otimes_{\z} \C$, $\Ni = N^{in} \otimes_{\z} \C$, $\No = N^{out} \otimes_{\z} \C$, and 
\[
\zcl : = \zc \otimes \C [\ell, \ell^{-1}],
\]
the center of the localization. An element of $\zcl \otimes \N$ can naturally be identified with a derivation of $\JW$,
\[
\chi: \zcl \otimes \N \hookrightarrow \Der_{\bk}(\JW), \; \; \chi(f \otimes \gamma)(a)= f \gamma(a) \, a \; \; \forall a \in \overline{\Q}_1.
\]
For a vertex $v$, the coboundary $\partial v^* \in N^{in}$ corresponds to $ad_v$, and so $\N^{in}$ maps into the subspace $\Inner_{\bk}(\JW)$. Therefore, there is a well-defined map
\[
\overline{\chi}: \zcl \otimes \No \longrightarrow HH^1(\JW).
\]

\begin{lem} \label{derivations}
Suppose $\Q$ is a zigzag consistent dimer that is embedded in $\Sigma$ and admits a perfect matching. The map $\overline{\chi}$ is an injection of $\zcl$-modules. If the genus is $1$, then both $\chi$ and $\overline{\chi}$ are isomorphisms.
\end{lem}

\begin{proof}
To prove the first assertion, let
\[
f \in \bigoplus_{v \in \Q_0} v \, \JW \, v,
\]
and suppose $ad_f$ is in the image of $\chi$. We may assume that $f$ is homogeneous in the $H_1(\Sigma, \z) \times \z^{PM(\Q)}$-grading. The degrees must coincide with those of an element of $\zcl$, so by Lemma 7.4 of \cite{Bocklandt2012}, $f$ must itself be central. Then
\[
ad_f = f \sum_{v \in \Q_0} ad_v,
\]
which is in the image of $\zcl \otimes \N^{in}$. 

Now suppose $\Sigma$ is a torus. We have to show that $\chi$ is surjective. If $D \in \Der_{\bk}(\JW)$, then for every $a \in \Q_1$, $D(a)$ is an element of $t(a) \JW h(a)$, so $D(a) a^{-1}$ is an element of $t(a) \JW t(a)$. Thus, by the discussion in \S \ref{section HH_0}, there exists a unique $f_a \in \zcl$ such that 
\[
D(a) a^{-1} = f_a \, t(a).
\]
The assignment of arrows $a \mapsto f_a$ must satisfy the equation \eqref{derivation}, meaning that for all $F_1, F_2 \in \Q_2$,
\[
\sum_{a \in \partial F_1} f_a = \sum_{a \in \partial F_2} f_a.
\]
Hence, it is a $\zcl$-grading of $\J$ and an element of $\zcl \otimes \N$ which $\chi$ maps to $D$. 
\end{proof}

\begin{notn}
From here onwards, we will not notationally distinguish between $N$, $N^{in}$, and $N^{out}$ on the one hand and there isomorphic images under $\chi$ and $\overline{\chi}$ on the other.
\end{notn}

We would like to classify the derivations of the Jacobi algebra. Since $\Q$ is zigzag consistent, the universal map $L: \J \rightarrow \JW$ is injective, so the task is to determine which derivations of $\JW$ preserve $\J$. It will help our analysis to decompose derivations into perfect matchings under $\chi$. The image of $\Perf \in PM(\Q)$ is the derivation  
\[ 
E_P: \JW \rightarrow \JW,  \; \; \; E_{\Perf}(p) = \degree_{\Perf}(p) \, p,
\]
the Euler operator associated to the grading by $\Perf$. By Lemma 2.11 of \cite{Broomhead}, the collection of such derivations generate $\Der_{\bk}(\JW)$ over $\zcl$.

\begin{lem} \label{Jacobi der}
Suppose $\Q$ is a zigzag consistent dimer in a torus. As a $\zc$-module, $\Der_{\bk}(\J)$ is isomorphic to 
\[
\zc  \cdot \N \oplus \bigoplus_{\substack{\alpha \in \Int \sigma_i \\ 1 \leq i \leq N}} \C \, x_{\alpha} \ell^{-1} E_{\Perf_i}.
\]
\end{lem}

\begin{proof}
For $\alpha \in H_1(\Sigma, \z)$, $n \in \z$, and $\gamma \in N$, consider the generic element $x_{\alpha} \ell^n \gamma \in \zcl \cdot \N$. The map $\gamma$ sends $a$ to an integer multiple of itself. If the derivation preserves $\J$, then for all $a \in \Q_1$, the path
\[
x_{\alpha} \ell^n \gamma(a)
\]
must be in $\J$, which occurs if and only if $x_{\alpha} \ell^n \gamma(a)$ has nonnegative degrees in all perfect matchings (\cite{abc} Lemma 3.19):
\[
\degree_{\Perf}(x_{\alpha} \ell^n a) = \degree_{\Perf}(x_{\alpha}) + n + \degree_{\Perf}(a) \geq 0 \; \; \forall \, \Perf \in PM(\Q), \; \forall \, a \; \text{such that} \; \gamma(a) \neq 0.
\] 
We consider cases based on the location of $\alpha$ in the antizigzag fan. 

First, suppose $\alpha \in \Int \sigma_i$ for some $i$. Then for all $a \in \Q_1$,
\begin{equation} \label{nonneg degree}
\degree_{\Perf_i}(x_{\alpha} \ell^n a) = n + \degree_{\Perf_i}(a) = \begin{cases}
n+1 & \text{if} \; a \in \Perf_i \\
n & \text{if} \; a \notin \Perf_i.
\end{cases}
\end{equation}
If $n \leq -2$, then this quantity is always negative, so $\gamma$ must be trivial. If $n = - 1$, then $\gamma(a)$ can be nonzero only if $a \in \Perf_i$, and hence if $\gamma$ is nontrivial, $\gamma = E_{\Perf_i}$. The derivation is
\[
x_{\alpha} \ell^{-1} E_{\Perf_i}.
\] 
In any other perfect matching, $x_{\alpha}$ has positive degree, so $x_{\alpha} \ell^{-1} E_{\Perf_i}$ preserves $\J$. If $n \geq 0$, then $\gamma$ can be any element of $N$. 

Now suppose $\alpha \in \langle \eta_{i+1} \rangle$. Then in addition to \eqref{nonneg degree}, we have
\[
\degree_{\Perf_{i+1}}(x_{\alpha} \ell^n a) = n + \degree_{\Perf_{i+1}}(a) = \begin{cases}
n+1 & \text{if} \; a \in \Perf_{i+1} \\
n & \text{if} \; a \notin \Perf_{i+1}.
\end{cases}
\] 
Again, if $n \leq -2$, then this quantity is always negative, so $\gamma$ must be trivial. If $n = -1$, then $\gamma(a)$ can be nonzero only if $a \in \Perf_i \cap \Perf_{i+1}$. However, by Theorem \ref{matching polygon}, in any boundary cycle meeting a zigzag cycle of homology $-\eta_i$, there is no arrow in the intersection of the two corner matchings. This forces $\gamma$ to be zero everywhere. If $n \geq 0$, then $\gamma$ can be any element of $N$. 
\end{proof}

For a corner matching $\Perf_i$ and $\alpha \in \Int \sigma_i$, let $\partial_i, \partial_{\alpha} \in HH^1(\J)$ denote the classes of $E_{\Perf_i}$ and $x_{\alpha} \ell^{-1} E_{\Perf_i}$, respectively. According to the lemma, these elements generate $HH^1(\J)$ as a $\zc$-module. Specifically, the collection 
\[
\{ \partial_i \, | \, 1 \leq i \leq N \}
\]
generates the submodule $\zc \cup \No $. 

\begin{thm} \label{HH^1 free}
Suppose $\Q$ is a zigzag consistent dimer in a torus. Additively, $HH^1(\J)$ is isomorphic to
\[
\zc \cup \No \oplus \bigoplus_{\substack{\alpha \in \Int \sigma_i \\ 1 \leq i \leq N}} \C \partial_{\alpha},
\]
and it is torsion free under the monoid action of $\{1, \ell, \dots, \ell^n, \dots \}$.
\end{thm}
 
\begin{proof}
Let $D$ be an inner derivation of $\JW$ that preserves $\J$. By Lemma \ref{Jacobi der}, $D$ can be decomposed as
\[
D = D' + \sum_{\substack{\alpha \in \Int \sigma_i \\ 1 \leq i \leq N}} \lambda_{\alpha} x_{\alpha} \ell^{-1} E_{\Perf_i}
\]
where $D' \in \zc \cup \N$ and only finitely many of the coefficients $\lambda_{\alpha} \in \C$ are nonzero. Evaluate both sides at $\ell$ to obtain
\[
0 = D'(\ell) + \sum_{\substack{\alpha \in \Int \sigma_i \\ 1 \leq i \leq N}} \lambda_{\alpha} x_{\alpha},
\]
Since $D'(\ell) \subset \ell \cdot \zc$ and the sum consists of minimal elements, it must be that
\[
 \sum_{\substack{\alpha \in \Int \sigma_i \\ 1 \leq i \leq N}} \lambda_{\alpha} x_{\alpha} = 0 \; \Longrightarrow \; \lambda_{\alpha} = 0 \; \; \forall \, \alpha.
\] 
Hence, $D$ is an element of $\zc \cup \N$, implying it is an inner derivation of $\J$.

The map $L^*: HH^1(\J) \rightarrow HH^1(\JW)$ is therefore injective; equivalently, $HH^1(\J)$ is torsion free (Corollary \ref{kernel tor}). From this fact and Lemma \ref{derivations}, we deduce that the surjective map
\[
\zc \cup \No \oplus \bigoplus_{\substack{\alpha \in \Int \sigma_i \\ 1 \leq i \leq N}} \C \partial_{\alpha} \rightarrow HH^1(\J)
\]
is an isomorphism.
\end{proof}

\subsection{Zeroth {H}ochschild homology and third {H}ochschild cohomology}

Rather than directly compute the third cohomology group, which is less intuitive, we compute the zeroth homology group and then make an identification through Van den Bergh duality. The zeroth homology is simply the quotient
\[
HH_0(\J) \cong \bigoplus_{v \in \Q_0} v \J v \, / \, \big( [p, q] \, | \, h(p) = t(q), t(p) = h(q) \big).
\]
In other words, two closed paths in $\J$ are equivalent in $HH_0(\J)$ if they can be realized as products that differ by a linear combination of commutators. Let $[p] \in HH_0(\J)$ denote the class of a closed path $p \in \J$.   

\begin{prop} \label{torsion HH_0} 
Let $\Q$ be a dimer model and $v, w$ be two vertices. Then
\begin{enumerate}
\item $[v] = [w]$ if and only if $v = w$;
\item $\ell \cdot [v] = \ell \cdot [w]$. 
\end{enumerate}
\end{prop}

\begin{proof}
The terms in the relations of the Jacobi algebra \eqref{} have path length at least $2$, so if $p, q$ are paths such that $pq = v$, then $p = q = v$. The vertices must therefore project to distinct nonzero classes in $HH_0(\J)$.

For the second assertion, it suffices to consider vertices $v$ and $w$ lying in the same face. Let $p: v \rightarrow w$ and $q: w \rightarrow v$ be paths such that $pq$ and $qp$ are boundary paths of the face. Then
\[
\ell \cdot [v] = [pq] = [qp] = \ell \cdot [w].
\]  
\end{proof}

Consequently, differences of vertices are torsion elements under the monoid action of the potential,
\[
[v] - [w] \in tor_{\ell}(HH_0(\J)) \; \; \forall \, v, w \in \Q_0.
\] 
They lie in the kernel of the map $L_*: HH_0(\J) \rightarrow HH_0(\JW)$ according to Corollary \ref{kernel tor}. So unlike the zeroth and first cohomologies, the zeroth homology (third cohomology) of the Jacobi algebra cannot be realized as a subspace of the localization.

When the ambient surface $\Sigma$ is a torus, every closed path is a summand of a unique homogeneous central element by the discussion in \S \ref{section HH_0}. Thus, $HH_0(\J)$ is generated by the vertices as a $\zc$-module,
\begin{equation} \label{surjection HH_0}
\bigoplus_{v \in \Q_0} \zc \cdot [v] \twoheadrightarrow HH_0(\J). 
\end{equation}
To compute the kernel of this map, we examine how multiplication by a central element relates the vertices. As seen in Proposition \ref{torsion HH_0}, multiplication by $\ell$ equates any two vertices because they can be connected by boundary paths. More generally, if $f \in \zc$ and $v$ and $v'$ could be connected by a chain of intersecting closed paths that were summands of $f$, then $f \cdot v -f \cdot v'$ could be expressed as a sum of commutators, implying $f \cdot [v] = f \cdot [v']$. The existence of such a chain is equivalent to the existence of a path between $v$ and $v'$ whose arrows are contained in lifts of $f$ to $\C \Q$. We formalize this notion as follows.

\begin{defn}
For an $H_1(\Sigma, \z) \times \z^{PM(\Q)}$-homogeneous element $f \in \zc$, let 
\[
\Q_1(f) = \{ a \in \Q_1 \, | \, a \; \text{is contained in a lift of $f$ to $\C \Q$} \}.
\]
A path $p \in \J$ is \textbf{contained in} $f$ if it has a lift $\tilde{p} \in \C \Q$ such that, as a collection of arrows, $\tilde{p} \subset \Q_1(f)$. 
\end{defn}

We would like to show that, if $p$ is contained in $f$, then in fact any lift of $p$ is contained in $\Q_1(f)$, justifying the use of the term for an element of $\J$. This requires further analysis of the equivalence of paths in the Jacobi algebra. 

If two paths in the dimer project to the same class in $\J$, then one can be obtained from the other by a sequence of substitutions of partial cycles \eqref{relations}. We define a groupoid $\mJ$ to keep track of the possible substitution sequences. The objects of $\mJ$ are the paths in $\C \Q$. For each path $p$, let $Id_p$ be the identity element in $\mJ(p, p)$. If $q$ is obtained from $p$ by replacing a single instance of $R_a^{\pm}$ with $R_{a}^{\mp}$ for some $a \in p$, let $S(p, q)$ be a formal arrow from $p$ to $q$. The morphism spaces of $\mJ$ are generated by these arrows, subject to the relations
\[
Id_q S(p, q) = S(p, q) = S(p, q) Id_p, \; \; S(q, p) S(p, q) = Id_p.
\]
In general, for arbitrary paths $p$ and $q$, an element of $\mJ(p, q)$ is a string of arrows 
\[
S(p_{n-1}, q = p_n) \dots S(p_1, p_2) S(p = p_0, p_1). 
\]
Notice that, if $p$ and $q$ are not in the same class in $\J$, then $\mJ(p, q) = \emptyset$.

We say that $S(p, q)$ is an \textbf{upwards substitution} if it replaces a negative partial cycle $R_a^-$ with the corresponding positive one $R_a^+$; otherwise, we say $S(p, q)$ is a \textbf{downwards substitution}. The type of substitution is encoded in the data $(p, q)$, but for emphasis, we write 
\[
\Su(p, q), \; \; \Sd(p, q)
\]
for upwards and downwards substitutions, respectively.


\begin{lem} \label{groupoid commutativity}
Suppose $\Q$ is a zigzag consistent dimer and $p \neq r$ are minimal paths. There exists a morphism
\[
\Sd(q, r) \Su(p, q) \in \mJ(p, r)
\]
if and only if there exists a morphism
\[
\Su(q', r) \Sd(p, q') \in \mJ(p, r).
\]
\end{lem}

\begin{proof}
Suppose there is a morphism 
\[
\Sd(q, r) \Su(p, q) \in \mJ(p, r)
\]
where $\Su(p, q)$ replaces $R^-_{b_1}$ with $R^+_{b_1}$ and $\Sd(q, r)$ replaces $R^+_{b_2}$ with $R^-_{b_2}$. Write
\[
p = a_1 a_2 \dots a_i R_{b_1}^- a_{i+1} \dots a_n, \; \; q = a_1 a_2 \dots a_i R_{b_1}^+ a_{i+1} \dots a_n, \; \; a_k \in \Q_1 \; \; \forall k
\] 
We claim that $R^+_{b_2}$ is contained in the subpaths $a_1 \dots a_i$ or $a_{i+1} \dots a_n$. For a contradiction, suppose $R^+_{b_1}$ and $R^+_{b_2}$ overlap as subpaths of $q$. Since every arrow in a dimer is contained in a unique positive face, $R^+_{b_1}$ and $R^+_{b_2}$ must be part of the same positive boundary cycle. But since $p \neq r$, $\Sd(q, r)$ is not the inverse of $\Su(p, q)$, meaning $R^+_{b_1} \neq R^+_{b_2}$. This implies $R^+_{b_2}$ contains $a_i$ or $a_{i+1}$, which must be the arrow $b_1$. Therefore, $q$ contains the complete boundary cycle $b_1 R^+_{b_1}$ or $ R^+_{b_1} b_1$, contradicting the minimality of $p$, $q$, and $r$.

Without loss of generality, assume $R^+_{b_2}$ is contained in the segment $a_{i+1} \dots a_n$ and write
\[
p = a_1 \dots a_i R_{b_1}^- a_{i+1} \dots a_j R_{b_2}^+ a_{j+1} \dots a_n.
\]
Then let 
\[
q' = a_1 \dots a_i R_{b_1}^- a_{i+1} \dots a_j R_{b_2}^ -a_{j+1} \dots a_n.
\]
The path $r$ can be obtained from $q'$ by substituting $R_{b_1}^+$ for $R_{b_1}^-$. This sequence determines the morphism
\[
\Su(q', r) \Sd(p, q') \in \mJ(p, r).
\] 
\end{proof}

The result effectively states that, between equivalent minimal paths, upwards and downwards morphisms commute. This is the key to proving invariance of our notion of containment to the choice of lift.

\begin{lem} \label{invariance of containment}
Let $\Q$ be a zigzag consistent dimer. Suppose $p \in \J$ is a path contained in a homogeneous central element $f$. If $\tilde{p} \in \C \Q$ is any lift of $p$, then $\tilde{p} \subset \Q_1(f)$. 
\end{lem}

\begin{proof}
Since $\Q_1(\ell) = \Q_1$ and $\Q_1(1) = \emptyset$, all paths are contained in $\ell$, and no path is contained in $1$. So the statement is true for a hyperbolic surface, and for a torus, we have only to consider a minimal central element $f = x_{\alpha}$, $\alpha \in H_1(\Sigma, \z) \setminus \{0\}$. It suffices to show that $R^{+}_a \subset 
Q_1(f)$ if and only if $R^{-}_a \subset \Q_1(f)$ for any arrow $a$. In fact, we prove a stronger statement: $R_a^{\pm} \subset \Q_1(f)$ if and only if $R_a^{\pm}$ is contained in a closed path in $\C \Q$ representing $f \cdot v$ for some $v \in \Q_0$.

Write
\[
R_a^+ = b_1 b_2 \dots b_m, \; \; b_i \in \Q_i \; \; \forall i
\]
and suppose $R_a^+ \subset \Q_1(f)$. We proceed by induction: assume $b_1 \dots b_i$ is contained in a closed path
\[
p = c \dots b_1 b_2 \dots b_i \in \C \Q
\]
representing $f \cdot v$, where $v = t(c) = h(b_i) = t(b_{i+1})$. Since $b_{i+1} \in \Q_1(f)$, there is a closed path
 \[
 q = b_{i+1} \dots d, 
 \]
representing $f \cdot v$ as well.

\begin{figure}[h]
\[
\begin{tikzpicture}
\draw (0,0) circle [radius=.2];
\draw (0,2) circle [radius=.2];
\node at (2, 0) {$v$};
\draw (2,0) circle [radius=.2];
\draw (4,0) circle [radius=.2];
\draw (4, 2) circle [radius=.2];
\draw (4, -1) circle [radius=.2];
\draw (0, -1) circle [radius=.2];
\node at (2, 1) {$+$};
\draw[->, dashed, blue] (-2, 2) -- (-0.2, 2);
\node[above, blue] at (-1, 2) {$p$};
\draw[->, blue] (0, 1.8) -- (0, 0.2);
\node[right, blue] at (0, 1) {$b_1$};
\draw[->, blue] (0.2, 0) -- (1.8, 0);
\node[above, blue] at (1, 0) {$b_2$};
\draw[->, red] (2.2, 0) -- (3.8, 0);
\node[above, red] at (3, 0) {$b_3$};
\draw[->] (4, 0.2) -- (4, 1.8);
\node[left] at (4, 1) {$b_4$};
\draw[->, dashed] (3.8, 2) -- (0.2, 2);
\node[above] at (2, 2) {$a$}; 
\draw[->, dashed, red] (4.2, 0) -- (6, 0);
\draw[->, blue] (2.18, -0.09) -- (3.82, -0.91);
\node[above, blue] at (3, -0.5) {$c$};
\draw[->, red] (0.18, -0.91) -- (1.82, -0.09); 
\node[above, red] at (1, -0.5) {$d$};
\draw[->, dashed, red] (-2.18, -2.09) -- (-0.18, -1.09);
\draw[->, dashed, blue] (4.18, -1.09) -- (6.18, -2.09);
\node[above, red] at (-1.18, -1.59) {$q$};
\end{tikzpicture}
\]
\caption{An example of the inductive step at $i = 2$.}
\end{figure}
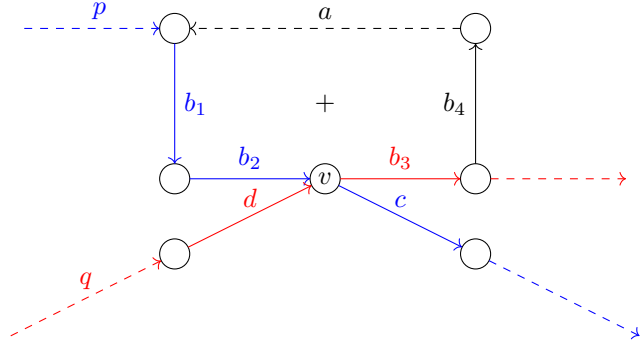
 
The paths $p$ and $q$ are equivalent in $\J$, and hence there exists a sequence of substitutions by which $q$ is obtained from $p$,
\[
S(p_{n-1}, q = p_n) S(p_{n-2}, p_{n-1}) \dots S(p = p_0, p_1) \in \mJ(p, q).
\]
By repeated application of Lemma \ref{groupoid commutativity}, there must also be a morphism from $p$ to $q$ in which all the downward substitutions occur first,
\[
\Su(p_{n-1}', q = p_n') \dots \Su(p_j', p_{j+1}') \Sd(p_{j-1}', p_{j}') \dots \Sd(p = p_0', p_1') \in \mJ(p, q).
\]
We claim that $b_1 b_2 \cdots b_i$ is a subpath of $p_j'$. To see this, observe that any positive partial cycle that ends in $b_k$, $1 \leq k \leq i$, must contain $a$:
\[
R^+_{b_{k+1}} = b_{k+2} \dots b_m a b_1 \dots b_k. 
\]  
Because $f$ is minimal and $R_a^+ \subset \Q_1(f)$, the arrow $a$ must be in a perfect matching in which $f$ has degree $0$. Therefore, none of the paths $p_0', \dots, p_n'$ contain $a$, implying the downward substitutions preserve $b_1 \dots b_i$.  

We can thus write
\[
p_j' = c' \dots b_1 \dots b_i. 
\]
If $c' = b_{i+1}$, then we have finished. Otherwise, the upwards substitutions must turn $c'$ into $b_{i+1}$, in which case there exists $p_k'$, $j < k \leq n$,  containing the positive partial cycle beginning with $b_{i+1}$:
\[
R_{b_i}^+ = b_{i+1} \dots b_{m} a b_1 \dots b_{i-1}.
\]
As before, this is disallowed because $a$ is in the perfect matching killing $f$. Consequently, it must have been that $c' = b_{i+1}$, and $p_j'$ is a representative of $f \cdot v$ containing $b_1, \dots, b_{i+1}$. 
\end{proof}

\begin{cor} \label{arrow sets same}
Suppose $\Q$ is a zigzag consistent dimer. For any homogeneous $f \in \zc$ and $n > 0$, $\Q_1(f) = \Q_1(f^n)$.
\end{cor}

\begin{proof}
The direction $\Q_1(f) \subset Q_1(f^n)$ is obvious. For the reverse, let $a \in \Q_1(f^n)$, so there exists a path $q \in \C \Q$ containing $a$ and lifting $f^n \cdot t(a) \in \J$. But $f^n \cdot t(a) = (f \cdot t(a))^n$, which is contained in $f$. Thus, Lemma \ref{invariance of containment} tells us that $q \subset \Q_1(f)$.
\end{proof}

Returning to the computation of the kernel of \eqref{surjection HH_0}, we make a definition to capture the notion of two vertices being connected by cycles in a central element.

\begin{defn}
Let $f$ be an $H_1(\Sigma, \z) \times \z^{PM(\Q)}$-homogeneous element of $\zc$. A vertex $v$ is $f$-\textbf{connected} to a vertex $v'$ if there exists a path $p: v \to v'$ that is contained in $f$.
\end{defn}

\begin{lem} \label{connectedness equiv}
Suppose $\Q$ is a zigzag consistent dimer embedded in a torus. Let $f$ be a homogeneous element of $\zc$. The relation of $f$-connectedness is an equivalence relation. Moreover, there is an isomorphism of $\zc$-modules
\[
\bigoplus_{v \in \Q_0} \zc \cdot [v] \, / \, \{ f \cdot [v] - f \cdot [v'] \, | \, f \in \zc, \; v, v' \; \text{are $f$-connected} \, \} \longrightarrow HH_0(\J).
\]
\end{lem}

\begin{proof}
Reflexivity and transitivity of $f$-connectedness are clear. To see that the relation is symmetric, write
\[
p = a_1 a_2 \dots a_n: v \rightarrow v', \; \; a_i \in \Q_1
\]  
for a path contained in $f$ that connects $v$ to $v'$. Then there exist paths $C_i: h(a_i) \rightarrow t(a_i)$, $1 \leq i \leq n$, that complete the arrows $a_i$ to summands of $f$:
\[
a_iC_i = f \cdot t(a_i) \in \J.
\]
But then the path 
\[
p: C_n C_{n-1} \dots C_1: v' \rightarrow v
\]
is also contained in $f$, so $v'$ is $f$-connected to $v$.

Observe that, for each $1 \leq i < n$, 
\[
C_i a_i = f h(a_i) =f t(a_{i+1}) =  a_{i+1}C_{i+1}
\]
Therefore, in $\J$, we have
\[
f v - f v' = \sum_{i = 1}^n a_i C_i - C_i a_i = \sum_{i = 1}^n [a_i, C_i],
\]
implying $f [v] - f [v']$ is trivial in $HH_0(\J)$. Therefore, the submodule
\begin{equation} \label{submodule}
\{ f \cdot [v] - f \cdot [v'] \, | \, f \in \zc, \; v, v' \; \text{are $f$-connected} \, \} \subset \bigoplus_{v \in \Q_0} \zc \cdot [v] 
\end{equation}
lies in the kernel of \eqref{surjection HH_0}.

Conversely, if $p, q \in \J$ are paths such that $h(p) = t(q)$ and $h(q) = t(p)$, then
\[
[p, q] = pq - qp = f t(p) - f h(p)
\]
where $f$ is the unique central element with summands $pq$ and $qp$. The vertices $t(p)$ and $h(p)$ are then clearly $f$-connected, so the kernel of the map \eqref{surjection HH_0} lies inside the submodule \eqref{submodule}.
\end{proof}

According to Proposition \ref{torsion HH_0}, all vertices are $\ell$-connected, but no two vertices are $1$-connected. It remains to determine how vertices are connected by the minimal central elements $x_{\alpha}$, $\alpha \in H_1(\Sigma, \z) \setminus \{0\}$. The answer depends on whether $\alpha$ lies on a ray or in the interior of a cone of the antizigzag fan.

For each $1 \leq i \leq N$, enumerate the zigzag cycles representing the homology class $-\eta_i$ as
\[
Z_{i, j}, \; \; 1 \leq j \leq m_i.
\]
By Proposition \ref{intersection properties}, parallel zigzag cycles do not intersect in arrows. Therefore, we may assume the ordering is such that, for each $j$ mod $m_i$, $\st^+(Z_{i, j}) \cup \st^-(Z_{i, j+1})$ bounds a closed connected subspace of the torus $\Sigma$ that does not contain a zigzag of homology $-\eta_i$ (see Figure \ref{zigzag strip}). Label this subspace $\E_{i, j}$. Notice that the $\E_{i, j}$ are pairwise disjoint strips of the torus.

\begin{lem} \label{opposite connectedness}
Suppose $\Q$ is a zigzag consistent dimer embedded in a torus, and let $v$ and $w$ be vertices of $\Q$. For each $1 \leq i \leq N$, $v$ and $w$ are $x_{\eta_i}$-connected if and only if there exists $1 \leq j \leq m_i$ such that
\[
v, w \in \E_{i, j}.
\]
\end{lem}

\begin{proof}
First, we show that any vertex $v \in \E_{i, j}$ is $x_{\eta_i}$-connected to the vertices of $\st^+(Z_{i, j})$ and $\st^-(Z_{i, j+1})$ and thus connected to all vertices in $\E_{i, j}$. Let $p \in \C \Q$ be a minimal closed path at $v$ of homology $\eta_{i-1}$, projecting to a summand of $x_{\eta_{i-1}}$ in $\J$. Since $\eta_{i-1}$ and $\eta_i$ are linearly independent, $p$ must intersect each of $Z_{i, j}$ and $Z_{i, j+1}$ in an arrow. Both $x_{\eta_{i-1}}$ and $x_{\eta_i}$ are killed by the perfect matching $\Perf_{i-1}$, which, according to Theorem \ref{matching polygon}, contains all zigs of zigzag cycles of homology $-\eta_i$. Therefore, $p$ must intersect $Z_{i, j}$, through which it enters $\E_{i, j}$, and $Z_{i, j+1}$, through which is leaves $\E_{i, j}$, only in zags. Let $a$ be the first zag of $Z_{i, j+1}$ through which $p$ leaves, and let $b$ be last zag of $Z_{i, j}$ through which $p$ enters.

Similarly, there is a minimal closed path $q$ at $t(a)$ of homology $\eta_{i+1}$, projecting to a summand of $x_{\eta_{i+1}}$. Both $x_{\eta_i}$ and $x_{\eta_{i+1}}$ are killed by the perfect matching $\Perf_{i}$, which contains all zags of zigzag cycles of homology $-\eta_{i}$. Therefore, $q$ must intersect $Z_{i, j}$, through which it leaves $\E_{i, j}$, and $Z_{i, j+1}$, through which it enters $\E_{i. j}$, only in zigs. Let $c$ be the first zig of $Z_{i, j}$ through which $q$ leaves.

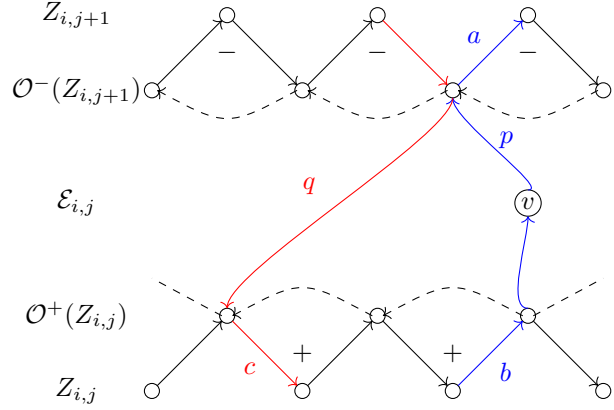
\begin{figure}[h]
\[
\begin{tikzpicture}
\node at (-1, 1) {$Z_{i, j+1}$};
\node at (-1, 0) {$\st^-(Z_{i, j+1})$};
\node at (-1, -1.5) {$\E_{i, j}$};
\node at (-1, -3) {$\st^+(Z_{i, j})$};
\node at (-1, -4) {$Z_{i, j}$}; 
\draw (0,0) circle [radius=0.1];
\draw (1,1) circle [radius=0.1];
\draw (2,0) circle [radius=0.1];
\draw (3,1) circle [radius=0.1];
\draw (4,0) circle [radius=0.1];
\draw (5,1) circle [radius=0.1];
\draw (6,0) circle [radius=0.1];
\draw (0,-4) circle [radius=0.1];
\draw (1,-3) circle [radius=0.1];
\draw (2,-4) circle [radius=0.1];
\draw (3,-3) circle [radius=0.1];
\draw (4,-4) circle [radius=0.1];
\draw (5,-3) circle [radius=0.1];
\draw (6,-4) circle [radius=0.1];
\draw (5, -1.5) circle [radius=0.175]; 
\node at (5, -1.5) {$v$};
\node at (1, 0.5) {$-$};
\node at (3, 0.5) {$-$};
\node at (5, 0.5) {$-$};
\node at (2, -3.5) {$+$};
\node at (4, -3.5) {$+$};
\draw[->] (0.07, 0.07) -- (0.93, 0.93);
\draw[->] (1.07, 0.93) -- (1.93, 0.07);
\draw[->] (2.07, 0.07) -- (2.93, 0.93);
\draw[->, red] (3.07, 0.93) -- (3.93, 0.07);
\draw[->, blue] (4.07, 0.07) -- (4.93, 0.93);
\node[above left, blue] at (4.5, 0.5) {$a$};
\draw[->] (5.07, 0.93) -- (5.93, 0.07);
\draw[->, dashed] (1.93, 0)..controls (1, -0.5) and (1, -0.5)..(0.07, 0);
\draw[->, dashed] (3.93, 0)..controls (3, -0.5) and (3, -0.5)..(2.07, 0);
\draw[->, dashed] (5.93, 0)..controls (5, -0.5) and (5, -0.5)..(4.07, 0);
\draw[->] (0.07, -3.93) -- (0.93, -3.07);
\draw[->, red] (1.07, -3.07) -- (1.93, -3.93);
\node[red, below left] at (1.5, -3.5) {$c$};
\draw[->] (2.07, -3.93) -- (2.93, -3.07);
\draw[->] (3.07, -3.07) -- (3.93, -3.93);
\draw[->, blue] (4.07, -3.93) -- (4.93, -3.07);
\node[below right, blue] at (4.5, -3.5) {$b$};
\draw[->] (5.07, -3.07) -- (5.93, -3.93);
\draw[->, dashed] (2.93, -3)..controls (2, -2.5) and (2, -2.5)..(1.07, -3);
\draw[->, dashed] (4.93, -3)..controls (4, -2.5) and (4, -2.5)..(3.07, -3);
\draw[-, dashed] (0.93, -3) -- (0, -2.5);
\draw[-, dashed] (5.07, -3) -- (6, -2.5);
\draw[->, blue] (5, -1.325)..controls +(right:3mm) and +(down:3mm)..(4, -0.1); 
\node[above right, blue] at (4.5, -0.9) {$p$};
\draw[->, blue] (5, -2.9)..controls +(left:3mm) and +(down:3mm)..(5, -1.675);
\draw[->, red] (4, -0.1)..controls +(down:5mm) and +(up:5mm)..(1, -2.9);
\node[above left, red] at (2.3, -1.5) {$q$};
\end{tikzpicture}
\]
\caption{\label{zigzag strip} The subspace $\E_{i, j}$.}
\end{figure}

We thus have a closed path at $v$ that is completely contained in the strip $\E_{i, j}$, 
\[
r : = p[v \to t(a)] \, q[t(a) \to t(c)] \, \st^+(Z_{i, j})[t(c) \to h(b)] \, p[h(b) \to v]
\]
where $X[v_1 \to v_2]$ denotes the subpath of $X \in \{p, q, \st^+(Z_{i, j}) \}$ from vertex $v_1$ to vertex $v_2$. Since the strip $\E_{i, j}$ contains no zigzag cycle of homology $-\eta_i$, by Theorem \ref{matching polygon}, the perfect matchings $\Perf_{i-1}$ and $\Perf_{i}$ coincide in all boundary cycles contained in $\E_{i, j}$. Therefore, the segments of $p$ and $q$ contained in $\E_{i, j}$ are killed by both $\Perf_{i-1}$ and $\Perf_{i}$, implying that the same is true for the whole path $r$. Consequently, the homology of $r$ must be in $\langle \eta_i \rangle$, and it is clearly nonzero. So $r$ must project to a summand of $x_{\eta_i}^m$ in $\J$, for some $m \in \z_{>0}$, i.e., $r \subset \Q_1(x_{\eta_i}^m)$. But $\Q_1(x_{\eta_i}^m) = \Q_1(x_{\eta_i})$ by Corollary \ref{arrow sets same}, so we conclude that $v$ is $x_{\eta_i}$-connected to the vertices of $\st^+(Z_{i, j})$ and $\st^-(Z_{i, j+1})$. 

For the other direction, note that any path connecting vertices in different strips must cross a zigzag of homology $-\eta_i$. It must therefore have positive degree in $\Perf_{i-1}$ or $\Perf_{i}$. Hence, the vertices cannot be $x_{\eta_i}$-connected.
\end{proof}

The last part of the computation is to determine how minimal elements corresponding to lattice points in the cone interiors connect vertices.

\begin{thm} \label{HH_0 comp}
Suppose $\Q$ is a zigzag consistent dimer in a torus. There is an isomorphism of $\zc$-modules
\[
HH_0(\J) \cong \bigoplus_{v \in \Q_0} \zc \cdot [v]  \, / \, \mathcal{I} 
\]
where $\mathcal{I}$ is the submodule generated by
\begin{enumerate}
\item $x_{\eta_i} ( [v] - [v'] )$ for all $1 \leq i \leq N$, $1 \leq j \leq m_i$, and $v, v' \in \E_{i, j}$;
\item $x_{\alpha}( [v] - [v'])$ for all $1 \leq i \leq N$, $\alpha \in \Int \sigma_i$, and $v, v' \in \Q_0$;
\item $\ell ([v] - [v'])$ for all $v, v' \in \Q_0$.
\end{enumerate}
\end{thm}

\begin{proof}
First, consider 
\[
\alpha = \eta_i + \eta_{i+1} \; \Longrightarrow \; x_{\alpha} = x_{\eta_i} x_{\eta_{i+1}}.
\]
Since $\eta_i$ and $\eta_{i+1}$ are linearly independent, any two closed paths representing these homologies intersect. This implies that vertices in any two strips $\E_{i,j}$ and $\E_{i, k}$ are $x_{\eta_i} x_{\eta_{i+1}}$-connected. Combining this with Lemma \ref{opposite connectedness}, we conclude all vertices are $x_{\eta_i} x_{\eta_{i+1}}$-connected.

For general $\alpha \in \Int \sigma_i$, there exist $r, s, t \in \z_{>0}$ such that
\[
r \alpha = s \eta_i + t \eta_{i+1} \; \Longrightarrow \: x_{\alpha}^r = x_{\eta_i}^s x_{\eta_{i+1}}^t,
\]
so by the above, any two vertices are $x_{\alpha}^r$-connected. Corollary \ref{arrow sets same} requires that they are $x_{\alpha}$-connected as well.
\end{proof}

Van den Berg duality with respect to the nondegenerate element $\pi_0$ establishes an isomorphism of $\zc$-modules $HH_0(\J) \cong HH^3(\J)$. For each $v \in \Q_0$, let $\theta_v \in HH^3(\J)$ be the class corresponding to $[v]$,
\[
[v] = \theta_v \cap \pi_0.
\]
It is deduced from Lemma \ref{degree VDB} that $\theta_v$ has degree $0$ in $H_1(\Sigma, \z)$ and degree $-1$ in all perfect matchings.

The $H_1(\Sigma, \z) \times \z_{\Perf}$-degree of $\pi_0$ is $(0, 1)$ for all perfect matchings, so the degree of $\theta_v$ is $(0, -1)$.  

\begin{cor} \label{HH^3 first}
Suppose $\Q$ is a zigzag consistent dimer in a torus. There is an isomorphism of $\zc$-modules
\[
HH^3(\J) \cong \bigoplus_{v \in \Q_0} \zc \cdot \theta_v \, / \, \mathcal{I}'
\]
where $\mathcal{I}'$ is the submodule generated by
\begin{enumerate}
\item $x_{\eta_i} ( \theta_v - \theta_{v'} )$ for all $1 \leq i \leq N$, $1 \leq j \leq m_i$, and $v, v' \in \E_{i, j}$;
\item $x_{\alpha}( \theta_v - \theta_{v'})$ for all $1 \leq i \leq N$, $\alpha \in \Int \sigma_i$, and $v, v' \in \Q_0$;
\item $\ell (\theta_v- \theta_{v'})$ for all $v, v' \in \Q_0$.
\end{enumerate}
\end{cor}

This characterization of $HH^3(\J)$ can be related to our previous description of $HH^1(\J)$ in Theorem \ref{HH^1 free}. Recall from Lemma \ref{derivations} that
\[
HH^1(\JW) \cong \zcl \cup \No.
\] 
Since $\JW$ is Morita equivalent to Laurent polynomials in three variables (Theorem \ref{matrix algebra}), the Hochschild cohomology of $\JW$ is the free graded commutative algebra generated by $HH^1(\JW)$ over $HH^0(\JW)$. Therefore,
\[
HH^3(\JW) \cong \zcl \cup \No \cup \No \cup \No,
\]
which is a free $\zcl$-module of rank $1$. According to Corollary \ref{kernel tor}, the morphism of algebras
\[
L^*: HH^*(\J) \longrightarrow HH^*(\JW) 
\] 
has kernel in degree $3$ equal to
\[
tor_{\ell}(HH^3(\J)) = \{ \theta_v - \theta_v' \, | \, v, v' \in \Q_0 \} \, \cup \, \{ x_{\eta_i}^m (\theta_v - \theta_v') \, | \, 1 \leq i \leq N, \, 1 \leq j, k \leq m_i, \, v \in \E_{i, j}, \, v' \in \E_{i, k}, \, m \in \z_{> 0} \}
\]
Hence, $L^*$ is injective on the subspace concentrated in nonnegative degrees in all perfect matchings,
\[
\zc \cup \No \cup \No \cup \No,
\]
and on the subspaces concentrated in degree $\alpha \in \Int \sigma_i$ and degree $-1$ in $\Perf_i$ for each $1 \leq i \leq N$,
\[
\partial_i \cup \No \cup \No.
\]
Degree considerations then allow us to deduce that, for all $v \in \Q_0$, 
\begin{align} \label{theta coeff}
\ell \cdot \theta_v  & = \lambda_{i, j, k} \, \partial_i \cup \partial_j \cup \partial_k, \\
x_{\alpha} \cdot \theta_v & = \nu_{i, j, k} \, \partial_{\alpha} \cup \partial_j \cup \partial_k \nonumber
\end{align}  
where $\alpha \in \Int \sigma_i$, $i, j, k$ are distinct, and $\lambda_{i, j, k}, \nu_{i. j, k} \in \C^*$. We can recast Corollary \ref{HH^3 first} as follows.

\begin{thm} \label{HH^3 second}
Suppose $\Q$ is a zigzag consistent dimer model in a torus. There is an isomorphism of $\zc$-modules
\begin{align*}
HH^3(\J) \cong & \, \zc \cup \No \cup \No \cup \No \, \oplus \, \bigoplus_{\substack{ \alpha \in \Int \sigma_i \\ 1 \leq i \leq N}} \partial_{\alpha} \cup \No \cup \No \\
& \C \{ x_{\eta_i}^m \theta_v \, | \, 1 \leq i \leq N, \, 1 \leq j \leq m_i, \, v \in \E_{i, j}, \, n \in \z_{> 0} \} \, \oplus \, \C \{ \theta_v \, | \, v \in \Q_0 \}.
\end{align*}
\end{thm}

\subsection{Second {H}ochschild cohomology} 

Suppose that the dimer $\Q$, in addition to being zigzag consistent, has a strictly positive integer grading. By Lemma 2.11 of \cite{Broomhead}, this condition is equivalent to every arrow being contained in a perfect matching, which is always satisfied in genus $1$. Then $\J$ is a nonnegatively graded, connected $\bk$-algebra, so by Lemma 3.6.1 of \cite{EtingofGinzburg07}, the rows of the following diagram are exact:
\begin{equation} \label{exact sequence}
\xymatrix{
\bk \ar[r] & HH^3(\J) \ar[r]^-{\Delta} \ar[d]^-{\mathbb{D}} & HH^2(\J)) \ar[r]^-{\Delta} \ar[d]^-{\mathbb{D}} & HH^1(\J) \ar[r]^-{\Delta} \ar[d]^-{\mathbb{D}} & HH^0(\J) \ar[r]^-{\Delta} \ar[d]^-{\D} & 0 \\
\bk \ar[r] & HH_0(\J) \ar[r]^-{B} & HH_1(\J) \ar[r]^-{B} & HH_2(\J) \ar[r]^-{B} & HH_3(\J) \ar[r]^-{B} & 0.
}
\end{equation}
The map $\bk \rightarrow HH_0(\J)$ sends $v$ to $[v]$, and the dual map $\bk \rightarrow HH^3(\J)$ sends $v$ to $\theta_v$. Thus, as a vector space, the second cohomology decomposes as
\begin{eqnarray} \label{HH^2 decomp}
HH^2(\J) & \cong & Im \big( \Delta: HH^3(\J) \rightarrow HH^2(\J) \big) \\
&  \oplus &  Ker \big( \Delta: HH^1(\J) \rightarrow HH^0(\J) \big). \nonumber
\end{eqnarray}

When the genus is $1$, the structure of $HH^2(\J)$ can be deduced from the explicit computations of $HH^1(\J)$ and $HH^3(\J)$. The map 
\[
\Delta: HH^3(\J) \rightarrow HH^2(\J)
\]
restricted to the subspace complementary to $\C \{ \theta_v \, | \, v \in \Q_0\}$ is injective. So for all $1 \leq i \leq N$, $1 \leq j \leq m_i$, and $v \in \E_{i, j}$, let 
\[
\psi_{i, j} = \Delta (x_{\eta_i} \theta_v ).
\]  

\begin{lem} \label{BV on theta}
Suppose $\Q$ is a zigzag consistent dimer in a torus. The BV differential satisfies
\[
\Delta( f x_{\eta_i} \theta_v) = f \psi_{i, j} + x_{\eta_i} \Delta(f \theta_v)
\]
where $f \in \zc$ is an $H_1(\Sigma, \z) \times \z^{PM(\Q)}$ homogeneous element, $1 \leq i \leq N$, $1 \leq j \leq m_i$, and $v \in \E_{i, j}$.
\end{lem}

\begin{proof}
Writing $d$ for the Hochschild differential \eqref{chain}, we have a 1-boundary
\[
d (1 \otimes_{\bk} f \otimes_{\bk} x_{\eta_i} \, v )  = f \otimes_{\bk} x_{\eta_i} \, v - 1 \otimes_{
\bk} f x_{\eta_i} \, v + x_{\eta_i} \otimes_{\bk} f \, v.
\]
Thus, in cohomology,
\[
0 = f B([x_{\eta_i} \, v]) - B([f x_{\eta_i} \, v]) + x_{\eta_i} B([ f \, v])
\]
where the formula for the Connes differential $B$ \eqref{Connes} has been applied. The Van den Bergh dual of this identity yields
\[
\Delta( f x_{\eta_i} \theta_v) = f \psi_{i, j} + x_{\eta_i} \Delta(f \theta_v).
\]
\end{proof}

In particular, if $f = x_{\eta_i}^n$, then an induction argument yields
\begin{equation} \label{psi formula}
\Delta(x_{\eta_i}^{n+1} \theta_v) = (n+1) x_{\eta_i}^n \psi_{i, j} \; \; \forall \, n \geq 0.
\end{equation}

\begin{thm} \label{HH^2 comp}
Suppose $\Q$ is a zigzag consistent dimer in a torus. Additively, the second Hochschild cohomology of the Jacobi algebra is
\[
HH^2(\J) \cong \zc \cup \No \cup \No \, \oplus \, \bigoplus_{\substack{\alpha \in \Int \sigma_i \\ 1 \leq i \leq N}} \C \,\partial_{\alpha} \cup \No \, \oplus \, \C \{ x_{\eta_i}^n \psi_{i, j} \, | \, 1 \leq i \leq N, \, 1 \leq j \leq m_i, \, n \in \z_{>0} \}.
\]
\end{thm}

\begin{proof}
The BV differential $\Delta$ preserves the grading by first homology and the perfect matchings. Consequently, additively computing $HH^2(\J)$ reduces to a dimension count for each homogeneous subspace using the decomposition \eqref{HH^2 decomp}. In fact, it suffices to consider only the grading by the subgroup $H_1(\Sigma, \z) \times \z^N$ where the $\z^N$ coordinates correspond to degrees in the corner matchings: i.e.,
\[
\Lambda = (\alpha, n_1, \dots, n_N) \in H_1(\Sigma, \z) \times \z^N,
\]
where $n_i$ is the degree in $\Perf_i$. Write $HH^*(\J)_{\Lambda}$ for the $\Lambda$-homogeneous subspace.\\

\noindent \underline{Case 1:} $\alpha \in H_1(\Sigma, \z)$, $n_i < 1$ for all $i$. \\

\noindent The subspaces $HH^1(\J)_\Lambda$ and $HH^3(\J)_\Lambda$ are trivial, implying $HH^2(\J)_\Lambda$ is as well. \\

\noindent \underline{Case 2:} $\alpha = 0$, $n_i = -1$ for all $i$. \\

\noindent The subspace $HH^1(\J)_\Lambda$ is trivial, but
\[
HH^3(\J)_{\Lambda} = \C \{ \theta_v \, | \, v \in \Q_0 \}.
\] 
However, this lies in the kernel of $\Delta$ according to the sequence \eqref{exact sequence}, so $HH^2(\J)_{\Lambda} = 0$. \\

\noindent \underline{Case 3:} $\alpha = n \cdot \eta_i$ for $n > 0$, $n_{i-1} = n_i = -1$, $n_j \geq 0$ for all $j \neq i$. \\

\noindent The subspace $HH^1(\J)_{\Lambda}$ is trivial, but
\[
HH^3(\J)_{\Lambda}= \C \{ x_{\eta_i}^n \theta_v \, | \, 1 \leq k \leq m_i, \, v \in \E_{i,k} \}. 
\]
By the definition of $\psi_{i, j}$ and Lemma \ref{BV on theta}, the injective image of this subspace under $\Delta$ is
\[
HH^2(\J)_{\Lambda} = \C \{ x_{\eta_i}^n \psi_{i, j} \, | \, 1 \leq k \leq m_i \},
\]
which is $m_i$-dimensional for each $i$. \\

\noindent \underline{Case 4:} $\alpha \in \Int \sigma_i$, $n_i = -1$, $n_j \geq 0$ for all $j \neq i$. \\

\noindent The subspace $HH^0(\J)_{\Lambda}$ is trivial, so
\[
HH^1(\J)_{\Lambda} = \C \partial_{\alpha}
\]
must lie in the kernel of $\Delta$. Moreover, the BV differential is injective on
\[
HH^3(\J)_{\Lambda} = \partial_{\alpha} \cup \No \cup \No,
\]
which is one dimensional. Therefore,
\[
dim_{\C} HH^2(\J)_{\Lambda}= 2.
\]
For distinct $i$, $j$, and $k$, the products
\[
\partial_{\alpha} \cup \partial_j, \; \partial_{\alpha} \cup \partial_k,
\]
are linearly independent as elements of $HH^2(\JW)$, since $HH^*(\JW)$ is the free graded commutative algebra generated by $\No$ over $\zcl$. Hence, the same is true for them as elements of $HH^2(\J)$, implying they span the homogeneous subspace
\[
HH^2(\J)_{\Lambda} = \partial_{\alpha} \cup \No.
\] 

\noindent \underline{Case 5:} $\alpha \in H_1(\Sigma, \z)$, $n_i \geq 0$ for all $i$ \\

\noindent The dimension of 
\[
ker( \Delta: HH^1(\J)_{\Lambda} \rightarrow HH^0(\J)_{\Lambda}) 
\]
is $2$, while the injective image of
\[
HH^3(\J)_{\Lambda} = x_{\alpha} \ell^n \cdot \No \cup \No \cup \No
\]
under $\Delta$ has dimension $1$. Hence,
\[
dim_{\C} HH^2(\J)_{\Lambda} = 3.
\]
The same argument as for the preceding case then shows that
\[
HH^2(\J)_{\Lambda} = x_{\alpha} \ell^n \cdot \No \cup \No.
\]
\end{proof}

As is well known, the second Hochschild cohomology of an associative algebra classifies its infinitesimal deformations up to gauge equivalence (see e.g. \cite{TS}). In fact, deformations of $\J$ within the class of Calabi--Yau algebras can be identified using the BV differential. For each $\Phi \in HH_0(\C \Q)$, a new algebra is obtained by infinitesimally perturbing the superpotential of the Jacobi algebra in the direction of $\Phi$,
\begin{equation} \label{CY def}
J(\Q, \Phi_0 + \hbar \, \Phi) : = \frac{ \C \Q [\hbar]}{ \Big( \hbar^2, \, \partial_a \big( \Phi_0 + \hbar \Phi \big) \, | \, a \in \Q_1 \Big)}.
\end{equation}
The second cohomology class corresponding to $\Phi$ under the composition
\begin{equation} \label{def comp}
\begin{tikzcd}
HH_0(\C \Q) \arrow[r, twoheadrightarrow] & HH_0(\J) \arrow[r, "\D"] & HH^3(\J) \arrow[r, "\Delta"] & HH^2(\J),
\end{tikzcd}
\end{equation}
is precisely that of $J(\Q, \Phi_0 + \hbar \, \Phi)$ (\cite{EGP} Proposition 2.1.5).

\begin{prop}
Suppose $\Q$ is a zigzag consistent dimer in a torus. The first order Calabi--Yau deformations of the Jacobi algebra are of the form \eqref{CY def} where $\Phi \in HH_0(\C \Q)$ is a linear combination of cycles of the following types:
\begin{enumerate}
\item a multiple of a boundary cycle, depending up to gauge equivalence on the homology and perfect matching degrees;
\item a minimal closed cycle of homology in $\Int \sigma_i$ for some $1 \leq i \leq N$, depending up to gauge equivalence only on the homology class;
\item $\st^+(Z_{i, j})^n$, where $1 \leq i \leq N$, $1 \leq j \leq m_i$, and $n \in \z_{> 0}$.  
\end{enumerate}
\end{prop}

\begin{proof}
The classes in $HH^2(\J)$ corresponding to Calabi--Yau deformations are those in the subspace (\cite{EGP} \S 2)
\[
Ker( \Delta: HH^2(\J) \rightarrow HH^1(\J).
\] 
Exactness of \eqref{exact sequence} means this subspace equals
\[
Im(\Delta: HH^3(\J) \rightarrow HH^2(\J),
\]
so in fact every first order Calabi--Yau deformation is a deformation of the superpotential. Our computation of $HH_0(\J)$ (Theorem \ref{HH_0 comp}) then determines the gauge equivalence classes of the deformations. A cycle in $\C \Q$ of the first type maps to $f \ell [v] \in HH_0(\J)$ for some homogeneous $f \in \zc$ and vertex $v$; the class in $HH_0(\J)$ depends only on the homology and perfect matching degrees of $f$. Similarly, a cycle of the second type maps to $x_{\alpha} [v] \in HH_0(\J)$ where $\alpha \in \Int \sigma_i$ is its homology and $v$ is any vertex; the class depends on the homology. The cycle $\st^+(Z_{i, j})^n$ maps to $x_{\eta_i}^n [v] \in HH_0(\J)$ with $v \in \E_{i,j}$, which depends on the homology, $n \in \z_{>0}$, and the strip $\E_{i, j}$. 
\end{proof}

For example, consider the antizigzag cycle $\st^+(Z_{i,j})$, corresponding under the composition \eqref{def comp} to  $\psi_{i, j} \in HH^2(\J)$. The associated first order deformation is, up to gauge equivalence,
\[
J(\Q, \Phi_0 + \hbar \st^+(Z_{i, j})) = \frac{ \C \Q [\hbar]}{ \Big( \hbar^2, \, \partial_a \big( \Phi_0 + \hbar \st^+(Z_{i, j}) \big) \, | \, a \in \Q_1 \Big)},
\] 
If $a$ is an arrow not in $\st^+(Z_{i, j})$, then the relation \eqref{relations} is unchanged from before, 
\[
\partial_a \big( \Phi_0 + \hbar \st^+(Z_{i, j}) \big) = R_a^+ - R_a^-.
\]
However, if $a \in \st^+(Z_{i, j})$, then the relation is perturbed by the path that completes $a$ to the antizigzag: if 
\[
a a_1 \dots a_m, \; \; a_k \in \Q_1
\]
represents the cycle $\st^+(Z_{i, j}) \in HH_0(\C \Q)$ at $t(a) = h(a_m)$, then
\[
\partial_a \big( \Phi_0 + \hbar \st^+(Z_{i, j}) \big) = R_a^+ - R_a^- + \hbar a_1 \dots a_m.
\]
 
\subsection{Batalin--Vilkovisky structure} \label{section Batalin--Vilkovisky structure}
 
In the following, $\Q$ is assumed to be a zigzag consistent dimer embedded in a torus.
 
\subsubsection{BV operator}
 
First, we determine how the BV operator and the Gerstenhaber bracket evaluate on $HH^1(\J)$.

\begin{prop} \label{Euler BV}
Suppose $\Q$ is a zigzag consistent dimer in a torus. For all $H_1(\Sigma, \z) \times \z^{PM(\Q)}$-homogeneous $f \in \zc$, $1 \leq i, j \leq N$, $\alpha \in \Int \sigma_i$, and $\beta \in \Int \sigma_j$,
\begin{enumerate}
\item $\Delta (f \partial_i) = \big( 1 + \degree_{\Perf_i}(f) \big) \, f $;
\item $\Delta(\partial_{\alpha}) = 0$;
\item $\{ \partial_i, f \} = \degree_{\Perf_i}(f) \, f$;
\item $\{ \partial_{\alpha}, f \} = \degree_{\Perf_i} (f) x_{\alpha} \ell^{-1} f$;
\item $\{ \partial_i, \partial_j \} = 0$;
\item $\{\partial_i, \partial_{\beta} \} =  (\degree_{\Perf_i} (x_{\beta}) - 1) \partial_{\alpha}$;
\item
\begin{align*}
\{ \partial_{\alpha}, \partial_{\beta} \} & = \\ 
& \begin{cases}
0 & \text{if} \; \degree_{\Perf_j}(x_{\alpha}) = 0 \; \text{or} \; \degree_{\Perf_i}(x_{\beta}) = 0  \\
( \degree_{\Perf_i}(x_{\beta}) - 1) \partial_{\alpha + \beta} & \text{if} \; \degree_{\Perf_j}(x_{\alpha}) = 1 \; \text{and} \; \degree_{\Perf_i}(x_{\beta}) \geq 1 \\
(1 - \degree_{\Perf_j}(x_{\alpha})) \partial_{\alpha + \beta} & \text{if} \; \degree_{\Perf_j}(x_{\alpha}) \geq 1 \; \text{and} \; \degree_{\Perf_i}(x_{\beta}) = 1 \\
x_{\alpha + \beta} \ell^{\degree_{\Perf_k}(x_{\alpha}) + \degree_{\Perf_k}(x_{\beta}) - 2} & \text{if} \; \alpha + \beta \in \sigma(\Perf_k), \degree_{\Perf_j}(x_{\alpha}) \geq 2, \; \text{and} \\
\qquad \cdot \big( (\degree_{\Perf_i}(x_{\beta}) - 1) \partial_j + (1 - \degree_{\Perf_j}(x_{\beta})) \partial_i \big) & \; \degree_{\Perf_i}(x_{\beta}) \geq 2
\end{cases} 
\end{align*}
\end{enumerate}
\end{prop}
 
\begin{proof}
As the map
\[
L^*: HH^*(\J) \rightarrow HH^*(\JW)
\]
is a morphism of BV algebras (Theorem \ref{loc of cohomology}) that is injective in degrees $0$ and $1$ (Theorems \ref{HH^1 free}), computations can be done faithfully in $HH^*(\JW)$. Consider the isomorphism
\[
\zeta: = cotr^* \Psi^*: (HH^*(\JW), \Delta_{L_*(\pi_0)} \big) \rightarrow \big( \C[x^{\pm 1}, y^{\pm 1}, z^{\pm 1}, \partial_x, \partial_y, \partial_z], \Div_0 \big). 
\]
Recall that the map $\Psi$ (Theorem \ref{matrix algebra}) was defined for a choice of basepoint $v_o$ and a perfect matching $\Perf_o$. Let $p_x$ and $p_y$ be paths in $v_o \JW v_o$ that represent the homology classes $x$ and $y$ and have degree $0$ in $\Perf_o$. The element $\partial_i$ is sent to the Euler vector field weighted by the cohomology class $(n_x, n_y) : = ( \degree_{\Perf_i}(p_x), \degree_{\Perf_i}(p_y))$ of the cocycle $\Perf_i - \Perf_o$,
\[
\zeta( \partial_i) = n_x \, x \partial_x + n_y \, y \partial_y + z \partial_z.
\]
Since
\[
\Div_0(n_x x \partial_x + n_y y \partial_y + z \partial_z) = 1,
\]
we must have $\Delta_{L_*(\pi_0)}(\partial_i) = 1$. A similar computation shows $\Delta_{L_*(\pi_0)}(\partial_{\alpha}) = 0$. (Alternatively, formula (2) can be deduced from the fact that $\Delta$ preserves the $H_1(\Sigma, \z) \times \z^{PM(\Q)}$-grading and $HH^0(\J) \cong \zc$ has no element of degree $-1$ in any $\Perf_i$.)

Formulas (3) and (4) follow immediately from the definition of the Gerstenhaber bracket. Then identity \eqref{BVGers} can be used to obtain formula (1) for general homogeneous coefficient $f \in \zc$.

On the cochain level, for homogeneous $f, g \in \zcl$ and $\Perf, \Perf' \in PM(\Q)$, we have the general formula
\[
\{ f E_{\Perf}, g E_{\Perf'} \} = \degree_{\Perf}(g) fg E_{\Perf'} - \degree_{\Perf'}(f) fg E_{\Perf}
\]
from which formulas (5), (6), and (7) are easily obtained.
\end{proof}

Proposition \ref{Euler BV} and the identities \eqref{Leibniz} and \eqref{BVGers} determine $\Delta$ on the parts of $HH^2(\J)$ and $HH^3(\J)$ that are generated by $HH^1(\J)$ and $HH^0(\J)$. As for the remainder, the exact sequence \eqref{exact sequence} and formula \eqref{psi formula} give 
\[
\Delta(x_{\eta_i}^n \theta_v ) =
\begin{cases}
0 & \text{if} \; n = 0 \\
n x_{\eta_i}^{n-1} \psi_{i, j} & \text{if} \; n \geq 1.  
\end{cases}
\]
The fact that $\Delta^2 = 0$ implies 
\[
\Delta(x_{\eta_i}^n \psi_{i, j} ) = 0,
\]
completing the computation of the BV operator.


\subsubsection{Products $HH^0(\J) \cup HH^1(\J) \rightarrow HH^1(\J)$} 

We have only to describe the action of $\zc$ on the subspace of $HH^1(\J)$ complementary to $\zc \cup \No$, which is spanned by the $\partial_{\alpha}$. Let $\alpha \in \Int \sigma_i$ and $\beta \in H_1(\Sigma, \z) \setminus \{ 0 \}$, and suppose $\alpha + \beta \in \sigma_j$. Observe
\[
x_{\beta} x_{\alpha} \ell^{-1} E_{\Perf_i} = x_{\alpha + \beta} \ell^{\degree_{\Perf_j}(x_{\alpha}) + \degree_{\Perf_j}(x_{\beta}) - 1} E_{\Perf_i},
\]
and so
\[
x_\beta \cdot \partial_\alpha =
\begin{cases}
\partial_{\alpha + \beta} & \text{if} \; \beta \in \sigma_i \\
x_{\alpha + \beta} \ell^{\degree_{\Perf_j}(x_{\alpha}) + \degree_{\Perf_j}(x_{\beta}) - 1} \partial_i & \text{otherwise}.
\end{cases}
\]
Moreover, it is easily checked that
\[
\ell \cdot \partial_{\alpha} = x_{\alpha} \cdot \partial_i.
\]

\subsubsection{Products $HH^0(\J) \cup HH^2(\J) \rightarrow HH^3(\J)$}

After the preceeding calculations, we have only to describe the action of $\zc$ on the subspace of $HH^2(\J)$ complementary to $\zc \cup \No \cup \No$ and $\partial_{\alpha} \cup \No$. By Lemma \ref{BV on theta},
\[
f \cdot \psi_{i, j} = \Delta( f \, x_{\eta_i} \theta_v) - x_{\eta_i} \Delta( f \theta_v)
\]
where $f \in \zc$ is a homogeneous central element and $v \in \E_{i, j}$. The formulas for $\Delta$ and the identities \eqref{theta coeff} can be used to write the expression for various $f$ explicitly in terms of our basis.

\subsubsection{Products $HH^1(\J) \cup HH^1(\J) \rightarrow HH^2(\J)$}

We only need to calculate products of the form $\partial_{\alpha} \cup \partial_{\beta}$ in terms of our basis for $HH^2(\J)$ in Theorem \ref{HH^2 comp}.

\begin{prop}
Suppose $\Q$ is a zigzag consistent dimer in a torus. For $\alpha \in \Int \sigma_i$ and $\beta \in \Int \sigma_j$ with $i \leq j$,
\[
\partial_{\alpha} \cup \partial_{\beta} = 
\begin{cases}
0 & \text{if} \; \degree_{\Perf_j}(x_{\alpha}) = \degree_{\Perf_i}(x_{\beta}) = 0 \\
x_{\eta_{i+1}}^{n-1} \psi_{i+1, 1} & \text{if} \; \degree_{\Perf_{j}}(x_{\alpha}) = \degree_{\Perf_{i}}(x_{\beta}) = 1, \alpha + \beta \in \langle \eta_{i+1} \rangle \\
\partial_i \cup \partial_{\alpha + \beta} & \text{if} \; \degree_{\Perf_j}(x_{\alpha}) = 1, \degree_{\Perf_i}(x_{\beta}) > 1 \\
x_{\alpha + \beta} \ell^{\degree_{\Perf_k}(x_{\alpha}) + \degree_{\Perf_k}(x_{\beta}) - 2} \partial_i \cup \partial_j & \text{if} \; \degree_{\Perf_j}(x_{\alpha}), \degree_{\Perf_i}(x_{\beta}) > 1, \alpha + \beta \in \sigma_k
\end{cases}
\]
\end{prop}

\begin{proof}
On the cochain level, the product $\partial_{\alpha} \cup \partial_{\beta}$ is represented by
\[
x_{\alpha} x_{\beta} \ell^{-2} E_{\Perf_i} \cup E_{\Perf_j}.
\]

\noindent \underline{Case $1$:} $\degree_{\Perf_j}(x_{\alpha}) = \degree_{\Perf_i}(x_{\beta}) = 0$. \\

\noindent The element $\partial_{\alpha} \cup \partial_{\beta}$ has degree $-2$ in $\Perf_i = \Perf_j$ and so must be trivial. \\

\noindent \underline{Case $2$:} $\degree_{\Perf_j}(x_{\alpha}) = \degree_{\Perf_i}(x_{\beta}) = 1$. \\

\noindent Then $\degree_{\Perf_i} (x_{\alpha} x_{\beta}) = \degree_{\Perf_j}(x_{\alpha} x_{\beta}) = 1$, so
\[
x_{\alpha} x_{\beta} = x_{\alpha + \beta} \ell
\]
and $x_{\alpha + \beta}$ is killed by both $\Perf_i$ and $\Perf_j$. Hence, they must be consecutive corner matchings in $MP(\Q)$: $j = i+1$ and $\alpha + \beta = n \cdot \eta_{i+1}$ for some $n \geq 1$. Observe that
\[
\degree_{\Perf_{j}}(x_{\alpha}), \degree_{\Perf_i}(x_{\beta}) \geq m_{i+1},
\]
since a closed path of homology $\alpha$ or $\beta$ must intersect each zigzag cycle of homology $-\eta_{i+1}$. So the assumption that both degrees equal $1$ implies $m_{i+1} = 1$, i.e., there is only one zigzag cycle of homology $-\eta_{i+1}$. The cocycle
\[
x_{\alpha} x_{\beta} \ell^{-2} E_{\Perf_i} \cup E_{\Perf_j} = x_{\eta_{i+1}}^n \ell^{-1} E_{\Perf_i} \cup E_{\Perf_{i+1}},
\]
represents the cohomology class $x_{\eta_{i+1}}^{n-1} \psi_{i+1, 1}$. \\

\noindent \underline{Case $3$:} $\degree_{\Perf_j}(x_{\alpha}) = 1$, $\degree_{\Perf_i}(x_{\beta}) > 1$. \\

\noindent As in the preceding case, $\Perf_j$ kills the element $x_{\alpha + \beta}$, but in this instance, it is the unique perfect matching that does so. Thus, $\alpha + \beta \in \Int \sigma_j$ and 
\[
x_{\alpha} x_{\beta} \ell^{-2} E_{\Perf_i} \cup E_{\Perf_j} = x_{\alpha + \beta} \ell^{-1} E_{\Perf_i} \cup E_{\Perf_j},
\]
representing the cohomology class $\partial_i \cup \partial_{\alpha + \beta}$. \\

\noindent \underline{Case $4$:} $\degree_{\Perf_j}(x_{\alpha}), \degree_{\Perf_i}(x_{\beta}) > 1$. \\

\noindent The product $x_{\alpha} x_{\beta}$ is a multiple of at least $\ell^2$, so
\[
\partial_{\alpha} \cup \partial_{\beta} = x_{\alpha + \beta} \ell^{\degree_{\Perf_k}(x_{\alpha}) + \degree_{\Perf_k}(x_{\beta}) - 2} \partial_i \cup \partial_j
\]
where $\alpha + \beta \in \sigma_k$. 
\end{proof}

\subsubsection{Products $HH^1(\J) \cup HH^2(\J) \rightarrow HH^3(\J)$}

The previous computations inform us how to take products of $HH^1(\J)$ with
\[
\zc \cup \No \cup \No, \; \; \partial_{\alpha} \cup \No. 
\]
It is left to calculate products with the elements $\psi_{i, j}$.

\begin{prop}
For $1 \leq i, k \leq N$ and $1 \leq j \leq m_i$,
\[
\partial_k \cup \psi_{i, j} = \degree_{\Perf_k} (x_{\eta_i}) x_{\eta_i} \theta_v
\]
where $v \in \E_{i,j}$. Moreover, if $\alpha \in \Int \sigma_k$, 
\[
\partial_{\alpha} \cup \psi_{i, j} = \degree_{\Perf_k}(x_{\eta_i}) x_{\alpha + \eta_i} \ell^{\degree_{\Perf_r}(x_{\alpha}) + \degree_{\Perf_r}(x_{\eta_i}) - 1} \theta_v
\]
where $v \in \E_{i, j}$ and $\alpha + \eta_i \in \sigma_r$.
\end{prop}

\begin{proof}
Under Van den Bergh duality, the product 
\[
\partial_k \cup \psi_{i, j} = \partial_k \cup \Delta(x_{\eta_i} \theta_v) 
\]
is sent to
\[
\partial_k \cap B([x_{\eta_i} \cdot v]) = \partial_k \cap [1 \otimes_{\bk} x_{\eta_i} \cdot v] = \degree_{\Perf_k}(x_{\eta_i}) [x_{\eta_i} \cdot v].
\]
The inverse of $\D$ maps this to
\[
\degree_{\Perf_k}(x_{\eta_i}) x_{\eta_i} \theta_v.
\]
Replacing $\partial_k$ with $\partial_{\alpha}$, we compute
\[
\partial_\alpha \cap [1 \otimes_{\bk} x_{\eta_i} \cdot v] = \degree_{\Perf_k}(x_{\eta_i}) [x_{\alpha} x_{\eta_i} \ell^{-1} \cdot v ]
= \degree_{\Perf_k}(x_{\eta_i}) [x_{\alpha + \eta_i} \ell^{\degree_{\Perf_r}(x_{\alpha}) + \degree_{\Perf_r}(x_{\eta_i})-1} \cdot v],
\] 
assuming $\alpha + \eta_i \in \sigma_r$. The inverse of $\D$ maps this to 
\[
\degree_{\Perf_k}(x_{\eta_i}) x_{\alpha + \eta_i} \ell^{\degree_{\Perf_r}(x_{\alpha}) + \degree_{\Perf_r}(x_{\eta_i})-1} \theta_v.
\]
\end{proof}

\subsection{Example: mirror to four punctured sphere} 

\begin{figure}[h]
\centering
\begin{subfigure}{.4\textwidth}
\centering
\[
\begin{tikzpicture}
\node at (0,0) {$v$} ;
\draw (0, 0) circle [radius=0.2] ;
\node at (4,0) {$v$} ;
\draw (4, 0) circle [radius=0.2] ;
\node at (0, 4) {$v$} ;
\draw (0, 4) circle [radius=0.2] ;
\node at (4, 4) {$v$} ;
\draw (4, 4) circle [radius=0.2] ;
\node at (2, 2) {$w$} ;
\draw (2, 2) circle [radius = 0.2] ;
\draw[dotted, ->] (0.2, 0) -- (3.8, 0) ;
\draw[dotted, ->] (4, 0.2) -- (4, 3.8 );
\draw[dotted, ->] (0.2, 4) -- (3.8, 4) ;
\draw[dotted, ->] (0, 0.2) -- (0, 3.8) ;
\draw[->] (3.86, 0.14) -- (2.14, 1.86) ;
\node[below left] at (3, 1) {$a_1$} ;
\draw[->] (1.86, 1.86) -- (0.14, 0.14) ;
\node[below right] at (1, 1) {$b_1$} ;
\draw[->] (0.14, 3.86) -- (1.86, 2.14) ;
\node[above right] at (1, 3) {$a_2$} ;
\draw[->] (2.14, 2.14) -- (3.86, 3.86) ;
\node[above left] at (3, 3) {$b_2$} ;
\node at (3.25, 2) {$-$} ;
\node at (2, 1) {$+$} ;
\node at (0.75, 2) {$-$} ;
\node at (2, 3) {$+$} ;
\end{tikzpicture}
\]
\caption{$\Q$}
\end{subfigure}
\begin{subfigure}{.4\textwidth}
\centering
\[
\begin{tikzpicture}
\draw[black, thick] (0,0) -- (0,2);
\draw[black, thick] (0,2) -- (2,2);
\draw[black, thick] (2,2) -- (2, 0);
\draw[black, thick] (2,0) -- (0,0);
\draw[black, thick, ->] (0, 1) -- (-1, 1);
\draw[black, thick, ->] (1, 2) -- (1, 3);
\draw[black, thick, ->] (2, 1) -- (3, 1);
\draw[black, thick, ->] (1, 0) -- (1, -1);
\node[left] at (1, -0.5) {$- \eta_1$ };
\node[above] at (2.5, 1) {$ - \eta_2$ } ;
\node[right] at (1, 2.5) {$ - \eta_3$ } ;
\node[above] at (-0.5, 1) {$-\eta_4$} ;
\filldraw[black] (0,0) circle (2pt) node[anchor=east] {$ \{b_1\} = \Perf_4$}; 
\filldraw[black] (0,2) circle (2pt) node[anchor=east] {$\{ a_1\} = \Perf_3 $};
\filldraw[black] (2,2) circle (2pt) node[anchor=west] {$\Perf_2 = \{b_2\}$};
\filldraw[black] (2,0) circle (2pt) node[anchor=west] {$\Perf_1 = \{a_2\}$};
\end{tikzpicture}
\]
\caption{$MP(\Q)$}
\end{subfigure}
\begin{subfigure}{.4\textwidth}
\centering
\[
\begin{tikzpicture}
\draw[->, thick] (0, 0) -- (1.5, 0);
\node[right] at (1.5, 0) {$\langle \eta_4 \rangle$};
\node at (1, 1) {$\sigma_4$};
\draw[->, thick] (0, 0) -- (0, 1.5);
\node[above] at (0, 1.5) {$\langle \eta_1 \rangle$};
\node at (-1, 1) {$\sigma_1$};
\draw[->, thick] (0, 0) -- (-1.5, 0);
\node[left] at (-1.5, 0) {$\langle \eta_2 \rangle$};
\node at (-1, -1) {$\sigma_2$};
\draw[->, thick] (0, 0) -- (0, -1.5);
\node[below] at (0, -1.5) {$\langle \eta_3 \rangle$};
\node at (1, -1) {$\sigma_3$};
\end{tikzpicture}
\]
\caption{The antizigzag fan}
\end{subfigure}
\caption{\label{4 punctured sphere} Mirror dual to the four punctured sphere.}
\end{figure}
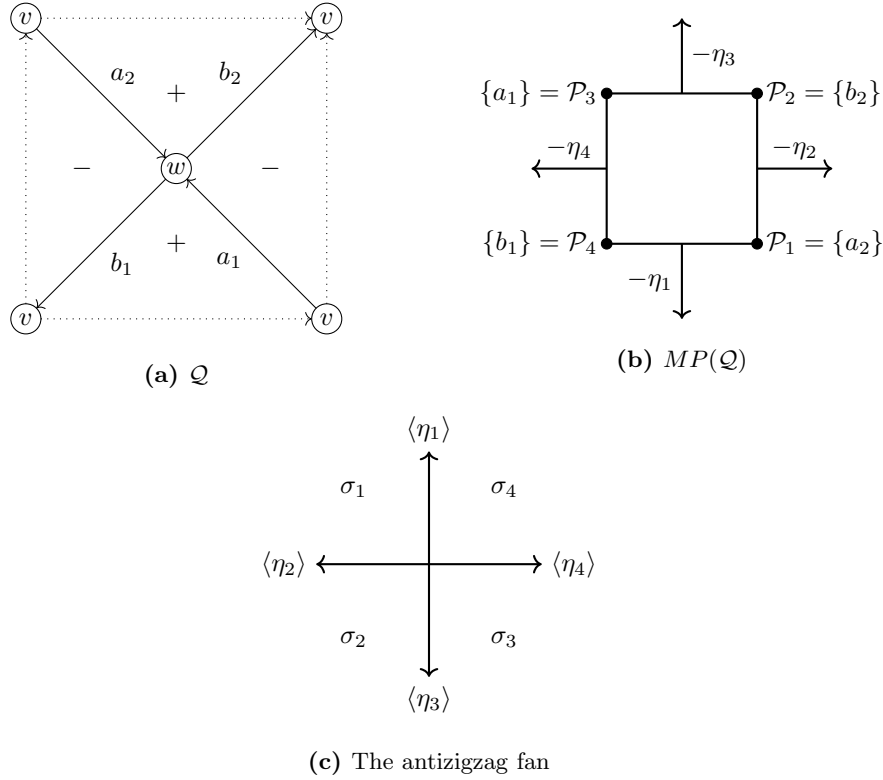

Consider the zigzag consistent dimer in a torus illustrated in Figure \ref{4 punctured sphere}. There are four zigzag cycles, which coincide with the antizigzag cycles, represented by
\[
a_1 b_1, \; a_2 b_1, \; a_2 b_2, \; a_1 b_2.
\]
There are four perfect matchings, one for each arrow. The dimer dual has genus $0$ and four vertices, determining the sphere with four punctures.  

The minimal central elements corresponding to the antizigzags are
\begin{eqnarray*}
x_{\eta_1} & = & a_1 b_2 + b_2 a_1 \\
x_{\eta_2} & = & a_1 b_1 + b_1 a_1 \\
x_{\eta_3} & = & a_2 b_1 + b_1 a_2 \\
x_{\eta_4} & = & a_2 b_2 + b_2 a_2.
\end{eqnarray*}
In fact, they generate the entire center, subject to the single relation $\ell = x_{\eta_1} x_{\eta_3} = x_{\eta_2} x_{\eta_4}$:
\[
HH^0(\J) \cong \zc \cong \C[x_{\eta_1}, x_{\eta_2}, x_{\eta_3}, x_{\eta_4}] / (x_{\eta_1} x_{\eta_3} - x_{\eta_2} x_{\eta_4}).
\]
By Theorem \ref{HH^1 free}, the first Hochschild cohomology is
\[
HH^1(\J) = \zc \cup \No \; \oplus \; \bigoplus_{\substack{i \in \z / 4 \z \\ m, n > 0}} \C \cdot \partial_{m \eta_i + n \eta_{i+1}}.
\]
No two zigzag cycles are parallel, so in second cohomology, there is only one element $\psi_i : = \psi_{i, 1}$ for each $1 \leq i \leq 4$, represented by the cocycle 
\[
x_{\eta_i} \ell^{-1} E_{\Perf_{i-1}} \cup E_{\Perf_i} = x_{\eta_{i+2}}^{-1} E_{\Perf_{i-1}} \cup E_{\Perf_i} . 
\]  
Then according to Theorem \ref{HH^2 comp}, 
\[
HH^2(\J) = \zc \cup \No \cup \No \, \oplus \, \bigoplus_{\substack{i \in \z / 4 \z \\ m, n > 0}} \C \cdot \partial_{m \eta_i + n \eta_{i+1}} \cup \No \, \oplus \, \C \{ x_{\eta_i}^n \psi_i \, | \, i \in \z / 4 \z, \, n \geq 0 \}.
\]
Finally, in the presentation of Theorem \ref{HH^3 second}, the third cohomology is 
\begin{align*}
HH^3(\J) \cong & \, \zc \cup \No \cup \No \cup \No \, \oplus \, \bigoplus_{\substack{i \in \z / 4 \z \\ m, n > 0}} \C \cdot \partial_{m \eta_i + n \eta_{i+1}} \cup \No \cup \No \\ 
& \oplus \, \C \{ x_{\eta_i}^n \theta_v \, | \, i \in \z / 4 \z, \, n \geq 1 \} \, \oplus \, \C \{ \theta_v, \theta_w \}.
\end{align*}
This can be written more simply in the form of Corollary \ref{HH^3 first}, 
\[
HH^3(\J) \cong \zc \cdot \theta_v \, \oplus \, \C \cdot \theta_w
\]
where $f \cdot \theta_w = f \cdot \theta_v$ for any $H_1(\Sigma, \z) \times \z^{PM(\Q)}$ homogeneous $f \in \zc \setminus \C$. The kernel of the universal map
\[
L^*: HH^*(\J) \rightarrow HH^*(\JW)
\]
is the torsion
\[
tor_{\ell}(HH^*(\J)) = \C \cdot ( \theta_v - \theta_w).
\]

\section{Hochschild cohomology of the category of matrix factorizations}

According to Theorem \ref{compact type}, the compactly supported Hochschild cohomology and the Borel--Moore Hochschild homology of $\M$ are isomorphic to those of the curved algebra $(\J, \ell)$. We compute them by the same spectral sequence used by C\u{a}ld\u{a}raru--Tu \cite{CaldararuTu}, who address the case of a Landau-Ginzburg model describing an isolated hypersurface singularity. It is in the arguments for degeneration of the spectral sequence that properties specific to the Jacobi algebra are needed. Throughout, $\Q$ is assumed to be a zigzag consistent dimer that admits a perfect matching.

\subsection{Borel--Moore homology and compactly supported cohomology} 

Let $A = \J$. A double complex supported above the diagonal is obtained by letting 
\[ 
C_{i, j} = \begin{cases}
A \otimes A^{\otimes (j-i)} & \text{if} \; j \geq i \\
0 & \text{otherwise},
\end{cases}
\] 
equipped with $d_A$ as the vertical differential and $\Lie_{\ell} = d_{\ell}$ as the horizontal differential \eqref{chain}. 
\[
\begin{tikzcd}
& \vdots \arrow[d] & \vdots \arrow[d] & \vdots \arrow[d] \\
\dots & A \otimes A^{\otimes 3} \arrow[l, "d_{\ell}"] \arrow[d, "d_A"] & A \otimes A^{\otimes 2} \arrow[l, "d_{\ell}"]\arrow[d, "d_A"]  & A \otimes A \arrow[l, "d_{\ell}"] \arrow[d, "d_A"] & A \arrow[l, "d_{\ell}"]  \arrow[d] & \dots \arrow[l] \\
\dots & A \otimes A^{\otimes 2} \arrow[l, "d_{\ell}"]  \arrow[d, "d_A"] & A \otimes A  \arrow[l, "d_{\ell}"] \arrow[d, "d_A"] & A  \arrow[l, "d_{\ell}"] \arrow[d] & 0 \arrow[l] \arrow[d] & \dots \arrow[l]  \\
\dots & A \otimes A  \arrow[l, "d_{\ell}"] \arrow[d, "d_A"] & A  \arrow[l, "d_{\ell}"] \arrow[d]& 0 \arrow[l] \arrow[d] & \dots \arrow[l] \\
\dots & A  \arrow[l, "d_{\ell}"] \arrow[d] & 0 \arrow[l] \arrow[d]& 0 \arrow[l] \arrow[d] \\
& \vdots & \vdots & \vdots
\end{tikzcd}
\]
Note that $C_{*, *}$ is $2$-periodic along the $i = j$ diagonal. Modulo $2$, each homogeneous subspace of the direct product totalization $\Tprod(C_{*,*})$ coincides modulo $2$ with the homogeneous subspace of $C_*^{BM}(A)$ of the same parity. Therefore,
\[
H^* (\Tprod(C_{*,*})) = H_*^{BM}(A, \ell) \; * \, \text{mod} \, 2.
\]

The periodicity can be leveraged to reduce the computation essentially to the first quadrant. Let $C_{*, *}^+$ be the truncation
\[
C_{i, j}^+ = \begin{cases}
A \otimes A^{\otimes (j-i)} & \text{if} \; j \geq i \geq 0 \\
0 & \text{otherwise}.
\end{cases}
\]
For $r \in \mathbb{N}$, let $C_{*, *}^+[2r]$ denote the complex shifted by $2r$ along the $i = j$ diagonal (in the direction of the third quadrant). If $s > r$, then $C_{*, *}^+[2r]$ is a quotient of $C_{*, *}^+[2s]$ by the subcomplex consisting of terms $C_{i, j}$ for which $-2s \leq i < -2r$ or $-2s \leq j < -2r$. Thus, there are quotient maps on the totalizations
\[
\Tot(C_{*,*}^+)[2s] \rightarrow \Tot(C_{*, *}^+)[2r], \; \; \; s > r \geq 0.
\]
Here, we ignore the distinction between direct product and direct sum totalizations since they coincide for the truncated complexes. The inverse system $\{\Tot(C_{*,*}^+)[2r] \, | \, r \in \mathbb{N}\}$ has limit $\Tprod(C_{*, *})$, and because it satisfies the Mittag-Leffler condition (\cite{Weibel} Theorem 3.5.8), there is an exact sequence
\begin{equation} \label{spectral exact sequence}
0 \rightarrow \varprojlim\nolimits^1 H^{i+1}(\Tot(C_{*,*}^+)[2r]) \rightarrow H^i(\Tprod(C_{*,*})) \rightarrow \varprojlim H^i(\Tot(C_{*,*}^+)[2r]) \rightarrow 0.  
\end{equation}
for all $i \in \z$. The symbol $\varprojlim\nolimits^{1}$ denotes the first derived functor of the inverse limit. Since $H^i(\Tot(C_{*,*}^+)[2r]) = H^{i + 2r}(\Tot(C_{*,*}^+))$, the Borel--Moore Hochschild homology is determined from \eqref{spectral exact sequence} by the first quadrant complex.




\begin{lem} \label{cokernel bid}
Suppose $\Q$ is a zigzag consistent dimer embedded in $\Sigma$ and admitting a perfect matching.
\begin{enumerate}
\item If $\Sigma$ has genus $g > 1$ , then 
\[
HH^0(\J) \, / \{ \ell, HH^1(\J) \} \cong \C.
\]
\item If $\Sigma$ has genus $g = 1$, then
\[
HH^0(\J) \, / \{ \ell, HH^1(\J) \} \cong \C[x_{\eta_1}, \dots, x_{\eta_{N}}] / \big( x_{\eta_i}x_{\eta_j} \, | \, i \neq j \big).
\]
\end{enumerate}
\end{lem}

\begin{proof}
For any $\Perf \in PM(\Q)$ and central element $f$,
\begin{equation} \label{image formula}
\{ \ell, f E_{\Perf} \} = f \ell.
\end{equation}
If $\Sigma$ has genus $g > 1$, the center is $\C [\ell]$ (Proposition \ref{higher genus center}). The above formula implies that 
\[
\ell \cdot \C[\ell] \subset \{\ell,  HH^1(\J) \}.
\]
On the other hand, the boundary cycles of $\Q$ have path length at least $3$, so the image of $\ell$ under any derivation of $\J$ must have filtered degree at least $2$ in the filtration of $\J$ by path length. Therefore, the reverse inclusion holds, meaning 
\[
HH^0(\J) \, / \{ \ell, HH^1(\J) \} = \C[ \ell] /\,  \ell \, \cdot \C[ \ell ] \cong \C.
\]

If $\Sigma$ has genus $g = 1$, Theorem \ref{HH^1 free} and formula \eqref{image formula} imply that $\ell$ and the minimal elements 
\[ 
x_{\alpha}, \; \; \alpha \in \Int \sigma_i, \; \; 1 \leq i \leq N
\]
generate the image $\{ \ell, HH^1(\J) \}$. Hence, the quotient is generated by the powers of the minimal elements associated to the antizigzag cycles.
\end{proof}

\begin{prop} \label{spectral}
Suppose $\Q$ is a zigzag consistent dimer admitting a perfect matching. There is an isomorphism of $\Z$-graded vector spaces
\begin{align} \label{final complex}
& HH_*^{BM} \big( \M \big) \cong H^*\big( HH_*(\J), \, \Lie_{\ell} \big) \\
& HH^*_c \big( \M \big) \cong  H^*\big( HH^*(\J), \, \{\ell, -\} \big. \nonumber
\end{align}
\end{prop}

\begin{proof}
Let $E^*_{*,*}$ be the homological spectral sequence for the first quadrant double complex $C_{*, *}^+$ with respect to the vertical filtration. The entries of $E^1_{*, *}$ are the Hochschild homology groups of $\J$, and the differential is $\Lie_\ell$. 
\[
\begin{tikzcd}[row sep = small]
& \vdots & \vdots & \vdots & \vdots & \\
0 & HH_3(\J) \arrow[l] & HH_2(\J) \arrow[l, "\Lie_{\ell}"] & HH_1(\J) \arrow[l, "\Lie_{\ell}"] & HH_0(\J) \arrow[l, "\Lie_{\ell}"]  & \dots \arrow[l] \\
0 & HH_2(\J)  \arrow[l] & HH_1(\J)  \arrow[l, "\Lie_{\ell}"] & HH_0(\J)  \arrow[l, "\Lie_{\ell}"] & 0 \arrow[l]  & \dots \arrow[l]  \\
0 & HH_1(\J)  \arrow[l] & HH_0(\J)  \arrow[l, "\Lie_{\ell}"] & 0 \arrow[l]  & \dots \arrow[l] \\
0 & HH_0(\J) \arrow[l]  & 0 \arrow[l] & \dots \arrow[l] 
\end{tikzcd}
\]
Evidently, the only possible nonzero components of the differential on $E^2_{*,*}$ are
\[
d^2_{i, i}: Ker \big( \Lie_\ell: HH_0(\J) \to HH_1(\J) \big) \rightarrow Coker \big( \Lie_\ell: HH_2(\J) \rightarrow HH_3(\J) \big).
\]
The differentials $d_A$ and $\Lie_\ell$ have degree $0$ in $H_1(\Sigma, \z)$ and degree $1$ in all perfect matchings, so $d^2_{*,*}$ has degree $0$ in $H_1(\Sigma, \z)$ and degree $2$ in all perfect matchings. Consequently, the image of $d^2_{i,i}$ must be concentrated in perfect matching degrees greater than or equal to $2$. However, Van den Bergh duality induces an isomorphism
\[
\overline{\D}: HH^0(\J) / \{\ell, HH^1(\J) \} \longrightarrow HH_3(\J) / \Lie_{\ell} \big( (HH_2(\J) \big).
\]
According to Lemma \ref{degree VDB}, $\overline{\D}$ has degree $0$ in $H_1(\Sigma, \z)$ and degree $1$ in all perfect matchings. We deduce from Lemma \ref{cokernel bid} that the subspace of 
\[
Coker \big( \Lie_\ell: HH_2(\J) \rightarrow HH_3(\J) \big)
\]
lying in perfect matching degrees greater than or equal to $2$ is trivial. Therefore, the differential $d^2_{*, *}$ is $0$, and the spectral sequence degenerates at the second page.

For $r \geq 2$, 
\[
H^{i + 2r}(\Tot(C^+_{*,*})) \cong \begin{cases}
H^{\text{even}}\big(HH_*(\J), \Lie_{\ell} \big) & \text{if} \; i \equiv 0 \; \text{mod} \; 2 \\
H^{\text{odd}}\big(HH_*(\J), \Lie_{\ell} \big) & \text{if} \; i \equiv 1 \; \text{mod} \; 2,
\end{cases}
\]
and so
\begin{equation} \label{inverse limit}
\varprojlim H^*(\Tot(C_{*, *}^+[2r])) = H^*\big( HH_*(\J), \Lie_\ell \big).
\end{equation}
Because the projections
\[ 
H^*(\Tot(C_{*, *}^+[2s])) \rightarrow H^*(\Tot(C_{*, *}^+[2r]))
\]
are isomorphisms for $s > r \geq 2$, the inverse system $\{H^*(\Tot(C_{*, *}^+[2r]) \, | \, r \in \mathbb{N}\}$ satisfies the Mittag-Leffler condition (\cite{Weibel} Theorem 3.5.8), ensuring that
\[ 
\varprojlim\nolimits^1 H^{*}(\Tot(C_{*,*}^+)[2r]) = 0
\]
in all degrees. As a result, \eqref{inverse limit} gives the desired description of the Borel--Moore Hochschild homology of $\M$.

The computation of compactly supported Hochschild cohomology follows more easily. The relevant double complex is 
\[
C^{i,j} = \begin{cases}
\Hom(A^{\otimes (i-j)}, A) & \text{if} \; j \leq i \\
0 & \text{otherwise}. 
\end{cases}
\]
equipped with vertical differential $d_A$ and horizontal differential $\{\ell, -\}$ \eqref{cochain}. Since compactly supported cohomology is a direct sum totalization, there is no need for truncating and taking inverse limits: the spectral sequence with respect to the vertical filtration converges to it. The same argument as above involving the gradings shows that the spectral sequence collapses at the second page.
\end{proof}

We can now use our explicit description of $HH^*(\J)$ in the case of a toric dimer to compute the cohomology in Proposition \ref{spectral}. In fact, we can compute the multiplicative structure on $HH^*_c( \M )$ endowed by the spectral sequence. To establish notation, fix a vertex $v_o$, and let $\Q_0^* = \Q_0 \setminus \{v_o\}$. It may be assumed that, for each homology $\eta_i$, the zigzag cycles are ordered so that $v_o \in \E_{i, 1}$. Then in the cohomology \eqref{final complex}, for each $1 \leq i \leq N$ and $1 < j \leq m_i$, let 
\begin{itemize}
\item $X_i$ denote the class of $x_{\eta_i}$,
\item $\Psi_{i,j}$ denote the class of $\psi_{i, j} - \psi_{i, 1}$,
\item $\Theta_v$ denote the class of $\theta_v - \theta_{v_o}$ where $v \in \Q_0^*$,
\item and $U, V$ respectively denote the classes of
\[
\partial_r - \partial_t, \; \; \partial_s - \partial_t
\]
for fixed $1 \leq r < s < t \leq N$.
\end{itemize}

\begin{thm} \label{compact HH comp}
Suppose $\Q$ is a zigzag consistent dimer a torus. There is an isomorphism of algebras
\[
HH^*_c \big( \M \big) = \frac{\C[X_i, \Psi_{i, j}, U, V, \Theta_v \, | \, 1 \leq i \leq N, \, 1 < j \leq m_i, \, v \in \Q_0^* ]}{(\mathcal{R})}
\]
where $X_i$ and $\Psi_{i, j}$ are even variables; $U$, $V$, and $\Theta_v$ are odd variables; and $\mathcal{R}$ is the vector space generated by
\begin{itemize}
\item $X_i X_j$ for all $i \neq j$;
\item $\Psi_{i, j} \Psi_{k, l}$ for all $i, j, k, l$;
\item $X_i \Psi_{j, k}$ for all $i \neq j$;
\item $(\degree_{\Perf_s}(X_i) - \degree_{\Perf_t}(X_i) ) X_i  \, U -  ( \degree_{\Perf_t}(X_i)  - \degree_{\Perf_r} (X_i) ) X_i \, V$ for all $i$;
\item $X_i \Theta_v $ for all $i$ and $v \in \E_{i, 1}$;
\item $\Psi_{i, j} U - ( \degree_{\Perf_r}(X_i ) - \degree_{\Perf_t}(X_i) ) X_i \, \Theta_v$ for all $i, j$ and $v \in \E_{i, j}$;
\item $\Psi_{i, j} V - (\degree_{\Perf_s}(X_i) - \degree_{\Perf_t}(X_i) ) X_i \, \Theta_v$ for all $i, j$ and $v \in \E_{i, j}$;
\item $\Psi_{i, j} \Theta_v$ for all $i, j$ and $v \in \Q_0^*$;
\item $UV$, $U \Theta_v$, and $V \Theta_v$ for all $v \in \Q_0^*$.
\end{itemize}
\end{thm}

\begin{proof}
To determine $HH^*_c( \M)$ additively, we simply evaluate the differential $\{\ell, -\}$ on each cohomology group $HH^*(\J)$ using the identities \eqref{Leibniz} and the formulas in \S \ref{section Batalin--Vilkovisky structure}. \\
\vspace{0.5em}

\noindent \underline{ $\{\ell, -\}: HH^3(\J) \rightarrow HH^2(\J)$ } \\
\vspace{0.5em}

\noindent For all $1 \leq i < j < k \leq N$, $1 \leq p \leq N$, $f \in \zc$, $\alpha \in \Int \sigma_i$, $n \geq 0$, and $v \in \Q_0$, observe
\begin{align} \label{3 to 2}
& \{ \ell, f \partial_i \cup \partial_j \cup \partial_k \}  = f \ell \big( \partial_i \cup \partial_j - \partial_i \cup \partial_k + \partial_j \cup \partial_k \big) \\
& \{\ell, \partial_{\alpha} \cup \partial_j \cup \partial_k \} = x_{\alpha} \big( \partial_i \cup \partial_j - \partial_i \cup \partial_k + \partial_j \cup \partial_k \big) \nonumber \\
& \{\ell, x_{\eta_r}^n \theta_v\} = \lambda_{i, j, k} x_{\eta_p}^n \big( \partial_i \cup \partial_j - \partial_i \cup \partial_k + \partial_j \cup \partial_k \big). \nonumber
\end{align}
So $\{\ell, -\}$ is injective when restricted to the subspaces
\[
\zc \cup \No \cup \No \cup \No, \; \; \partial_{\alpha} \cup \No \cup \No.
\]
However, the last evaluation in \eqref{3 to 2} is independent of the vertex $v$, so differences of the form
\[
x_{\eta_i}^n (\theta_v - \theta_{v_o}), \; \; n \geq 0
\] 
span the kernel. Hence,
\[
H^3 \big( HH^*(\J), \{\ell, - \} \big) \cong \C \{ \Theta_v  \, | \, v \in \Q_0^* \} \, \oplus \, \C \{ X_i^n \Theta_v \, | \, 1 \leq i \leq N, \, 1 < j \leq m_i, \, n > 0, \, v \in \E_{i, j} \}. 
\] \\

\noindent \underline{ $\{\ell, -\}: HH^2(\J) \rightarrow HH^1(\J)$ } \\
\vspace{0.5em}

\noindent For all $1 \leq i < j < k \leq N$, $1 \leq p \leq N$, $1 \leq q \leq m_p$, $f \in \zc$, $\alpha \in \Int \sigma_i$, and $n \geq 0$, observe
\begin{align} \label{2 to 1}
& \{ \ell, f \partial_i \cup \partial_j \} = f \ell (\partial_j - \partial_i) \\
& \{ \ell, \partial_{\alpha} \cup \partial_j \} = x_{\alpha} (\partial_j - \partial_i) \nonumber \\
& \{ \ell, x_{\eta_p}^n \psi_{p, q} \} =  \lambda_{i, j, k} \big( \degree_{\Perf_j}(x_{\eta_p}) - \degree_{\Perf_k}(x_{\eta_p}) \big) x_{\eta_p}^{n+1} (\partial_i - \partial_k) \nonumber \\
& + \lambda_{i, j, k} \big( \degree_{\Perf_k}(x_{\eta_p}) - \degree_{\Perf_i}(x_{\eta_p})  \big) x_{\eta_p}^{n+1} (\partial_j - \partial_k), \nonumber
\end{align} 
the factor $\lambda_{i, j, k}$ coming from \eqref{theta coeff}. Comparing to \eqref{3 to 2}, we see 
\[
Ker \big( \{\ell, - \}: \zc \cup \No \cup \No \rightarrow \zc \cup \No \big) = \{ \ell, HH^3(\J) \}.
\]
Moreover, the map $\{\ell, -\}$ is injective on the subspace 
\[
\bigoplus_{\substack{\alpha \in \Int \sigma_i \\ 1 \leq i \leq N}} \partial_{\alpha} \cup \No
\] 
The last evaluation in \eqref{2 to 1} is independent of the index $s$, so the kernel of $\{\ell, -\}$ in degree $2$ is spanned by elements of the form 
\[
x_{\eta_p}^n (\psi_{p, q} - \psi_{p, 1}), \; \; n \geq 0.
\]
Therefore,
\[
H^2 \big( HH^*(\J), \{ \ell, - \} \big) \cong \C \{ X_p^n \Psi_{p, q} \, | \, 1 \leq p  \leq N, \, 1 < q \leq m_i, \, n \geq 0 \}. 
\] \\

\noindent \underline{ $\{\ell, -\}: HH^1(\J) \rightarrow HH^0(\J)$ } \\
\vspace{0.5em}

\noindent The cokernel of this map is calculated in Lemma \ref{cokernel bid}:
\[
H^0 \big( HH^*(\J), \{ \ell, - \} \big) \cong \frac{\C [ X_1, \dots, X_N ]}{ (X_i X_j \, | \, i \neq j )}.
\] 
From Proposition \ref{Euler BV}, the kernel is spanned by elements of the form
\[
f (\partial_i - \partial_j)
\]
where $f \in \zc$. The formulas \eqref{2 to 1} indicate that all such elements with $f \in \ell \cdot \zc$ or $f = x_{\alpha}$ for some $\alpha \in \Int \sigma_i$ lie in $\{\ell, HH^2(\J) \}$. So the quotient is spanned by the elements  of the form
\[
x_{\eta_i}^n (\partial_r - \partial_t), \; \; x_{\eta_i}^n (\partial_s - \partial_t), \; \; n \geq 0,
\]  
subject to the relation imposed by the last identity of \eqref{2 to 1}. Hence,
\[
H^1 \big( HH^*(\J), \{\ell, -\} \big) = \frac{\C \{X_i^n U, X_i^n V \, | \, 1 \leq i \leq N, \, n \geq 0 \}}{ \C \{ (\degree_{\Perf_s}(X_i) - \degree_{\Perf_t}(X_i) ) X_i^n  \, U -  ( \degree_{\Perf_t}(X_i)  - \degree_{\Perf_r} (X_i) ) X_i^n \, V \, | \, n \geq 1 \} }
\] \\

We have thus obtained the desired additive description of $HH^*_c(\M)$. The product structure descends from that on $HH^*(\J)$ and is easily deduced from the formulas in \S \ref{section Batalin--Vilkovisky structure}.
\end{proof}

\begin{cor}
Suppose $\Q$ is a zigzag consistent dimer in a torus. There is an additive isomorphism
\[
HH^{BM}_* \big( \M \big)  = \frac{\C[X_i, \Psi_{i, j}, U, V, \Theta_v \, | \, 1 \leq i \leq N, \, 1 < j \leq m_i, \, v \in \Q_0^* ]}{(\mathcal{R})} \, [1]
\]
where $[1]$ denotes the parity shift.
\end{cor}

\begin{proof}
Van den Bergh duality establishes an isomorphism
\[
H^* \big( HH^*(\J), \{\ell, -\} \big) \cong H^{3-*} \big( HH_{3-*}(\J), \Lie_{\ell} \big).
\]
\end{proof}

The compactly supported cohomology actually inherits a BV algebra structure from the complex \eqref{final complex}, since the differential $\{\ell, -\}$ commutes with $\Delta$. Let $\overline{\Delta}$ denote the BV operator on $HH^*_c \big( \M \big)$. From the formulas for $\Delta$ established in \S \ref{section Batalin--Vilkovisky structure}, it is straightforward to evaluate $\overline{\Delta}$ on the basis in Theorem \ref{compact HH comp}. 

\begin{thm}
Suppose $\Q$ is a zigzag consistent dimer in a torus. The BV operator $\overline{\Delta}$ is trivial on $HH^{even}_c \big( \M \big)$ and satisfies the following formulas on $HH^{odd}_c \big( \M \big)$:
\begin{itemize}
\item $\overline{\Delta}(X_i^n \Theta_v) = n X_i^{n-1} \Psi_{i, j}$ for all $1 \leq i \leq N$, $1 < j \leq m_i$, $v \in \E_{i, j}$, and $n \geq 0$;  
\item $\overline{\Delta}(X_i^n U) = n (\degree_{\Perf_r}(X_i) - \degree_{\Perf_t}(X_i) \big) X_i^n$ for all $1 \leq i \leq N$ and $n \geq 0$;
\item $\overline{\Delta}(X_i^n V) = n (\degree_{\Perf_s}(X_i) - \degree_{\Perf_t}(X_i) \big) X_i^n$ for all $1 \leq i \leq N$ and $n \geq 0$.
\end{itemize}
\end{thm} 

We note that the zigzag cycles, of which there are
\[
K : = \sum_{i = 1}^N m_i,
\]
are in bijection with the even generators 
\[
\{ X_i, \Psi_{i, j} \, | \, 1 \leq i \leq N, \, 1 < j \leq m_i \} 
\]
as well as with the odd generators
\[
\{ X_i (U + V), X_i \Theta_v \, | \, 1 \leq i \leq N, \, 1 < j \leq m_i, \, v \in \E_{i, j} \}.
\]
Dimer duality exchanges zigzag cycles and vertices but preserves the number of edges and the number of faces (\cite{Bocklandt2016} \S 8). Therefore, the Euler characteristic of the dual to a toric dimer is
\[
2 - 2 g^{\vee} = K - \# \Q_0.
\]
The odd generators $U$, $V$, and $\Theta_v$ span a vector space of dimension 
\[
\# \Q_0 + 1 = 2 g^{\vee} + K - 1,
\]
which is precisely the rank of the first cohomology of the punctured surface $\Sigma^\vee \setminus \Q_0^\vee$. Hence, we may recast the compactly supported cohomology of $\M$ additively as
\[
HH^*_c \big( \M \big) \cong H^*(\Sigma^\vee \setminus \Q_0^\vee, \C) \, \oplus \, \bigoplus_{\substack{1 \leq i \leq K \\ 1 \leq n}} Z_i^n \, \oplus \, \bigoplus_{\substack{1 \leq i \leq K \\ 1 \leq n}}  Z_i^n \xi_i,
\]
where the $Z_i$ are even variables and the $\xi_i$ are odd variables. The BV differential satisfies
\[
\overline{\Delta}( Z_i^n \xi_i ) = n Z_i^n
\] 
and is trivial on all other components. This answer agrees with the symplectic cohomology of the punctured surface \cite{Lekili}, giving strong evidence that the two kinds of Hochschild cohomology of the matrix factorization category coincide. 

However, the multiplicative structures of the compactly supported and symplectic cohomologies are generally not the same. This is expected from the observation that, on the A-side, punctures are topologically independent, but on the B-side, zigzag cycles are grouped together according to first homology class. There is a filtration on $SH^*(\Sigma^\vee \setminus \Q_0^\vee)$ reflecting the degeneracy among zigzags in $HH^*_c(\M)$, and the associated graded multiplicative structure coincides with that of the compactly supported cohomology.  

\subsection{Example: mirror to the four punctured sphere}

Recall the zigzag consistent dimer in a torus in Figure \ref{4 punctured sphere}. There are four zigzag cycles, no two of which are parallel. Fixing $w = v_o$ as the reference vertex, we see the compactly supported cohomology is 
\begin{align*}
HH^*_c \big( \M \big) & \cong \frac{\C[X_i, U, V, \Theta_v \, | \, 1 \leq i \leq 4]}{(\mathcal{R})} \\
& \cong H^*(\Sigma^{\vee} \setminus \Q_0^\vee, \C) \, \oplus \, \bigoplus_{\substack{1 \leq i \leq 4 \\ 1 \leq n}} \C X_i^n \, \oplus \, \bigoplus_{\substack{1 \leq i \leq 4 \\ 1 \leq n}} \C X_i^n(U + V).
\end{align*}
In this case, $HH^*_c \big( \M \big)$ is isomorphic to the symplectic cohomology of the four punctured sphere as BV algebras.

\subsection{Example: mirror to the five punctured sphere}

Recall from Example \ref{pinchpoint} the suspended pinchpoint (reproduced below), which is mirror to the five punctured sphere. The two zigzag cycles represented by the paths $ag$ and $ce$ are parallel, of homology $-\eta_4$. Fixing $v_1$ as the reference vertex, the compactly supported cohomology is
\begin{align*}
HH^*_c \big( \M \big) & \cong \frac{\C[X_i, \Psi_{4, 2}, U, V, \Theta_{v_k} \, | \, 1 \leq i \leq 4, \, k = 2, 3]}{(\mathcal{R})} \\
& \cong \, H^*(\Sigma^\vee \setminus \Q_0^\vee, \C) \, \oplus \, \bigoplus_{\substack{1 \leq i \leq 4 \\ 1 \leq n}} \C X_i^n \, \oplus \, \bigoplus_{0 \leq n} \C X_4^n \Psi_{4, 2}  \, \oplus \, \bigoplus_{\substack{1 \leq i \leq 4 \\ 1 \leq n}} \C X_i^n(U + V) \\
& \oplus \, \bigoplus_{1 \leq n} \C X_4^n \Theta_{v_2}.
\end{align*}
Notice that $X_4 \Theta_{v_2} = X_4 \Theta_{v_3}$, since $v_2$ and $v_3$ are $x_{\eta_4}$-connected. The five even generators
\[
X_1, X_2, X_3, X_4, \Psi_{4, 2}
\]
and the five odd generators
\[
X_1(U+V), X_2(U+V), X_3(U+V), X_4(U+V), X_4 \Theta_{v_2}
\]
correspond to the five zigzag cycles, or the five punctures of the mirror dual. As a $\z / 2\z$-graded vector space equipped with the BV operator, $HH^*_c \big( \M \big)$ agrees with the symplectic cohomology of the five punctured sphere. However, the multiplicative structures are not isomorphic, owing to the existence of parallel zigzag cycles.


\begin{figure}[h] \label{5 punctured sphere}
\centering
\begin{subfigure}{.4\textwidth}
\centering
\[
\begin{tikzpicture}
\draw (0,0) circle [radius=0.2];
\node at (0, 0) {$v_1$};
\draw (0,2) circle [radius=0.2];
\node at (0, 2) {$v_2$};
\draw (0,4) circle [radius=0.2];
\node at (0, 4) {$v_3$};
\draw (0,6) circle [radius=0.2];
\node at (0, 6) {$v_1$};
\draw (2,0) circle [radius=0.2];
\node at (2, 0) {$v_1$};
\draw (2,2) circle [radius=0.2];
\node at (2, 2) {$v_2$};
\draw (2,4) circle [radius=0.2];
\node at (2, 4) {$v_3$};
\draw (2,6) circle [radius=0.2];
\node at (2, 6) {$v_1$};
\node at (1, -1) {$\Q$};
\draw[->] (.2, 0) -- (1.8, 0);
\node[below] at (1,0) {$d$};
\draw[->] (0, 0.2) -- (0, 1.8);
\node[left] at (0, 1) {$a$};
\draw[<-] (0, 2.2) -- (0, 3.8);
\node[left] at (0, 3) {$b$};
\draw[->] (0, 4.2) -- (0, 5.8);
\node[left] at (0, 5) {$c$};
\draw[->] (0.2, 6) -- (1.8, 6);
\node[above] at (1, 6) {$d$};
\draw[->] (2, 0.2) -- (2, 1.8);
\node[right] at (2, 1) {$a$};
\draw[<-] (2, 2.2) -- (2, 3.8);
\node[right] at (2, 3) {$b$};
\draw[->] (2, 4.2) -- (2, 5.8);
\node[right] at (2, 5) {$c$};
\draw[->] (1.86, 5.86) -- (0.14, 4.14);
\node[above left] at (1, 5) {$e$};
\draw[->] (0.14, 2.14) -- (1.86, 3.86);
\node[above left] at (1, 3) {$f$};
\draw[->] (1.86, 1.86) -- (0.14, 0.14);
\node[above left] at (1, 1) {$g$};
\end{tikzpicture}
\]
\caption{$\Q$}
\end{subfigure}
\begin{subfigure}{0.4\textwidth}
\centering
\[
\begin{tikzpicture}
\draw (0,2) -- (0,6);
\draw (0,6) -- (2,2);
\draw (2,2) -- (2, 0);
\draw (2,0) -- (0,2);
\draw[->] (0, 3) -- (-1, 3);
\node[left] at (-1, 3) {$ce$};
\draw[->] (0, 5) -- (-1, 5);
\node[above] at (-0.5, 3) {$-\eta_4$};
\node[left] at (-1, 5) {$ag$};
\node[above] at (-0.5, 5) {$-\eta_4$};
\draw[->] (1, 4) -- (3, 5);
\node[right] at (3, 5) {$dafc$};
\node[above left] at (2, 4.5) {$-\eta_3$};
\draw[->] (2, 1) -- (3, 1);
\node[right] at (3, 1) {$fb$};
\node[above] at (2.5, 1) {$-\eta_2$};
\draw[->] (1, 1) -- (0, 0);
\node[left] at (0, 0) {$debg$};
\node[above left] at (0.5, 0.5) {$-\eta_1$};
\filldraw[black] (0,2) circle (2pt) node[anchor=east] {$\{e, g\} = \Perf_4$ }; 
\filldraw[black] (0,4) circle (2pt) node[anchor=east] {$\{a, e\}, \{c, g\}$};
\filldraw[black] (0,6) circle (2pt) node[anchor=east] {$ \{ a, c\} = \Perf_3 $};
\filldraw[black] (2,2) circle (2pt) node[anchor=west] {$ \Perf_2 = \{d, f\} $};
\filldraw[black] (2, 0) circle (2pt) node[anchor= west] {$\Perf_1 = \{b, d\}$};
\node at (1, -1) {$MP(\Q)$};
\end{tikzpicture}
\]
\caption{$MP(\Q)$}
\end{subfigure}
\begin{subfigure}{0.4\textwidth}
\centering
\[
\begin{tikzpicture}
\draw[->, thick] (0, 0) -- (1, 1);
\node[above right] at (1, 1) {$\langle \eta_1 \rangle$};
\node at (0.9, 0.25) {$\sigma_4$};
\draw[->, thick] (0, 0) -- (-1.5, 0);
\node[left] at (-1.5, 0) {$\langle \eta_2 \rangle$};
\node at (-0.9, 0.25) {$\sigma_1$};
\draw[->, thick] (0,0) -- (-2, -1);
\node[below left] at (-2, -1) {$\langle \eta_3 \rangle$};
\node at (-0.9, -0.25) {$\sigma_2$};
\draw[->, thick] (0, 0) -- (1.5, 0);
\node[right] at (1.5, 0) {$\langle \eta_4 \rangle$};
\node at (0.9, -0.25) {$\sigma_3$};
\end{tikzpicture}
\]
\caption{The antizigzag fan}
\end{subfigure}
\caption{Mirror dual to the five punctured sphere}
\end{figure}
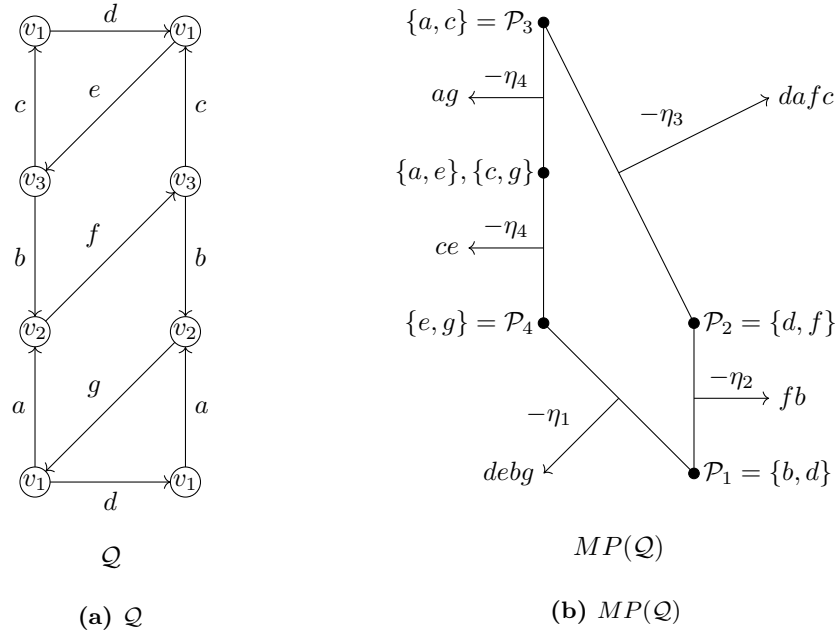


\bibliographystyle{plain}
\bibliography{Dimer_models_and_Hochschild_cohomology}

\begin{thebibliography}{10}

\bibitem{Abouzaid}
Mohammed Abouzaid, Denis Auroux, Alexander~I. Efimov, Ludmil Katzarkov, and
  Dmitri Orlov.
\newblock Homological mirror symmetry for punctured spheres.
\newblock {\em J. Amer. Math. Soc.}, 26(4):1051--1083, 2013.

\bibitem{AD}
Andr\'es Angel and Diego Duarte.
\newblock The {BV}-algebra structure of the {H}ochschild cohomology of the
  group ring of cyclic groups of prime order.
\newblock In {\em Geometric, algebraic and topological methods for quantum
  field theory}, pages 353--372. World Sci. Publ., Hackensack, NJ, 2017.

\bibitem{Armenta}
Marco Armenta and Bernhard Keller.
\newblock Derived invariance of the cap product in {H}ochschild theory.
\newblock arXiv:1711.02947.

\bibitem{Ballard}
Matthew Ballard, David Favero, and Ludmil Katzarkov.
\newblock A category of kernels for equivariant factorizations and its
  implications for {H}odge theory.
\newblock {\em Publ. Math. Inst. Hautes \'{E}tudes Sci.}, 120:1--111, 2014.

\bibitem{TS}
Gwyn Bellamy, Daniel Rogalski, Travis Schedler, J.~Toby Stafford, and Michael
  Wemyss.
\newblock {\em Noncommutative algebraic geometry}, volume~64 of {\em
  Mathematical Sciences Research Institute Publications}.
\newblock Cambridge University Press, New York, 2016.
\newblock Lecture notes based on courses given at the Summer Graduate School at
  the Mathematical Sciences Research Institute (MSRI) held in Berkeley, CA,
  June 2012.

\bibitem{Bocklandt2012}
Raf Bocklandt.
\newblock Consistency conditions for dimer models.
\newblock {\em Glasg. Math. J.}, 54(2):429--447, 2012.

\bibitem{Bocklandt2007}
Raf Bocklandt.
\newblock Calabi-{Y}au algebras and weighted quiver polyhedra.
\newblock {\em Math. Z.}, 273(1-2):311--329, 2013.

\bibitem{abc}
Raf Bocklandt.
\newblock A dimer {ABC}.
\newblock {\em Bull. Lond. Math. Soc.}, 48(3):387--451, 2016.

\bibitem{Bocklandt2016}
Raf Bocklandt.
\newblock Noncommutative mirror symmetry for punctured surfaces.
\newblock {\em Trans. Amer. Math. Soc.}, 368(1):429--469, 2016.
\newblock With an appendix by Mohammed Abouzaid.

\bibitem{Broomhead}
Nathan Broomhead.
\newblock Dimer models and {C}alabi-{Y}au algebras.
\newblock {\em Mem. Amer. Math. Soc.}, 215(1011):viii+86, 2012.

\bibitem{Brylinski}
Jean-Luc Brylinski.
\newblock Central localization in {H}ochschild homology.
\newblock {\em J. Pure Appl. Algebra}, 57(1):1--4, 1989.

\bibitem{ChasSullivan}
Moira Chas and Dennis Sullivan.
\newblock String topology.
\newblock arXiv:math/9911159.

\bibitem{CaldararuTu}
Andrei C\u{a}ld\u{a}raru and Junwu Tu.
\newblock Curved {$A_\infty$} algebras and {L}andau-{G}inzburg models.
\newblock {\em New York J. Math.}, 19:305--342, 2013.

\bibitem{Davison}
Ben Davison.
\newblock Consistency conditions for brane tilings.
\newblock {\em J. Algebra}, 338:1--23, 2011.

\bibitem{VDBV}
Louis de~Thanhoffer~de Volcesy and Michel Van~den Bergh.
\newblock Calabi-{Y}au deformations and negative cyclic homology.
\newblock arXiv:1201.1520. To appear in Journal of Noncommutative Geometry.

\bibitem{Dyckerhoff2011}
Tobias Dyckerhoff.
\newblock Compact generators in categories of matrix factorizations.
\newblock {\em Duke Math. J.}, 159(2):223--274, 2011.

\bibitem{EtingofGinzburg07}
Pavel Etingof and Victor Ginzburg.
\newblock Noncommutative complete intersections and matrix integrals.
\newblock {\em Pure Appl. Math. Q.}, 3(1, Special Issue: In honor of Robert D.
  MacPherson. Part 3):107--151, 2007.

\bibitem{EGP}
Pavel Etingof and Victor Ginzburg.
\newblock Noncommutative del {P}ezzo surfaces and {C}alabi-{Y}au algebras.
\newblock {\em J. Eur. Math. Soc. (JEMS)}, 12(6):1371--1416, 2010.

\bibitem{Farinati}
Marco Farinati.
\newblock Hochschild duality, localization, and smash products.
\newblock {\em J. Algebra}, 284(1):415--434, 2005.

\bibitem{Ganatra}
Sheel Ganatra.
\newblock {\em Symplectic cohomology and duality for the wrapped {F}ukaya
  category}.
\newblock PhD thesis, Massachusetts Institute of Technology, 2012.

\bibitem{Gers}
Murray Gerstenhaber.
\newblock The cohomology structure of an associative ring.
\newblock {\em Ann. of Math. (2)}, 78:267--288, 1963.

\bibitem{Ginzburg}
Victor Ginzburg.
\newblock Calabi-{Y}au algebras.
\newblock arXiv:math/0612139.

\bibitem{Gulotta}
Daniel~R. Gulotta.
\newblock Properly ordered dimers, {$R$}-charges, and an efficient inverse
  algorithm.
\newblock {\em J. High Energy Phys.}, (10):014, 31, 2008.

\bibitem{HKR}
G.~Hochschild, Bertram Kostant, and Alex Rosenberg.
\newblock Differential forms on regular affine algebras.
\newblock {\em Trans. Amer. Math. Soc.}, 102:383--408, 1962.

\bibitem{Lam}
T.~Y. Lam.
\newblock {\em Lectures on modules and rings}, volume 189 of {\em Graduate
  Texts in Mathematics}.
\newblock Springer-Verlag, New York, 1999.

\bibitem{Lekili}
Yanki Lekili and James Pascaleff.
\newblock Floer cohomology of {$\mathfrak{g}$}-equivariant {L}agrangian branes.
\newblock {\em Compos. Math.}, 152(5):1071--1110, 2016.

\bibitem{LinPomerleano}
Kevin~H. Lin and Daniel Pomerleano.
\newblock Global matrix factorizations.
\newblock {\em Math. Res. Lett.}, 20(1):91--106, 2013.

\bibitem{Loday}
Jean-Louis Loday.
\newblock {\em Cyclic homology}, volume 301 of {\em Grundlehren der
  Mathematischen Wissenschaften [Fundamental Principles of Mathematical
  Sciences]}.
\newblock Springer-Verlag, Berlin, second edition, 1998.
\newblock Appendix E by Mar\'\i a O. Ronco, Chapter 13 by the author in
  collaboration with Teimuraz Pirashvili.

\bibitem{PolPos}
Alexander Polishchuk and Leonid Positselski.
\newblock Hochschild (co)homology of the second kind {I}.
\newblock {\em Trans. Amer. Math. Soc.}, 364(10):5311--5368, 2012.

\bibitem{Positselski}
Leonid Positselski.
\newblock Two kinds of derived categories, {K}oszul duality, and
  comodule-contramodule correspondence.
\newblock {\em Mem. Amer. Math. Soc.}, 212(996):vi+133, 2011.

\bibitem{TTnc}
Dmitri Tamarkin and Boris Tsygan.
\newblock The ring of differential operators on forms in noncommutative
  calculus.
\newblock In {\em Graphs and patterns in mathematics and theoretical physics},
  volume~73 of {\em Proc. Sympos. Pure Math.}, pages 105--131. Amer. Math.
  Soc., Providence, RI, 2005.

\bibitem{Toen}
Bertrand Toen.
\newblock The homotopy theory of {$dg$}-categories and derived {M}orita theory.
\newblock {\em Invent. Math.}, 167(3):615--667, 2007.

\bibitem{Vaintrob}
Dmitry Vaintrob.
\newblock The string topology {B}{V} algebra, {H}ochschild cohomology and the
  {G}oldman bracket on surfaces.
\newblock arXiv:math/0702859.

\bibitem{VdB}
Michel van~den Bergh.
\newblock A relation between {H}ochschild homology and cohomology for
  {G}orenstein rings.
\newblock {\em Proc. Amer. Math. Soc.}, 126(5):1345--1348, 1998.

\bibitem{Weibel}
Charles~A. Weibel.
\newblock {\em An introduction to homological algebra}, volume~38 of {\em
  Cambridge Studies in Advanced Mathematics}.
\newblock Cambridge University Press, Cambridge, 1994.

\bibitem{Zeng}
Jieheng Zeng.
\newblock Derived categories and {C}alabi-{Y}au algebras.
\newblock arXiv:1711.08574.

\end{thebibliography}

\end{document}